\documentclass[a4paper,11pt]{article}

\usepackage[utf8]{inputenc}
\usepackage[T1]{fontenc}
\usepackage[english]{babel}

\usepackage[tbtags]{amsmath}
\usepackage{amssymb}
\usepackage{amsfonts}
\usepackage{amsthm}
\usepackage{calrsfs}
\usepackage{array}
 \usepackage{bbm}
\usepackage{stmaryrd}
\usepackage{upgreek}

\usepackage{mathtools}
\usepackage[svgnames]{xcolor}
\usepackage{hyperref}

\numberwithin{equation}{section}

\def\cprime{$'$} 

\newcommand{\bb}{{b}}
\newcommand{\cdeux}{{C_2}}
\newcommand\D{{\mathcal D}}
\newcommand{\ddeux}{{\delta_2}}
\newcommand{\ds}{\displaystyle}

\newcommand\dv{\mathop{\rm div}}
\newcommand\dun{\delta_1}
\newcommand{\eb}{{\epsilon_0}}
\newcommand{\edeux}{{\epsilon_2}}
\newcommand{\ee}{{\bar\epsilon}}
\newcommand{\equatre}{{\epsilon_{20}}}
\newcommand{\etrois}{{\epsilon_3}}
\newcommand{\eun}{{\epsilon_1}}
\newcommand{\evingt}{{\epsilon_3}}
\newcommand\fe{\varphi_{d, \epsilon}}
\renewcommand{\H}{{\mathcal H}}
\newcommand{\hz}{{\bar c_0}}
\newcommand{\hu}{{\bar c_1}}
\newcommand{\hd}{{\bar c_2}}
\renewcommand\iint{\displaystyle\int_{|y|<1}}
\newcommand{\m}[1]{\mathbbm{#1}}
\newcommand{\pnu}{\partial_\nu}
\newcommand{\q}[1]{\mathcal{#1}}
\newcommand\RR{{\mathcal R}}
\newcommand\spa{\mathop{\rm span}}
\renewcommand\SS{{\mathcal S}}
\newcommand{\vc}[2]{\begin{pmatrix} #1\\#2\end{pmatrix}}

\theoremstyle{plain}
\newtheorem{thm}{Theorem}
\newtheorem*{thm*}{Theorem}

\newtheorem{prop}{Proposition}[section]
\newtheorem{cl}[prop]{Claim}
\newtheorem{cor}[prop]{Corollary}
\newtheorem{lem}[prop]{Lemma}
\newtheorem{defi}[prop]{Definition}

\theoremstyle{definition}

\theoremstyle{remark}
\newtheorem*{nb}{Remark}

\makeatletter
\def\blfootnote{\xdef\@thefnmark{}\@footnotetext}
\makeatother

\title{\bf Dynamics near explicit stationary solutions in similarity variables for solutions of a semilinear wave equation in higher dimensions
\footnote{Both authors are supported by the ERC Advanced Grant no. 291214, BLOWDISOL. H.Z. is also supported by ANR Project ANA\'E ref. ANR-13-BS01-0010-03.}
}
\author{Frank Merle\\
{\it \small Universit\'e de Cergy Pontoise and IHES}\\
Hatem Zaag\\
{\it \small Université Paris 13, Sorbonne Paris Cité,}\\
{\it \small LAGA, CNRS (UMR 7539), F-93430, Villetaneuse, France.}
}

\allowdisplaybreaks

\begin{document}

\maketitle

\begin{abstract}
This is the first of two papers devoted to the study of the properties of the blow-up
surface for the $N$ dimensional semilinear wave equation with subconformal power nonlinearity. In a series of papers, we have clarified the situation in one space dimension. Our goal here is to extend some of the properties to higher dimension. In dimension one, an essential tool was to study the dynamics of the solution in similarity variables, near the set of non-zero equilibria, which are obtained by a Lorentz transform of the space-independent solution. As a matter of fact, the main part of this paper is to study similar objects in higher dimensions. More precisely, near that set of equilibria, we show that solutions are either non-global, or go to zero, or converge to some explicit equilibrium.
%
%
%
We also show that the first case cannot occur in the characteristic case, and that only the third possibility occurs in the non-characteristic case, thanks to the non-degeneracy of the blow-up limit, another new result in our paper. 
As a by-product of our techniques, we obtain the stability of the zero solution.
\end{abstract}

\medskip

 {\bf MSC 2010 Classification}:  35L05, 35L71, 35L67, 35B44, 35B40

%
\medskip
{\bf Keywords}: Semilinear wave equation, blow-up, higher-dimensional case.


\section{Introduction}
We consider the following semilinear wave equation:
\begin{equation}\label{equ}
\left\{
\begin{array}{l}
\partial_t^2 u =\Delta u+|u|^{p-1}u,\\
u(0)=u_0\mbox{ and }\partial_t u(0)=u_1,
\end{array}
\right.
\end{equation}
where $u(t):x\in{\m R}^N \rightarrow u(x,t)\in{\m R}$, $u_0\in \rm H^1_{\rm loc,u}$
and $u_1\in \rm L^2_{\rm loc,u}$.\\
 The space $\rm L^2_{\rm loc,u}$ is the set of all $v$ in $\rm L^2_{\rm loc}$ such that 
\[
\|v\|_{\rm L^2_{\rm loc,u}}\equiv\ds\sup_{a\in {\m R}^N}\left(\int_{|x-a|<1}|v(x)|^2dx\right)^{1/2}<+\infty,
\]
 and the space ${\rm H}^1_{\rm loc,u}= \{ v\;|\;v, \nabla v \in {\rm L}^2_{\rm loc,u}\}$.\\
 We assume in addition that 
\begin{equation}\label{condp}
1<p\mbox{ and }p< 1+\frac 4{N-1}\mbox{ if } N\ge 2.
\end{equation}
We solve equation \eqref{equ} locally in time in ${\rm H}^1_{\rm loc,u}\times {\rm L}^2_{\rm loc,u}({\m R}^N)$ (see Ginibre, Soffer and Velo \cite{GSVjfa92}, Lindblad and Sogge \cite{LSjfa95}).
Existence of blow-up solutions follows from ODE techniques or the energy-based blow-up criterion of Levine \cite{Ltams74}. More blow-up results can be found in
Caffarelli and Friedman \cite{CFtams86}, \cite{CFarma85},
Alinhac \cite{Apndeta95}, \cite{Afle02}, Kichenassamy and Littman \cite{KL1cpde93}, \cite{KL2cpde93}, Killip and Vi\c san \cite{KV11}.

\bigskip

The one-dimensional case has been understood completely, in a series of papers by the authors \cite{MZjfa07}, \cite{MZcmp08}, \cite{MZajm11} and \cite{MZisol10} and in C\^ote and Zaag \cite{CZmulti11} (see also the note \cite{MZxedp10}).\\
In higher dimensions, only the blow-up rate is known (see \cite{MZajm03}, \cite{MZma05} and \cite{MZimrn05}; see also the extension by Hamza and Zaag in \cite{HZnonl12} and \cite{HZjhde12} to the Klein-Gordon equation and other damped lower-order perturbations of equation \eqref{equ}). In fact, the program of dimension $1$ breaks down because there is no classification of selfsimilar solutions of equation \eqref{equ} in the energy space, except in the radial case outside the origin (see \cite{MZbsm11}). Considering the behavior of radial solutions at the origin, Donninger and Sch{\"o}rkhuber were able to prove the stability of the space-independent solution (i.e. the solution of the associated ODE $u"=u^p$) with respect to perturbations in initial data, in the Sobolev subcritical range \cite{DSdpde12} and also in the supercritical range in \cite{DStams14}. Some numerical results are available in a series of papers by Bizo\'n and co-authors (see \cite{Bjnmp01}, \cite{BCTnonl04}, \cite{BZnonl09}). Let us also mention that in the superconformal and Sobolev subcritical range, an upper bound on the blow-up rate is available  was proved by Killip, Stoval and Vi\c san in \cite{KSVma14}, then refined by Hamza and Zaag in \cite{HZdcds13}. \\
In a companion paper \cite{MZods}, we address the blow-up behavior for general solutions to equation \eqref{equ}, and prove the openness of the set of non-characteristic points 
with some explicit profile 
(that set is defined below in \eqref{defr0}).
A key step is to understand the dynamics of the similarity variables' version \eqref{eqw}, and to prove a trapping result near the set \eqref{defkd} made of Lorentz transform \eqref{defW} of the constant solution.
This is the aim of this paper.

\bigskip

Let us first introduce some notations before stating our result.

\medskip

In this paper, we consider $u(x,t)$ a blow-up solution to equation \eqref{equ}.
Following Alinhac \cite{Apndeta95}, we define 
a 1-Lipschitz surface $\Gamma=\{(x,T(x))\}$ 
such that the maximal influence domain $D$ of $u$ (or the domain of definition of $u$) is written as 
\begin{equation*}
D=\{(x,t)\;|\; t< T(x)\}.
\end{equation*}
The surface $\Gamma$ is called the blow-up graph of $u$. 
A point $x_0\in{\m R}^N$ is a non-characteristic point if there are 
\begin{equation*}
\delta_0\in(0,1)\mbox{ and }t_0<T(x_0)\mbox{ such that }
u\;\;\mbox{is defined on }{\mathcal C}_{x_0, T(x_0), \delta_0}\cap \{t\ge t_0\}
\end{equation*}
where 
\begin{equation*}
{\mathcal C}_{\bar x, \bar t, \bar \delta}=\{(x,t)\;|\; t< \bar t-\bar \delta|x-\bar x|\}.
\end{equation*}
 We denote by $\RR\subset {\m R}^N$ (resp. $\SS\subset {\m R}^N$) the set of non-characteristic (resp. characteristic) points.\\
A convenient tool to study blow-up for equation \eqref{equ} is to introduce the following similarity variables, for any $(x_0,T_0)$ such that $0< T_0\le T(x_0)$:
\begin{equation}\label{defw}
w_{x_0,T_0}(y,s)=(T_0-t)^{\frac 2{p-1}}u(x,t),\;\;y=\frac{x-x_0}{T_0-t},\;\;
s=-\log(T_0-t).
\end{equation}
If $T_0=T(x_0)$, we write $w_{x_0}$ for short. 
 The function $w_{x_0,T_0}$ (we write $w$ for simplicity) satisfies the 
following equation for all $|y|<1$ and $s\ge -\log T_0$:
\begin{equation}\label{eqw}
\partial^2_sw-\q L w+\frac{2(p+1)}{(p-1)^2}w-|w|^{p-1}w=
-\frac{p+3}{p-1}\partial_sw-2y.\nabla \partial_s w
\end{equation} 
where
\begin{equation}\label{defro}
\q L w =\frac 1\rho \dv \left(\rho \nabla w-\rho(y.\nabla w)y\right),\;\; \rho(y)=(1-|y|^2)^\alpha\mbox{ and }\alpha=\frac 2{p-1}-\frac{N-1}2>0.
\end{equation}
Equation \eqref{eqw} will be studied in the energy space
\begin{equation}\label{defnh}
\H = \left\{(r_1,r_2) 
\;\;|\;\;\|(r_1,r_2)\|_{\H}^2\equiv \iint \left(r_1^2+|\nabla r_1|^2-(y\cdot \nabla r_1)^2)+r_2^2\right)\rho dy<+\infty\right\}.
\end{equation}
For any $|d|<1$, we introduce the following stationary solutions of \eqref{eqw} (or solitons) defined by 
\begin{equation}\label{defkd}
\kappa(d,y)=\kappa_0 \frac{(1-|d|^2)^{\frac 1{p-1}}}{(1+d.y)^{\frac 2{p-1}}}\mbox{ where }\kappa_0 = \left(\frac{2(p+1)}{(p-1)^2}\right)^{\frac 1{p-1}} \mbox{ and }|y|<1.
\end{equation}

\bigskip

A key step in 
the one-dimensional case was to determine the set of all stationary solutions of equation \eqref{eqw} in the energy space.
As a matter of fact, we proved in Proposition 1 page 46 in \cite{MZjfa07} that such a set is given by
\begin{equation*}
S_0\equiv\{0;\pm \kappa(d,y)\:|\;|d|<1\}.
\end{equation*}
Unfortunately, in higher dimensions, we have been unable to determine that set, though
one can trivially see that $0$ and $\pm \kappa(d,y)$ for all $|d| < 1$ are still stationary solution of \eqref{eqw}. Note that when $N = 3$ and $p = 3$ or $p\ge 7$ is an odd integer, Bizon, Breitenlohner,
Maison and Wasserman proved in \cite{BMWnonl07} and \cite{BBMWnonl10} the existence of a countable family of radially-symmetric stationary solutions.
In particular, from this result, it is clearly untrue to say that the blow-up profile near a non-characteristic point is always given by $\pm \kappa(d,\cdot)$ as in one space dimension.
The goal of our two papers, the present one and \cite{MZods}, is to prove the following surprising result about the stability of the profile $\pm \kappa(d, \cdot)$, and thus, as a corollary, we will obtain the differentiability of the blow-up surface $x\mapsto T(x)$. More precisely, 
we will show in \cite{MZods} that
\begin{equation}\label{defr0}
\RR_0 = \{x_0\in \RR\;|\;\exists\; |d(x_0)|<1
\mbox{ s.t. }w_{x_0}(s)\to \pm \kappa(d(x_0))\mbox{ as }s\to \infty
\}
\end{equation}
is an open set. 
A key step is to understand the dynamics of equation \eqref{eqw}, particularly, to prove the following result for initial data near the set $\{\pm\kappa(d,y)\;|\;|d|<1\}$:
\begin{thm}[Behavior of solutions of equation \eqref{eqw} near $\pm \kappa(d,y)$ and trapping in the non-characteristic case]\label{proptrap} 
There exist $\eb>0$, $K_0>0$ and $\mu_0>0$ such that 
for any solution $w$ of equation \eqref{eqw} continuous in time with values in $\q H$,
$\bar\omega=\pm 1$ and $|\bar d|<1$, if
\begin{equation*}
\ee\equiv\left \|\vc{w(0)}{\partial_s w(0)}-\bar\omega\vc{\kappa(\bar d)}{0}\right\|_{\H}\le \eb,
\end{equation*}
then:\\
- either $w(s)$ is not defined for all $(y,s) \in B(0,1)\times [0,+\infty)$,\\
- or 
\begin{equation}\label{impossible}
\|(w(s), \partial_s w(s))\|_{\q H}\to 0\mbox{ as }s\to \infty\mbox{ exponentially fast};
\end{equation}
- or there exists $|d_\infty(\bar x)|<1$ such that
\begin{equation}\label{conv0}
\forall s\ge 0,\;\;\|(w(s), \partial_s w(s))- (\kappa(d_\infty),0)\|_{\q H}\le K_0\ee e^{-\mu_0 s}
\end{equation}
and 
\begin{equation}\label{proxi0}
|\arg\tanh |\bar d|-\arg\tanh |d_\infty||+\frac{|\bar d-d_\infty|}{\sqrt{1-|\bar d|}}\le K_0\ee.
\end{equation}
Moreover, if $w(y,s)=w_{\bar x}(y,+\sigma)$ for some $\bar x\in \m R$, $\sigma \ge- \log T(\bar x)$ where $w_{\bar x}$ is defined in \eqref{defw} from a blow-up solution $u$, then the first possibility does not occur. If in addition, $\bar x\in \RR$, then only \eqref{conv0} holds.  
\end{thm}
\begin{nb} In the first possibility, we suspect that $w$ blows up in finite time.
This result shows that $\pm \kappa(d,y)$ are threshold solutions. 
When $w=w_{\bar x}$, whether $\bar x$ is characteristic or not, we showed in \cite{MZajm11} that case \eqref{impossible} never occurs in one space dimension. We proved the same in \cite{MZbsm11} in higher dimensions under radial symmetry when $\bar x\neq 0$. 
\end{nb}
\begin{nb}
From the classical elliptic theory, we can show that the set $\{\pm \kappa(d,\cdot)\}$ is isolated in the set of finite-energy stationary solutions of equation \eqref{eqw}.
\end{nb}

The proof in higher dimensions is far from being a simple adaptation of the one-dimensional case. Indeed, several difficulties arise when $N\ge 2$, among them:\\
- we have $N-1$ new degenerate directions in the linearized operator of equation \eqref{eqw} around the stationary solution $\kappa(d)$ \eqref{defkd}; these directions come from the derivatives of $\kappa(d)$ with respect to the $N-1$ angular directions of $d$ (see Lemma \ref{l10} below); the projections of the solution on those directions have different features and are delicate to study as $|d|\to 1$. In particular, obtaining estimates with optimal bounds in $d$ require a new treatment;\\
- the modulation argument around the family $\kappa^*(d,\nu)$ defined below in \eqref{defk*} becomes particularly intricate when $d\to 0$ or $|d|\to 1$
(see Lemma \ref{lemmod} below and Section \ref{appmod} in the Appendix).

\bigskip

We proceed in two sections:\\
- In Section \ref{secdyn}, we study the linearized operator of equation \eqref{eqw} around a stationary solution;\\
- Section \ref{secrig} is devoted to the proof of Theorem \ref{proptrap}. 
\section{A dynamical system formulation for equation \eqref{eqw}}\label{secdyn}
This section is devoted to the study of the linearization of equation \eqref{eqw} around the stationary solution $\kappa(d,y)$ \eqref{defkd}. We also present a modulation technique adapted to the properties of the linear operator. We heavily use the dynamical system formulation we introduced in \cite{MZisol10} for the study of equation \eqref{eqw} around a sum of decoupled solitons. Here, the situation is easier, since we linearize around a single soliton. Therefore, we will summarize the key arguments and refer to \cite{MZisol10} for details.

 \bigskip

This section is divided into 
two
parts:\\
- we first study the linearized operator of equation \eqref{eqw} around $\kappa(d)$ where $d$ is fixed;\\
- then, we give a decomposition of solutions of \eqref{eqw} well-adapted to the spectral properties of its linearized operator around $\kappa(d)$;

\subsection{The linearized operator of equation \eqref{eqw} around $\kappa(d)$}
In this step, we aim at understanding the linearized operator of equation \eqref{eqw} around the stationary solution $\kappa(d)$. Consider then $w(y,s)$ a solution to equation \eqref{eqw} in the energy space $\H$. If $|d|<1$ and $q= (q_1,q_2)=\vc{q_1}{q_2}$ is defined for all $s$ in the domain of definition of $w$ by
\begin{equation*}
\vc{w(y,s)}{\partial_s w(y,s)}=\vc{\kappa(d,y)}{0}+\vc{q_1(y,s)}{q_2(y,s)},
\end{equation*}
then we see from equation \eqref{eqw} that $q$
 satisfies the following equation:
\begin{equation}\label{eqq}
\ds\frac \partial {\partial s}
\left(
\begin{array}{l}
q_1\\
q_2
\end{array}
\right)
=L_{d}
\left(
\begin{array}{l}
q_1\\
q_2
\end{array}
\right)
+\left(
\begin{array}{l}
0\\
f_{d}(q_1)
\end{array}
\right)\\
\end{equation}
where 
\begin{equation}\label{defld}
\begin{array}{rcl}
L_d\vc{q_1}{q_2}&=&\vc{q_2}{\q L q_1+\psi(d,y)q_1-\frac{p+3}{p-1}q_2-2y.\nabla q_2},\\
\\
f_d(q_1)&=&|\kappa(d,\cdot)+q_1|^{p-1}(\kappa(d,\cdot)+q_1)-\kappa(d,\cdot)^p-p\kappa(d,\cdot)^{p-1}q_1,\\
\\
\psi(d,y)&=&p\kappa(d,y)^{p-1}-\frac{2(p+1)}{(p-1)^2}=\frac{2(p+1)}{(p-1)^2}\left(p \frac{(1-|d|^2)}{(1+d\cdot y)^2}-1\right)
\end{array}
\end{equation}
and the operator $\q L$ is defined in \eqref{defro}. 
In this step, we study the linear operator $L_d$ in the energy space $\H$ defined in \eqref{defnh}. Note from \eqref{defnh} that we have
\begin{equation*}
\|q\|_{\H}=\left[\phi\left(q,q\right)\right]^{1/2}<+\infty
\end{equation*}
where the inner product $\phi$ is defined by
\begin{equation}\label{defphi}
\phi(q,r)=
\phi\left(\vc{q_1}{q_2}, \vc{r_1}{r_2}\right)= \int_{-1}^1 \left(q_1r_1+\nabla q_1\cdot\nabla r_1- (y\cdot\nabla q_1)(y\cdot\nabla r_1)+q_2r_2\right)\rho dy.
\end{equation}
Using integration by parts and the definition of $\q L$ \eqref{defro}, we have the following identity
\begin{equation}\label{defphi2}
\phi(q,r)= \int_{-1}^1 \left(q_1\left(-\q L r_1+r_1\right) +q_2 r_2\right)\rho dy.
\end{equation}
One of the major difficulties in the proof comes from the fact that the linear operator $L_d$ is neither self-adjoint nor anti-self-adjoint. In particular, standard spectral theory does not apply. Since equation \eqref{equ} is invariant under rotation of coordinates, the same occurs for equation \eqref{eqw}. For that reason, we will often change coordinates to another orthonormal basis such that $e_1=\frac d{|d|}$ if $d\neq 0$. More precisely, we claim the following:
\begin{lem}[An orthonormal basis associated to $d$]\label{lembase} Introducing $B_{\pm}= \{d\;|\;0<|d|<1,\;\;\pm d_1 > -\frac{|d|}2\}$, there exist a $C^\infty$ orthonormal basis $(e_{1,\pm}(d),\dots, e_{N,\pm}(d))$
defined on $B_\pm$ such that $e_{1,\pm}(d)=\frac d{|d|}$, $e_{i,\pm}(d) =e_{i,\pm}(\frac d{|d|})$ and $|\nabla e_{i,\pm}(d)|\le \frac C{|d|}$, for all $i=1,\dots,N$.
\end{lem}
\begin{nb} Note that $B_+$ and $B_-$ are open sets such that $B_+\cup B_-=B(0,1)\backslash \{0\}$. Therefore:\\
- if $0<|d|<1$, we know that either $d\in B_+$ or $d\in B_-$, and for simplicity, the notation $(e_1(d),\dots,e_N(d))$ will stand for $(e_{1,+}(d),\dots, e_{N,+}(d))$ or $(e_{1,-}(d),\dots, e_{N,-}(d))$, accordingly;\\
- if $d=0$, the notation $(e_1(0),\dots,e_N(0))$ stands for the canonical basis of $\m R^N$.
\end{nb}
\begin{proof} If $S^{N-1}_\pm= B_\pm \cap S^{N-1}$, it is classical that there exists a $C^\infty$ orthonormal basis $(e_{1,\pm}(d),\dots, e_{N,\pm}(d))$ defined on $S^{N-1}_\pm$ such that $e_1(d)=d$. Extending the definition to $B_\pm$ by setting $e_{i,\pm}(d) = e_{i,\pm}(\frac d{|d|})$ yields the result.  
\end{proof}

\medskip

In the following, we 
show that $\lambda=0$ and $\lambda=1$ are eigenvalues of $L_d$ and find explicit formulas for the eigenfunctions of $L_d$:
\begin{lem}[Nonnegative eigenvalues and corresponding eigenfunctions for $L_d$]\label{l10}
$ $\\
(i) For all $|d|<1$, $\lambda=1$ and $\lambda=0$ are eigenvalues of the linear operator $L_d$ with multiplicity $1$ and $N$ respectively, and the corresponding eigenfunctions are:\\
- for $\lambda=1$,\;\;
\begin{equation}\label{deffld}
F_0(d,y)=(1-|d|^2)^{\frac p{p-1}}\vc{(1+d\cdot y)^{-\frac {p+1}{p-1}}}{(1+d\cdot y)^{-\frac {p+1}{p-1}}};
\end{equation}
- for $\lambda=0$, 
\begin{equation}\label{deffid}
\begin{cases}
F_1(d,y)= (1-|d|^2)^{\frac 1{p-1}}\vc{\ds\frac{y\cdot e_1(d)+|d|}{(1+d\cdot y)^{\frac {p+1}{p-1}}}}{0},\\
F_i(d,y)= (1-|d|^2)^{\frac {p+1}{2(p-1)}}\vc{\ds\frac{y\cdot e_i(d)}{(1+d\cdot y)^{\frac {p+1}{p-1}}}}{0}\mbox{ for }i=2,\dots,N,
\end{cases}
\end{equation}
where the orthonormal basis $(e_1(d),\dots, e_N(d))$ is given in the remark following Lemma \ref{lembase}.\\
(ii) For all $i=0,\dots,N$,
it holds that 
$\|F_i(d)\|_{\H}\le C$.
\end{lem}
\begin{proof} The proof is similar to the $1$-dimensional case (see Appendix \ref{subeigenld*}). 
\end{proof}
 In order to compute the projectors on the eigenfunctions of $L_d$ for $\lambda=0$ and $\lambda=1$, we need to compute the eigenfunctions of $L_d^*$ for the same eigenvalues, where $L_d^*$ is the conjugate operator of $L_d$ with respect to the inner product $\phi$. Let us first compute $L_d^*$.
\begin{lem}[The conjugate operator of $L_d$]
For any $|d|<1$, the operator $L_d^*$ conjugate of $L_d$ with respect to $\phi$ is given by
\begin{equation}\label{pesanti}
L_d^* \vc{r_1}{r_2} = \vc{R_d(r_2)}{- \q L r_1 + r_1 + \frac{p+3}{p-1} r_2 + 2y\cdot\nabla r_2-\frac{4\alpha r_2}{1-|y|^2}}
\end{equation}
for any $(r_1, r_2) \in \left(\D(\q L)\right)^2$, 
where $r=R_d(r_2)$ is the unique solution of
\begin{equation}\label{defR}
- \q L r+r = \q L r_2 + \psi(d,y) r_2.
\end{equation}
\end{lem}
\begin{proof}
The proof is omitted since it is the same as in one space dimension (see Lemma 4.1 page 81 in \cite{MZjfa07}). 
\end{proof}

Let us now find the eigenfunctions of $L_d^*$ associated to the eigenvalues $\lambda = 1$ and $\lambda=0$.
\begin{lem}[Eigenfunctions of $L_d^*$ associated with the eigenvalues $\lambda=1$ and $\lambda=0$]\label{eigenld*} 
$ $\\
(i) For all $|d|<1$ and $0\le i\le N$, there exists $W_i(d)\in \q H$ such that $L_d^* W_i=\lambda_i W_i$, with $\lambda_0=1$, $\lambda_i=0$ for $i\ge 1$,
\begin{equation}\label{defWl2}
W_{i,2}(d,y)=c_{\lambda_i}\left(\frac{1-|y|^2}{1-|d|^2}\right)^{\lambda_i}F_{i,1}(d,y),\;\;\frac 1{c_{\lambda_i}} = 2(\lambda_i+\alpha)\int\left(\frac{y_1^2}{1-|y|^2}\right)^{1-\lambda_i}\rho(y)dy,
\end{equation}
and $W_{i,1}(d)$ is uniquely determined as the solution $v_1$ of
\begin{equation}\label{eqWl1}
-\q L v_1 + v_1 = \left(\lambda_i - \frac{p+3}{p-1}\right)v_2 - 2 y\cdot \nabla v_2+ 4\alpha \frac{v_2}{1-|y|^2}
\end{equation}
with $v_2= W_{i,2}(d)$.\\
(ii) {\bf (Orthogonality)} For all $|d|<1$, $i,j=0,\dots,N$, we have $\phi(F_i,W_j)=\delta_{i,j}$.\\
(iii) {\bf (Normalization)} For all $0<|d|<1$, $i=0,\dots,N$ and $j=2,\dots,N$, we have
\begin{align*}
\|W_i(d)\|_{\H}&\le C\|W_{i,2}(d)\|_{L^2_{\frac \rho{1-|y|^2}}}\le C,&\|\partial_{e_1} W_i(d)\|_{\H}&\le\frac C{1-|d|},\\
\|\partial_{e_j} W_i(d)\|_{\H'}&\le\frac C{|d|\sqrt{1-|d|}},&
\mbox{ where }\|W\|_{\H'}&= \ds\sup_{\|v\|_{\H}=1}|\phi(W, v)|.
\end{align*}
\end{lem}
\begin{proof} See Appendix \ref{subeigenld*}.
\end{proof}

\subsection{Expansion of $q$ with respect to the eigenspaces of $L_d$}
In the following, we expand any $q\in \H$ with respect to the eigenspaces of $L_d$ partially computed in Lemma \ref{l10}. We claim the following: 
\begin{defi}[Expansion of $q$ with respect to the eigenspaces of $L_d$]\label{lemexpansion} Consider $q\in \H$ and introduce
\begin{equation}\label{defpdi}
\pi^d_i(q)=\phi\left(W_i(d), q\right)\mbox{ for all }i=0,\dots,N,
\end{equation}
where $W_i(d)$ is the eigenfunction of $L_d^*$ computed in Lemma \ref{eigenld*}, and $\pi^d_-(q)=q_-$ defined by
\begin{equation}\label{expansion}
q=\sum_{i=0}^N \pi^d_i(q)F_i(d,y)+\pi^d_-(q).
\end{equation}
\end{defi}
\begin{nb} 
Applying the operator $\pi^d_i$ to \eqref{expansion}, we see from Lemma \ref{eigenld*} that 
\begin{equation}\label{defhd-}
\pi^d_-(q)\in\H^d_-=\left\{r\in \H\;\;|\;\;\pi^d_i(r)=0\mbox{ for all }i=0,\dots,N\right\}.
\end{equation}
Note that if $q\in \H^d_-$, then $\pi^d_-(q)=q$. Moreover, $\pi^d_-(F_i(d))=0$.
In fact, $\pi^d_i(q)$ is the projection of $q$ on $F_i(d)$, the eigenfunction of $L_d$ associated to $\lambda=0$ or $1$, and $\pi^d_-(q)$ is the negative part of $q$. 
\end{nb}

\medskip

For the proof of the main theorem, we will need to prove in some sense dispersive estimates on $q_-=\pi^d_-(q)$ when $q$ is a solution to \eqref{eqq}. In order to achieve this, we need to manipulate some functional of $q_-$ (equivalent to the norm $\|q_-\|_{\H}=\phi(q_-,q_-)^{1/2}$  in $\H^d_-$) which will capture the dispersive character of the equation \eqref{eqq}. Such a quantity will be 
\begin{eqnarray}
\varphi_d\left(q, r\right)&=& \int_{|y|<1} \left(-\psi(d,y)q_1r_1+\nabla q_1\cdot \nabla r_1-(y\cdot \nabla r_1)(y\cdot \nabla q_1)+q_2r_2\right)\rho dy\label{defphid}\\
&=& 
\int_{|y|<1} \left(-q_1\left(\q L r_1+\psi(d,y) r_1\right)+q_2 r_2 \right)\rho dy\nonumber
\end{eqnarray}
where $\psi(d,y)$ is defined in \eqref{defld}. This bilinear form is in fact the second variation of $E(w(s))$ 
around $\kappa(d,y)$ \eqref{defkd}, the stationary solution of \eqref{eqw}, and can be seen as the energy norm in $\H^d_-$ (space where it will be definite positive).
More precisely, we have the following:
\begin{prop}[Equivalence in $\H^d_-$ of the $\H$ norm and the $\varphi_d$ norm]\label{lemdefpos}
There exists $C_0>0$ such that for all $|d|<1$, $q_-\in \H^d_-$ and $q \in \H$, 
\begin{align}
\frac 1{C_0}\left\|q_-\right\|_{\H}^2\le \varphi_d\left(q_-,q_-\right)\le C_0\left\|q_-\right\|_{\H}^2,\label{i}\\
\frac 1{C_0}\left\|q\right\|_\H \le \left(\sum_{i=0}^N\left|\pi^d_i(q)\right|+ 
\sqrt{\varphi_d\left(q_-, q_-\right)}\right)\le C_0\left\|q\right\|_{\H},\label{ii}
\end{align}
\end{prop}
\begin{proof} The proof follows the same pattern as the one-dimensional case (see Appendix \ref{appcoer}).
\end{proof}

\section{Stability results for equation \eqref{eqw}}\label{secrig}
This section is devoted to the proof of Theorem \ref{proptrap}. We proceed in 3 subsections:\\
- at first, we give a companion of our trapping result, namely the asymptotic stability of the zero solution for equation \eqref{eqw}, resulting in a non-degeneracy property at non-characteristic points (see Proposition \ref{prop0} and Corollary \ref{propnd} below);\\
- then, we give the proof of Theorem \ref{proptrap};\\
- finally, we give an exponential convergence property for solutions of equation \eqref{eqw} near $\{\pm\kappa(d,y)\}$ (see Proposition \ref{propexpo} below), then, we indicate how to adapt the proof of Theorem \ref{proptrap} to prove that result. 
\subsection{Stability of the zero solution for equation \eqref{eqw}}\label{subcorstat}

It happens that the techniques we developed for the proof of Theorem \ref{proptrap} yield also the asymptotic stability of the zero solution for equation \eqref{eqw}, as stated in the following:
\begin{prop}[Asymptotic stability of the zero solution for equation \eqref{eqw}]\label{prop0} There exist $\epsilon_1>0$ and $\delta_1>0$ such that if $\ee\equiv\left\|(w_{\bar x}(\bar s),\partial_s w_{\bar x}(\bar s))\right\|_{\H}\le \epsilon_1$, then:\\
(i) for all $s\ge \bar s$, we have $\left\|(w_{\bar x}(s),\partial_s w_{\bar x}(s))\right\|_{\H}\le \frac\ee{\delta_1} e^{-\delta_1(s-\bar s)}$;\\
(ii) if in addition, $\bar x \in \RR$, then, for all $s\ge s_0(\bar x)$, 
$\left\|(w_{\bar x}(s),\partial_s w_{\bar x}(s))\right\|_{H^1\times L^2(|y|<1)}\le C(\bar x) \ee e^{-\delta_1(s-\bar s)}$, for some $s_0(\bar x)\ge \bar s$ and $C_0(\bar x)>0$.
\end{prop}
From the local Cauchy theory for equation \eqref{equ}, this result yields the  non-degeneracy of blow-up at non-characteristic points stated in the following:
\begin{cor}[Non-degeneracy of the limit of $w_{\bar x}$ when $\bar x$ is non-characteristic]\label{propnd} There exists $\bar \epsilon_1>0$ such that for all $\bar x \in \RR$ and $s\ge -\log T(\bar x)$, we have $\|(w_{\bar x}(s), \partial_s w_{\bar x}(s))\|_{\q H}\ge \bar\epsilon_1$.
\end{cor}
\begin{nb} From Theorem 1.6 page 1131 of \cite{MZimrn05} and the similarity variables transformation \eqref{defw}, we have the following lower bound in the $H^1\times L^2(|y|<1)$ norm:

\medskip

{\it There exists $\bar \epsilon_0>0$ such that for any $\bar x \in \RR$ and $s\ge -\log T(\bar x)$, we have
}
\begin{equation}\label{lb0}
 \|(w_{\bar x}(s), \partial_s w_{\bar x}(s))\|_{H^1\times L^2(|y|<1)}\ge \bar\epsilon_0.
\end{equation}
That result follows from the combination of the solution of the Cauchy problem of equation \eqref{equ} and the finite-speed of propagation.\\ 
Since $H^1\times L^2(|y|<1)\subset \neq \q H$, our result in Corollary \ref{propnd} is stronger. We would like to mention that our result is analogous to what Giga and Kohn proved for the Sobolev-subcritical semilinear heat equation in \cite{GKcpam89}.
\end{nb}

\bigskip


In the following, we first derive Corollary \ref{propnd} from Proposition \ref{prop0}, then we give the proof of Proposition \ref{prop0}.


\begin{proof}[Proof of Corollary \ref{propnd} assuming that Proposition \ref{prop0} holds]
We will show that the conclusion of Corollary \ref{propnd} holds with $\bar \epsilon_1 = \epsilon_1$ already introduced in Proposition \ref{prop0}. Proceeding by contradiction, we assume that for some $\bar x\in \RR$ and $\bar s \ge - \log T(\bar x)$, we have 
\begin{equation*}
\|(w_{\bar x}(\bar s), \partial_s w_{\bar x}(\bar s))\|_{\q H}\le \epsilon_1.
\end{equation*}
Applying Proposition \ref{prop0} and recalling that $\bar x$ is non-characteristic, we see that
\[
\left\|(w_{\bar x}(s),\partial_s w_{\bar x}(s))\right\|_{H_1\times L^2(|y|<1)}\to 0 \mbox{ as }s\to \infty.
\]
From the lower bound proved in Theorem 1.6 page 1131 of \cite{MZimrn05} and stated in \eqref{lb0} above, we reach a contradiction. This concludes the proof of Corollary \ref{propnd}, assuming that Proposition \ref{prop0} holds. 
\end{proof}

\bigskip

Now, we give the proof of Proposition \ref{prop0}, showing the stability of the $0$ solution for equation \eqref{eqw}.
\begin{proof}[Proof of Proposition \ref{prop0}]
Let us first recall from  \cite{MZajm03} (see also \eqref{dE} below) that the functional
$E(w(s), \partial_sw(s))$ is a Lyapunov functional for equation \eqref{eqw}, where 
\begin{equation}\label{defenergy}
E(r) = \iint \left(\frac 12 r_2^2+ \frac 12 |\nabla r_1|^2 -\frac 12 (y.\nabla r_1)^2+\frac{(p+1)}{(p-1)^2}r_1^2 - \frac 1{p+1} |r_1|^{p+1}\right)\rho dy.
\end{equation}
For simplicity in the notation, we will often write $E(w(s))$ instead of $E(w(s), \partial_sw(s))$.\\ 
Consider $\bar x \in \m R^N$, $\bar s\ge - \log T(\bar x)$ and
 \begin{equation}\label{init}
\ee\equiv\left\|(w(\bar s),\partial_s w(\bar s))\right\|_{\H},\mbox{ where }w=w_{\bar x}. 
\end{equation}
We proceed in 3 steps:\\
- First, using the Hardy-Sobolev inequality of Lemma \ref{lemhs} together with the monotonicity of $E$ \eqref{defenergy}, we prove the stability of $0$, with no exponential decay.\\
- Then, slightly perturbing $E$ \eqref{defenergy}, we find a new Lyapunov functional, equivalent to the square of the norm in $\q H$, and which decreases exponentially fast, finishing the proof of item (i).\\
- Finally, assuming that $\bar x \in \RR$, we use a covering argument from \cite{MZimrn05} to show that the norm in $H^1 \times L^2(|y|<1)$ exponentially decreases too, finishing the proof of item (ii).

\bigskip

{\bf Step 1: Stability of the zero solution for equation \eqref{eqw}}

We claim the following: 
\begin{lem}[Stability of $0$ for equation \eqref{eqw}]\label{lem00} There exists $\delta_3\in (0,1)$ such that if 
\begin{equation}\label{cond1}
\ee \le \frac{\delta_3^2}2,
\end{equation}
then, for all $s\ge \bar s$, 
\[
\|(w(s),\partial_s w(s))\|_{\q H}\le \frac {\ee}{\delta_3}\mbox{ and }
\delta_3 E(w(s)) \le \|(w(s),\partial_s w(s))\|_{\q H}^2 \le \frac{E(w(s))}{\delta_3}.
\]
\end{lem}
\begin{proof}
Note first from the definition \eqref{defenergy} of $E$ and the Hardy-Sobolev inequality of Lemma \ref{lemhs} that for some $\delta_3\in (0, 1)$, 
\begin{equation}\label{equiv3}
\mbox{if }\|v\|_{\q H}\le \delta_3,\mbox{ then }\delta_3 E(v) \le \|v\|_{\q H}^2 \le \frac{E(v)}{\delta_3}.
\end{equation}
Assuming \eqref{cond1}, we see from \eqref{init} that $\|(w(\bar s),\partial_s w(\bar s))\|_{\q H}\le \delta_3$, hence from \eqref{equiv3}, we have $E(w(\bar s))\le \frac{\ee^2}{\delta_3}$. Using the monotonicity of $E$ (see \cite{MZajm03} or \eqref{dE} below), we obtain
\begin{equation}\label{bde}
\forall s\ge \bar s,\;\;E(w(s))\le \frac{\ee^2}{\delta_3}.
\end{equation}
Therefore, using \eqref{equiv3}, we see that in order to conclude the proof of Lemma \ref{lem00}, it is enough to prove that
\[
\forall s\ge \bar s,\;\; \|(w(s),\partial_s w(s))\|_{\q H}\le \delta_3.
\]
Using \eqref{cond1} and \eqref{init}, we may define
\[
s^*\equiv \sup\{s>\bar s,\;|\; \|(w(s),\partial_s w(s))\|_{\q H}\le \delta_3\}\in (\bar s, +\infty]
\]
(the fact that $s^*\ge \bar s$ comes from the continuity of the wave flow).
Assuming by contradiction that $s^*<+\infty$, we see by continuity that
\begin{equation}\label{contra}
\forall s\in [\bar s, s^*],\;\;\|(w(s),\partial_s w(s))\|_{\q H}\le \delta_3\mbox{ and }\|(w(s^*),\partial_s w(s^*))\|_{\q H}= \delta_3.
\end{equation}
Therefore, \eqref{equiv3} holds for $s=s^*$, and using \eqref{bde}
and \eqref{cond1}, we see that $\|w(s^*)\|_{\q H}^2 \le \frac{E(w(s^*))}{\delta_3} \le \frac{\ee^2}{\delta_3^2}
\le \frac{\delta_3^2}4$, which is a contradiction by \eqref{contra}. This concludes the proof of Lemma \ref{lem00}.
\end{proof} 

\bigskip

{\bf Step 2: An exponentially decreasing Lyapunov functional}

Introducing
\begin{equation}\label{defh0}
h(s) = E(w) + \eta_4 \iint w \partial_s w \rho dy
\end{equation}
where $\eta_4>0$ will be fixed small enough shortly, we claim that the following lemma allows us to conclude with the norm in $\q H$:
\begin{lem}[An exponentially decaying Lyapunov functional]\label{exp0} There exists $\delta_4>0$ and $\epsilon_4\in(0, \frac{\delta_3^2}2)$ such that if 
$\ee \le \epsilon_4$, then for all $s\ge \bar s$, 
\[
\frac{\delta_3}2 h(s)\le \|(w(s), \partial_s w(s))\|_{\q H}^2 \le \frac 2{\delta_3}h(s)\mbox{ and }h'(s) \le - \delta_4 h(s).
\]
\end{lem}
Indeed, if this lemma holds
and
$\ee\le \epsilon_4$, we see that for all $s\ge \bar s$, we have $h(s)\le e^{-\delta_4(s-\bar s)}h(\bar s)$. Using the first estimate of Lemma \ref{exp0}, we see that
\[
\forall s\ge \bar s,\;\; \|(w(s), \partial_s w(s))\|_{\q H}^2 \le \frac 4{\delta_3^2}e^{-\delta_4(s-\bar s)}\|(w(\bar s), \partial_s w(\bar s))\|_{\q H}^2,
\]
which is the desired estimate of Proposition \ref{prop0}.
It remains then to prove Lemma \ref{exp0} in order to conclude the proof of Proposition \ref{prop0}.

\medskip

\begin{proof}[Proof of Lemma \ref{exp0}] Assume first that $\ee \le \frac{\delta_3^2}2$. 
Since $\left|\iint w \partial_s w \rho dy\right|\le \|(w,\partial_s w)\|_{\q H}^2$ by definition \eqref{defnh} of the norm in $\q H$, we see from Lemma \ref{lem00} and 
\eqref{defh0} that the first estimate in Lemma \ref{exp0} follows when $\eta_4\le \frac{\delta_3}2$. We then focus on the second estimate.\\
Multiplying equation \eqref{eqw} by $\partial_s w \rho$ and integrating, we get the well-known energy identity
\begin{equation}\label{dE}
\frac d{ds} E(w(s)) = - 2 \alpha \iint (\partial_s w)^2\frac \rho{1-|y|^2} dy, 
\end{equation}
where $\alpha>0$ is defined in \eqref{defro}.\\ 
Using equation \eqref{eqw}, integration by parts and the definition \eqref{defenergy} of $E$, we see that
\[
\frac {d}{ds} \iint w\partial_s w \rho dy = -2 E(w) + 2\iint (\partial_s w)^2 \rho dy +\frac{p-1}{p+1}\iint |w|^{p+1} \rho dy -G
\]
where
\[
G = -\frac{p+3}{p-1} \iint w \partial_s w \rho dy -2 \iint w y\cdot \nabla \partial_s w \rho dy.
\]
Using again integration by parts, the Hardy-Sobolev inequality of Lemma \ref{lemhs} and Lemma \ref{lem00}, we see that for some $C_4>0$,
we have for all $s\ge \bar s$, 
\begin{align*}
&|G|\le C_4\iint (\partial_s w)^2\frac \rho{1-|y|^2} + \delta_3\|w\|_{\q H_0}^2\le C_4\iint (\partial_s w)^2\frac \rho{1-|y|^2}+E(w),\\
&\frac{p-1}{p+1}\iint |w|^{p+1} \rho\le C_4\|w\|_{\q H_0}^{p+1}\le C_4 \ee^{p-1}\|w\|_{\q H_0}^2\le C_4 \ee^{p-1}E(w)
\end{align*}
(see page 112 of \cite{MZjfa07} for details on the bound for $G$).\\
Collecting the various estimates above, we write from 
\eqref{defh0} 
that for all $s\ge \bar s$, 
\[
h'(s) \le (-2\alpha +\eta_4(C_4+2))\iint (\partial_s w)^2\frac \rho{1-|y|^2}
+(-\eta_4+C_4 \ee^{p-1})E(w(s)).
\]
Fixing 
\[
\eta_4 = \min\left(\frac{\delta_3}2,\frac{2\alpha}{C_4+2}\right) 
\mbox{ and taking }
\ee \le \epsilon_4 \equiv \min\left(\frac{\delta_3^2}2,\left(\frac{\eta_4}{2C_4}\right)^{\frac 1{p-1}}\right),
\]
we see from Lemma \ref{lem00} and the first part of Lemma \ref{exp0} that
\[
\forall s\ge \bar s,\;\;h'(s) \le -\frac {\eta_4} 2E(w) \le -\frac{\eta_4 \delta_3^2}4 h(s)
\]
which concludes the proof of Lemma \ref{exp0}. 
\end{proof}
Since we have already seen that the exponential decreasing of the norm in $\q H$ follows from Lemma \ref{exp0}, this concludes the proof of Proposition \ref{prop0} in the general case. It remains to handle the case where $\bar x\in \RR$.

\bigskip

{\bf Step 3: A covering technique for the $H^1\times L^2(|y|<1)$ norm}

Assume now that $\bar x\in \RR$ and that 
\begin{equation}\label{petit0}
\ee\equiv\left\|(w_{\bar x}(\bar s),\partial_s w_{\bar x}(\bar s))\right\|_{\H}\le \frac{\epsilon_1}2,
\end{equation}
where $\epsilon_1>0$ is the constant that makes item (i) of the proposition works. \\
 Note first that since $\bar x$ is non-characteristic, using the uniform bound on the solution at blow-up (Theorem 2' in \cite{MZimrn05}) and the covering technique in that paper (Proposition 3.3 in \cite{MZimrn05}), we get for all $s\ge -\log T(\bar x)$, 
\begin{equation}\label{petit1}
\left\|(w_{\bar x}(s),\partial_s w_{\bar x}(s))\right\|_{H^1\times L^2 \left(|y|<\frac 1{\eta_0}\right)} \le K
\end{equation}
for some $\eta_0(\bar x)\in(0,1)$ and $K(\bar x)>0$.\\
Introduce $\bar t = T(\bar x)-e^{-\bar s}$ and consider $x'\in \m R^N$ such that $|x'-\bar x|\le \frac{e^{-\bar s}}{\eta_0}$. Consider then $w_{x',\bar T(x')}(y',s')$, the similarity variables transformation of $u(x,t)$, centered at $(x', \bar T(x'))$, where
 \begin{equation*}
\bar T(x') = T(\bar x) - \eta_0|x'-\bar x|
\end{equation*}
(note that the point $(x',\bar T(x'))$ lays on the boundary of the backward cone of vertex $(\bar x, T(\bar x))$ and slope $\eta_0$). From \eqref{defw}, we see that for any $s'\ge \sigma(x',\bar s)$, the function $w_{x',\bar T(x')}(y',s')$ is defined for all $|y'|<1$, 
\begin{equation}\label{delta}
w_{x',\bar T(x')}(y',s')= (1+\delta)^{\frac 2{p-1}}w_{\bar x}(y,s)\mbox{ with }\delta = (\bar T(x')-T(\bar x))e^s=-\eta_0|x'-\bar x|e^s,
\end{equation}
and
\begin{equation}\label{s'y}
 s'=\sigma(x',s)=s-\log(1+\delta),\;\; y = y'(1+\delta) + e^s(x'-\bar x).
\end{equation}
Using the bound \eqref{petit1} with $s=\bar s$, we see from Lebesgue's theorem that
\[
\|(w_{x',\bar T(x')}(\bar s'), \partial_s w_{x',\bar T(x')}(\bar s'))\|_{\q H}\to \|(w_{\bar x}(\bar s), \partial_s w_{\bar x}(\bar s))\|_{\q H}\mbox{ as }x'\to \bar x
\]
with $\bar s'= \sigma(x', \bar s)$.
Therefore, using \eqref{petit0}, we see that for some small enough $\epsilon^*(\bar x)>0$, if $ |x'-\bar x|< \epsilon^*$, then
\[
\|(w_{x',\bar T(x')}(\bar s'), \partial_s w_{x',\bar T(x')}(\bar s'))\|_{\q H}\le \epsilon_1.
\] 
Since $\epsilon_1$ is the constant that makes item (i) of Proposition \ref{prop0} works (see the beginning of Step 3 above), applying that item, we see that for all $s'\ge \bar s'$, we have 
\[
 \|(w_{x',\bar T(x')}(s'), \partial_s w_{x',\bar T(x')}(s'))\|_{\q H}\le \frac{\epsilon_1}{\delta_1}e^{-\delta_1(s'-\bar s')},
\]
for some $\delta_1>0$. Integrating only on the ball $\{|y'|<\frac 12\}$, we get rid of the weights and see that if $ |x'-\bar x|< \epsilon^*$ and $s'\ge \bar s'$, then
\[
 \|(w_{x',\bar T(x')}(s'), \partial_s w_{x',\bar T(x')}(s'))\|_{H^1\times L^2(|y'|<\frac 12)}\le C \frac{\epsilon_1}{\delta_1}e^{-\delta_1(s'-\bar s')}.
\]
Using the covering method of Proposition 3.3 in \cite{MZimrn05}, we see that for $s\ge \max(\bar s,-\log \epsilon^*)$, 
\begin{align}
&\|(w_{\bar x}(s), \partial_s w_{\bar x}(s))\|_{H^1\times L^2(|y|<1)}\nonumber\\
\le&C\sup_{|x'-\bar x|\le e^{-s}}\|(w_{x',\bar T(x')}(s'), \partial_s w_{x',\bar T(x')}(s'))\|_{H^1\times L^2(|y'|<\frac 12)}\le C \frac{\epsilon_1}{\delta_1}e^{-\delta_1(s'-\bar s')}.\label{gojo}
\end{align}
Since we see from \eqref{delta} and \eqref{s'y} that for  $s\ge \max(\bar s,-\log \epsilon^*)$ and $|x'-\bar x|\le e^{-s}$, we have 
\[
e^{-\delta_1(s'-\bar s')} = e^{-\delta_1(s-\bar s)}\left(\frac{1-\eta_0|x'-\bar x|e^s}{1-\eta_0|x'-\bar x|e^{\bar s}}\right)^{\delta_1}\le  e^{-\delta_1(s-\bar s)}(1-\eta_0)^{-\delta_1}, 
\]
we see that item (ii) follows from \eqref{gojo}. 
This concludes the proof of Proposition \ref{prop0}. Since we have already derived Corollary \ref{propnd} from Proposition \ref{prop0}, this finishes the justification of Corollary \ref{propnd} too.
\end{proof}

\subsection{Behavior of solutions of equation \eqref{eqw} near $\pm \kappa(d,y)$}
We give the proof of Theorem \ref{proptrap} here, showing the dynamics of equation \eqref{eqw} near the set $\{ \pm \kappa(d,y)\}$. 

\medskip

\begin{proof}[Proof of Theorem \ref{proptrap}]
Note first that the last two sentences of Theorem \ref{proptrap} follow directly from the first part. Indeed, if $w(y,s)=w_{\bar x}(y,+\sigma)$ for some $\bar x\in \m R$, $\sigma \ge- \log T(\bar x)$ where $w_{\bar x}$ is defined in \eqref{defw} from a blow-up solution $u$, then we know from the Cauchy theory that $u$ is continuous in time with values in $H^1 \times L^2$ of slices of the truncated backward light cone $\q C_{\bar x, T(\bar x), 1} \cap \{t\ge 0\}$. Since $\q H\subset H^1\times L^2(|y|<1)$, it follows that $w_{\bar x}$ is continuous in time with values in $\q H$. Since by definition, $w_{\bar x}$ is defined for all $(y,s) \in B(0,1)\times [\- \log T(\bar x))$, the first possibility of Theorem \ref{proptrap} never occurs. If in addition, $\bar x \in \RR$, the the second possibility is ruled out thanks to the non-degeneracy of blow-up at non-characteristic points stated in Corollary \ref{propnd}. Thus, we only prove the first sentence of Theorem \ref{proptrap}.\\
From the invariance of equation \eqref{equ}, we reduce to the case where $\bar\omega=1$, and assume that for some
solution $w$ of equation \eqref{eqw} continuous in time with values in $\q H$,
 $|\bar d|<1$ and $\eb>0$, we have 
\begin{equation}\label{trap}
\ee\equiv\left \|\vc{w(0)}{\partial_s w(0)}-\vc{\kappa(\bar d)}{0}\right\|_{\H}\le \eb.
\end{equation}
We proceed in two parts in order to prove Theorem \ref{proptrap}:\\
- In the first part, we give a modulation technique near the set $\{\pm \kappa(d,y)\}$, killing the $N+1$ nonnegative directions of $L_d$, the linearized operator around $\kappa(d,y)$ (see Lemma \ref{l10} above);\\
- In the second part, giving the dynamics of the linearized equation, we finish the proof of Theorem \ref{proptrap}.

\bigskip

{\bf Part 1: A modulation technique}

From the continuity of the flow associated with equation \eqref{equ} in $\q H$, we see that the solution $w$ will stay close to that set, at least for a short time after $0$. 
In fact, we can do better, and impose some orthogonality conditions, killing the expanding and zero directions of $L_d$, the linearized operator of equation \eqref{eqw} around $\kappa(d,y)$.\\ 
This needs the introduction of $\kappa^*(d,\nu,y) = (\kappa_1^*, \kappa_2^*)(d,\nu,y)$ with
\begin{equation}\label{defk*}
\kappa_1^*(d,\nu, y) =
\ds\kappa_0\frac{(1-|d|^2)^{\frac 1{p-1}}}{(1+d.y+\nu)^{\frac 2{p-1}}},\;
\kappa_2^*(d,\nu, y) = \nu \pnu \kappa_1^*(d,\nu, y) =
\ds-\frac{2\kappa_0\nu}{p-1}\frac{(1-|d|^2)^{\frac 1{p-1}}}{(1+d.y+\nu)^{\frac {p+1}{p-1}}}.
\end{equation}
Note that for all $\mu \in \m R$, 
$\kappa^*_1(d,\mu e^s,y)$ is a particular solution to equation \eqref{eqw}.\\
Let us consider $A\ge 2$ to be fixed later large enough
and give the following:
\begin{lem}[A modulation technique]\label{lemmod}For all $A\ge 2$, there exists $C_1(A)>0$ and $\eun(A)>0$ such that if $v\in \H$, $|\bar d|<1$ and $\bar \nu\in (-1,\infty)$ satisfy
\begin{equation}\label{condnu}
-1+\frac 1A \le \frac{\bar \nu}{1-|\bar d|}\le A\mbox{ and }
\|v- \kappa^*(\bar d,\bar \nu)\|_{\H}\le \eun,
\end{equation}
then, there exist $|d|<1$ and $\nu\in (-1,\infty)$ (both continuous as a function of $v$) such that
\begin{equation}\label{ortho}
\forall i=0,\dots,N,\;\;\pi^{d^*}_i(q)=0,
\end{equation}
where $d^*=\frac d{1+\nu}$ and the operator $\pi_i^{d^*}$ is defined in \eqref{defpdi}. Moreover, we have 
\begin{align}
&\|v-\kappa^*(d,\nu)\|_{\q H} 
+|\arg\tanh |d| - \arg\tanh |\bar d||+\left|\frac\nu{1-|d|} - \frac{\bar \nu}{1-|\bar d|}\right|
+\frac{|d-\bar d|}{\sqrt{1-|\bar d|}}\nonumber\\
\le& C_1\|v- \kappa^*(\bar d,\bar \nu)\|_{\q H}.\label{proximite}
\end{align}
\end{lem}
\begin{nb} One may think that condition \eqref{ortho} is ambiguous, since it involves $\pi_i^d$ whose definition \eqref{defpdi} depends on the choice of the basis $(e_1(d),\dots, e_N(d))$ given in Lemma \ref{lembase}. This is not the case, since if \eqref{ortho} holds for some choice of that basis, then it holds for any other basis. 
\end{nb}
\begin{proof} As in the one-dimensional case (see Lemma 2.1 in \cite{MZisol10}), the result follows from the application of the implicit function theorem, no more. However, since there are new directions to kill, in comparison with the one-dimensional case (take $i=2,\dots,N$ in \eqref{ortho}), we end-up with some technical difficulties, involving in particular a Lipschitz property on $\kappa^*(d,\nu)$ in item (v) of Claim \ref{propk*} below. For details, we refer the interested reader to Appendix \ref{appmod}. 
\end{proof} 
Assuming that $\eb$ appearing in \eqref{trap} is smaller than $\eun$ defined in Lemma \ref{lemmod} (note that we will fix $\eb>0$ even smaller later), we see that condition \eqref{condnu} in Lemma \ref{lemmod} is satisfied with $v=w(y,0)$, $\bar d$ and $\bar \nu=0$. Applying Lemma \ref{lemmod} and using the continuity of the flow associated with equation \eqref{eqw} with values in $H^1\times L^2$, hence in $\q H$, 
we get the existence of a maximal $\hat s\in (0,\infty]$ such that $w$ can be modulated, in the sense that
\begin{equation}\label{defq}
\vc{w(y,s)}{\partial_s w(y,s)}=\kappa^*(d(s), \nu(s))+q(y,s)
\end{equation}
where the $C^1$ parameters $d(s)$ and $\nu(s)$ are such that for all $s\in [0, \hat s)$,
\begin{equation}\label{kill}
\pi^{d^*(s)}_i(q) =0\mbox{ for all }i=0,\dots,N, \mbox{ where }d^*(s) = \frac{d(s)}{1+\nu(s)},
\end{equation}
with
\begin{equation}\label{condmod}
-1+\frac 1A \le \frac {\nu(s)}{1-|d(s)|}\le A\mbox{ and }
\|q(s)\|_{\q H}\le \eb.
\end{equation}
When $s=0$, we obtain better estimates. Indeed, using \eqref{trap} and \eqref{proximite}, we see that
\begin{align}
\|q(0)\|_{\q H}
+|\arg\tanh |d(0)| - \arg\tanh |\bar d||+\frac{|\nu(0)|}{1-|d(0)|}
+\frac{|d(0)-\bar d|}{\sqrt{1-|\bar d|}}
&\le C_1(A) \ee.\label{difficile}
\end{align}
Considering an arbitrary $\eta \le \frac 12$ (that will be fixed later small enough in terms of $A$), and assuming that $\eb \le \frac \eta{C_1(A)}$, we introduce from \eqref{difficile}, \eqref{trap} and \eqref{condmod} $\tilde s = \tilde s(\eta) \in (0, \hat s]$ such that for all $s\in [0, \tilde s)$
\begin{equation}\label{cond0}
\frac{|\nu(s)|}{1-|d(s)|}\le \eta \mbox{ and }\|q(s)\|_{\q H}\le \eb.
\end{equation}
Following \eqref{cond0}, two cases then arise:\\
- either $\tilde s = \infty$;\\
- or $\tilde s <\infty$, and  \eqref{cond0} holds at $s=\tilde s$, with one of the two $\le$ symbols in \eqref{cond0} replaced by a $=$ symbol.\\
At this stage, we see that controlling the solution $w(s)$ in $\q H$ for $s\in [0,\hat s)$ is equivalent to controlling the three components defined in \eqref{defq}: $q(s)\in \q H$, $|d(s)|<1$ and $\nu(s) >-1$. 

\bigskip

{\bf Part 2: Behavior of equation \eqref{eqw} near $\kappa(d,y)$ and conclusion of the proof}

Adapting the techniques of \cite{MZisol10}, we obtain the dynamics of $q$, $d$ and $\nu$ in the following:
\begin{prop}[Dynamics of the parameters]\label{propdyn} 
There exist $\edeux(A)>0$, $\cdeux(A)>0$, $\eta_2(A)>0$  and $\mu_2(A)>0$ such that if $\eb\le \edeux$ and $\eta\le \eta_2$, then:\\
For all $s\in [0, \hat s)$,
\begin{align}
\frac {|\nu'-\nu|}{1-|d|}
+\frac{|d' \cdot e_1|}{1-|d|}
+\sum_{i=2}^N\frac{|d' \cdot e_i|}{\sqrt{1-|d|}} 
&\le \cdeux \|q(s)\|_{\q H}\left(\eb + \frac{|\nu|}{1-|d|}\right),\label{first}\\
\|q(s)\|_{\q H}^2 &\le \cdeux e^{C_2s}\|q(0)\|^2;\label{second}
\end{align}
For all $s\in [0, \tilde s)$,
\begin{equation}\label{third}
\|q(s)\|_{\q H}^2 \le \cdeux e^{-2\mu_2s}\|q(0)\|^2.
\end{equation}
\end{prop}
\begin{proof} The proof follows from the projection of equation \eqref{eqw} on the different components of $w(y,s)$ introduced in \eqref{defq}. We have already performed that projection in several of our former papers: \cite{MZjfa07}, \cite{MZajm11}, \cite{MZisol10} and \cite{CZmulti11}. For that reason, we leave the proof of this proposition to Appendix \ref{appdyn}.
\end{proof}
Now, we are in a position to finish the proof of Theorem \ref{proptrap}. Introducing
\begin{equation}\label{defzeta}
 \zeta(s) = -\arg\tanh |d(s)|,
 \end{equation}
  we see that $\zeta'(s) = -\frac{d'\cdot e_1}{1-|d|^2}$. Assuming that
\begin{equation}\label{smallbe}
A\ge 2,\;\;
\eta \le \eta_2(A),\;\;\eb\le \min\left(\epsilon_1(A),\frac \eta{C_1(A)}, \epsilon_2(A)\right)\mbox{ and }\ee \le \frac{\eb}{2C_1(A)\sqrt{C_2(A)}},
\end{equation}
we see from Proposition \ref{propdyn} and \eqref{difficile} that
\begin{equation}\label{fruit}
\begin{array}{lrl}
\forall s\in [0, \hat s),\;\;&
\frac {|\nu'(s)-\nu(s)|}{1-|d(s)|}+\frac{|d'(s)|}{\sqrt{1-|d(s)|}}+|\zeta'(s)|&\le C_3(A)\|q(s)\|_{\q H},\\
\forall s\in [0, \tilde s),\;\;&\|q(s)\|_{\q H}^2 \le \cdeux(A) C_1(A)^2\ee^2e^{-2\mu_2(A)s}&\le \frac{\epsilon_0^2}4,
\end{array}
\end{equation}
for some $C_3(A)>0$. In the following, we handle the two cases mentioned in the alternative following \eqref{cond0}. More precisely, we will show that if the various constants are suitably chosen, the first case leads to \eqref{conv0} and the second to \eqref{impossible}, which is the desired conclusion of Theorem \ref{proptrap}. 

\medskip

{\bf Case 1: $\tilde s=+\infty$}. {\it We will show in this case that \eqref{conv0} holds with \eqref{proxi0} satisfied, whatever is the value of $A\ge 2$, provided that $\eta$ and $\eb$ are small enough in terms of $A$, and $\ee$ small enough in terms of $A$ and $\epsilon_0$}.\\
Note that in this case, $\hat s = +\infty$ too. From \eqref{fruit} and \eqref{cond0}, we see that $\zeta(s)\to \zeta_\infty\in \m R$, $d(s) \to d_\infty$ and $\nu(s) \to 0$, as $s\to \infty$. Moreover, we have for all $s\ge 0$,
\begin{align}
&|\zeta(s) - \zeta_\infty|\le C_4(A)\ee e^{-\mu_2s},\label{kifkif}
\mbox{ hence }
\frac 1{C_4(A)}\le  \frac{1-|d_\infty|}{1-|d(s)|}\le C_4(A),\\
&|d(s)-d_\infty|\le C_4(A) \ee \sqrt{1-|d_\infty|}e^{-\mu_2s}\label{rough},\\
&|\nu(s)|\le C_4(A) \ee(1-|d_\infty|)e^{-\mu_2s},\mbox{ hence }
\frac {|\nu(s)|}{1-|d(s)|}\le C_4(A) \ee e^{-\mu_2s}\label{estnu}
\end{align}
for some $C_4(A)>0$. Assuming that in addition to \eqref{smallbe}, we have
\begin{equation}\label{simply}
\ee \le \frac 1{4C_4(A)},\mbox{ hence }\forall s\ge 0,\;\;|d(s) - d_\infty|\le \frac 14,
\end{equation}
we claim that
\begin{equation}\label{alech}
\frac{|l_{d_\infty}(d(s)-d_\infty)|}{1-|d_\infty|}
\le C_5(A)\ee e^{-\mu_2s},
\end{equation}
for some $C_5(A)>0$. 
Indeed, if $|d_\infty|\le \frac 12$, this follows from \eqref{rough}. If not, then
we see from \eqref{simply} that the norm of all points in the segment $[d(s), d_\infty]$ is bounded from below by $\frac 14$. Using Lemma \ref{lembase}, we see that
\begin{equation}\label{proche}
|e_1(d(s))-e_1(d_\infty)|\le C|d(s)-d_\infty| \le CC_4(A) \ee \sqrt{1-|d_\infty|}e^{-\mu_2s}.
\end{equation}
Therefore, using \eqref{first}, \eqref{fruit}, \eqref{kifkif} and \eqref{proche}, we see that
\[
\frac{|d'(s) \cdot e_1(d_\infty)|}{1-|d_\infty|}
\le CC_4(A) \ee e^{-\mu_2s}.
\]
Integrating this on the interval $[s,+\infty)$, we obtain \eqref{alech}. 
If we further impose in addition to \eqref{smallbe} and \eqref{simply} that
\[
C_5(A) \ee +C_4(A)^2\ee^2 \le \equatre(A), 
\]
where $\equatre(A)$ is introduced in Claim \ref{propk*}, then we see from \eqref{condmod} that we can apply that claim and get from \eqref{estnu}, \eqref{rough} and \eqref{alech}
\[
\|\kappa^*(d(s),\nu(s))- (\kappa(d_\infty),0)\|_{\q H}\le C_6(A)\ee e^{-\mu_2s}
\]
for some $C_6(A)>0$. 
Using \eqref{fruit} and a triangular inequality, we get \eqref{conv0}.
Taking $s=0$ in \eqref{kifkif} and \eqref{alech}, we see that
\begin{equation}\label{manar}
|\arg\tanh |d(0)|-\arg\tanh |d_\infty||+\frac{|d(0)-d_\infty|}{\sqrt{1-|d(0)|}}\le C_7(A)\ee
\end{equation}
for some $C_7(A)>0$. Since we see from \eqref{difficile} that 
\begin{equation}\label{c8}
\frac 1{C_8(A)} \le \frac{1-|\bar d|}{1-|d(0)|}\le C_8(A)
\end{equation}
 for some $C_8(A)>0$, \eqref{proxi0} follows from \eqref{manar} and \eqref{difficile}. 

\bigskip

{\bf Case 2: $\tilde s<+\infty$}.\label{case2} {\it In this case, we will show that \eqref{impossible} holds, provided that $A$ is fixed large enough, and as before, $\eta$ and $\eb$ are small enough in terms of $A$, and $\ee$ small enough in terms of $A$ and $\epsilon_0$}.\\
 From \eqref{cond0}, \eqref{fruit} and \eqref{condmod}, we see that 
\begin{equation}\label{alter0}
\nu(\tilde s)=\theta \eta(1-|d(\tilde s)|)\mbox{ for some }\theta=\pm 1
\end{equation}
 and that $\hat s >\tilde s$. The following lemma allows us to conclude:
\begin{lem}\label{lem2} For any $A\ge 2$, there exists $C_9(A)>0$ such that for any $\eta \in (0, \eta_2(A))$, there exists $\epsilon_9(A,\eta)>0$ such that if $\epsilon_0\le \epsilon_9(A,\eta)$, then:\\
(i) $\hat s -\tilde s \le C_9(A)$;\\
(ii) For all $s\in [\tilde s, \hat s]$, $\|q(s)\|_{\q H} \le C_9(A) \ee$;\\
(iii) For all $s\in  [\tilde s, \hat s]$, $\frac{1-|d(s)|}{1-|\bar d|}\le C_9(A)$.
\end{lem}
Indeed, let us assume this Lemma and finish the proof of Theorem \ref{proptrap}, then we will give the proof of Lemma \ref{lem2}.\\
 Take $A\ge 2$ to be fixed large enough soon, then fix $\eta = \eta_2(A)$ defined in Proposition \ref{propdyn}. Then, impose that $\eb \le \epsilon_9(A,\eta_2(A))$. From item (i), we see that $\hat s <+\infty$. Imposing further that $C_9(A) \ee \le \frac \eb 2$, we see from item (ii) and \eqref{condmod} that 
\begin{equation}\label{alter}
\mbox{either }\frac{\nu(\hat s)}{1-|d(\hat s)|} = -1+\frac 1A\mbox{ or }\frac{\nu(\hat s)}{1-|d(\hat s)|} = A.
\end{equation}
Remember that up to this point in the proof, the value of $A$ is arbitrary in the interval $[2,+\infty)$. As a matter of fact, we will fix now $A$ large enough to show that the first case in \eqref{alter} 
implies that $(w(s), \partial_s w(s))$ is not defined for all $(y,s) \in B(0,1) \times [0, \infty)$,
and that in the second case, $(w(s), \partial_s w(s))\to 0$ as $s\to \infty$, which gives \eqref{impossible} and concludes the proof of Theorem \ref{proptrap}.\\
From items (iii) and (iv) in Claim \ref{propk*}, let us fix $A$ large enough such that
\begin{equation}\label{eneg}
\sup_{|d|<1} E\left(\kappa^*\left(d,(-1+\frac 1A)(1-|d|)\right)\right)\le -2
\mbox{ and }
\sup_{|d|<1}\|\kappa^*(d,A(1-|d|))\|_{\q H}\le \frac{\delta_1}2,
\end{equation}
where $\delta_1>0$ is defined in Proposition \ref{prop0}.
From item (i) in Claim \ref{propk*}, item (i) in Lemma \ref{lemE}, \eqref{defq} and \eqref{condmod}, we see that
\begin{equation}\label{pro}
|E(w(\hat s)) - E(\kappa^*(d(\hat s), \nu(\hat s))|\le C(A)\|q(\hat s)\|_{\q H} \le C(A) \eb\le 1,
\end{equation}
provided that $\eb \le \epsilon_{12}(A)$, for some $\epsilon_{12}(A)>0$ small enough. \\
Assume now that the first case in \eqref{alter} occurs. From \eqref{eneg} and \eqref{pro}, we see that $E(w(\hat s))\le -1$. Using item (ii) of Lemma \ref{lemE}, we see that $w(s)$ cannot be defined for all $(y,s) \in B(0,1) \times [\hat s, +\infty)$, and the first possibility in Theorem \ref{proptrap} holds.\\
Now, if the second case in \eqref{alter} occurs, 
then, we see from \eqref{eneg} and \eqref{condmod} that
\[
\|(w(\hat s), \partial_s w(\hat s))\|_{\q H}\le \frac{\delta_1}2+\epsilon_0\le \delta_1,
\]
provided that $\epsilon_0\le \frac{\delta_1}2$. Therefore, Proposition \ref{prop0} applies and we see that
\[
\forall s\ge \hat s,\;\;\|(w(s), \partial_s w(s))\|_{\q H}\le \frac 1{\delta_1} e^{-\delta_1(s-\hat s)}. 
\]
Thus, case \eqref{impossible} holds. It remains then to prove Lemma \ref{lem2} in order to conclude the proof of Theorem 1.

\medskip

\begin{proof}[Proof of Lemma \ref{lem2}]$ $\\
(i) The idea is simple: on the interval $[\tilde s, \hat s)$, from \eqref{fruit} and \eqref{condmod}, we will see that $\zeta(s)$ does not change much and so does $1-|d(s)|$. Therefore, $\nu$ satisfies the differential inequality $|\nu'-\nu|\le 1-|d(\tilde s)|$, provided that $\epsilon_0$ is small enough. Since this equation needs a finite time 
to take $\eta(1-|d(\tilde s)|)$ to $A(1-|d(\tilde s)|)$ and  $-\eta(1-|d(\tilde s)|)$ to $(-1+\frac 1A)(1-|d(\tilde s)|)$, we get the conclusion. We first claim the following :
\begin{cl} \label{clhadaf} For any $A\ge 2$, there exists $C_{10}(A)>0$ such that for any $\eta \in (0, \eta_2(A))$ and $L>0$, there exists $\epsilon_{10}(A,L, \eta)>0$ such that if $\epsilon_0\le \epsilon_{10}(A,L, \eta)$, then for all $s\in [0, \min(\tilde s +L, \hat s)]$,
\begin{equation}\label{hadaf}
|\zeta(s)-\zeta(0)|\le C_{10}(A),\;\;
\frac 1{C_{10}(A)}\le \frac{1-|d(s)|}{1-|d(0)|} \le C_{10}(A)
\mbox{ and }
|\nu'(s) - \nu(s)|\le \frac \eta 2 (1-|d(\tilde s)|),
\end{equation}
where $\zeta(s)=-\arg\tanh |d(s)|$ is defined in \eqref{defzeta}.
\end{cl}
\begin{proof} Clearly, the second estimate follows from the first. Furthermore, using \eqref{first} and \eqref{condmod}, we see that the third follows from the second, so, we only prove the first.\\
If $s\in [0, \tilde s]$, we write from \eqref{fruit}, \eqref{third},  \eqref{condmod} (with $s=0$) and \eqref{smallbe} that $|\zeta'(s)|\le C_{11}(A) e^{-\mu_2s}$ for some $C_{11}(A)>0$. Integrating this on the interval $[0, s]$, we get the result.\\
If $s\in [\tilde s, \min(\tilde s+L, \hat s)]$, using \eqref{fruit} and \eqref{cond0}, we see that $|\zeta'(s)|\le C_{12}(A)\eb$. Integrating this on the interval $[\tilde s, s]$ we see that $|\zeta(s)-\zeta(\tilde s)|\le  C_{12}(A)\eb(s-\tilde s)\le C_{12}(A)\eb L\le 1$ if $\eb$ is smaller than some $\epsilon_{10}(A,L)$ small enough. Since \eqref{hadaf} already holds with $s=\tilde s$, we conclude the proof of Claim \ref{clhadaf}.
\end{proof}
As we said right before Claim \ref{clhadaf}, $L$ should be thought as the minimal time needed to take the variable $\frac{\nu(s)}{1-|d(s)|}$ from $\eta$ to $A$ and from $-\eta$ to $-1+\frac 1A$. From \eqref{hadaf}, we see that $L$ is larger than $\bar L(A,\eta)$, which is the time needed to take the variable 
\begin{equation}\label{defnb}
\bar \nu(s) \equiv \frac{\nu(s)}{1-|d(\tilde s)|}
\end{equation}
 from $\eta$ to $C_{10}(A)^2A$ and from $-\eta$ to $-C_{10}(A)^2$, given that $\bar \nu$ satisfies the differential equation 
\begin{equation}\label{eqnb}
|\bar \nu '-\bar\nu|\le \frac \eta 2.
\end{equation}
 Since we easily see that $\bar L=\log \left(\frac{2C_{10}(A)^2A}\eta \right)$, fixing
\begin{equation}\label{defL}
L(A, \eta)= \bar L(A, \eta)+\log 2=\log \left(\frac{4C_{10}(A)^2A}\eta \right),
\end{equation}
than imposing that $\eb \le \epsilon_{10}(A,L,\eta)$, we claim that $\hat s \le \tilde s +L$. Indeed, assuming by contradiction that $\hat s > \tilde s +L$, we see from Claim \ref{clhadaf} that for all $s\in [\tilde s, \tilde s +L]$, \eqref{eqnb} is satisfied, where $\bar \nu$ is defined in \eqref{defnb}. Integrating that equation on the interval $[\tilde s, s]$, we see that
\[
|\bar \nu(\tilde s+L)-e^L\bar \nu(\tilde s)|\le \frac \eta 2.
\]
Using the alternative in \eqref{alter0} and the definition \eqref{defL} of $L$, we see that:\\
- either $\bar \nu(\tilde s)=\eta$, hence $\bar \nu(\tilde s+L)\ge \eta e^L -\frac \eta 2\ge \frac \eta 2e^L=2C_{10}(A)^2A$;\\
- or $\bar \nu(\tilde s)=-\eta$, hence $\bar \nu(\tilde s+L)\le -\eta e^L +\frac \eta 2\le -\frac \eta 2e^L=-2C_{10}(A)^2$.\\
Using \eqref{hadaf}, we see that for $s=\tilde s +L<\hat s$,
\[
\mbox{ either }\frac{\nu(s)}{1-|d(s)|}\ge 2A,\mbox{ or }\frac{\nu(s)}{1-|d(s)|}\le -2A. 
\]
From \eqref{cond0}, this is a contradiction. This concludes the proof of item (i) in Claim \ref{lem2}.\\
(ii) When $s\in [0, \tilde s]$, this comes from \eqref{fruit}. Since the estimate holds for $s=\tilde s$, using \eqref{second} together with item (i), we get the result when $s\in [\tilde s, \hat s]$.\\
(iii) This is a direct consequence of Claim \ref{clhadaf} and \eqref{c8}.\\
This concludes the proof of Lemma \ref{lem2}.
\end{proof}
Since we have already finished the proof of Theorem \ref{proptrap}, assuming that Lemma \ref{lem2} holds, this is also the conclusion of the proof of Theorem \ref{proptrap}. 
\end{proof}

\subsection{Exponential convergence for solutions of equation \eqref{eqw} near the set $\{\pm \kappa(d,y)\}$}

We would like to mention another consequence of our techniques: an exponential convergence property for solutions of equation \eqref{eqw} 
near $\{\pm\kappa(d,y)\}$:
\begin{prop}[Exponential convergence for solutions of \eqref{eqw} near $\{\pm\kappa(d,y)\}$] \label{propexpo} There exists $\delta_2>0$ such that for any $s_1\ge s_0$, if $w\in C([s_0,s_1],\q H)$ is a solution of equation \eqref{eqw} satisfying
\begin{equation}\label{proche0}
\forall s\in [s_1,s_2],\;\;\left\|\vc{w(s)}{\partial_s w(s)}-\bar\omega\vc{\kappa(\bar d(s))}{0}\right\|_{\q H}\le \delta_2,
\mbox{ with }|\bar d(s)|<1\mbox{ and }\bar \omega=\pm 1, 
\end{equation}
then, there exist $C^1$ parameters $d(s)$ and $\nu(s)$ such that
\[
\forall s\in [s_0,s_1],\;\;
\|q(s)\|_{\q H}\le \frac{e^{-\delta_2(s-s_0)}}{\delta_2}\|q(s_0)\|_{\q H}
\]
where $q(y,s) =(w(y,s),\partial_s w(y,s)) - \bar \omega \kappa^*(d(s), \nu(s))$.  
\end{prop}
\begin{nb}As one may see from the proof, the parameters $d(s)$ and $\nu(s)$ are chosen so that we kill the two nonnegative modes of the solutions. This way, the solution lays in the negative part of the spectrum, where for $\frac{\nu}{1-|d|}$ small, we have an exponentially decreasing Lyapunov functional, which is equivalent to the square of the norm. 
\end{nb}
\begin{proof}
The proof is a by-product of our proof of Theorem \ref{proptrap}, as we explain in the following.\\
Consider $s_1\ge s_0$ and $w\in C([s_0, s_1], \q H)$ a solution of equation \eqref{eqw} satisfying \eqref{proche0}. If we take $A=2$ and assume that $\delta_2 \le \epsilon_1(2)$ defined in Lemma \ref{lemmod}, then that lemma applies, and we have the existence of $C^1$ parameters $d(s)$ and $\nu(s)$ such that for all $s\in [s_0,s-1]$,
\[
\forall i=1,\dots,N,\;\;\pi^{d^*(s)}_i(q(s))=0
\]
and 
\[
\|q(s)\|_{\q H}+\frac{|\nu(s)|}{1-|d(s)|}\le C_1(2)\delta_2,
\] 
where $q(y,s) =(w(y,s),\partial_s w(y,s)) - \bar \omega \kappa^*(d(s), \nu(s))$.
Consider then the constants $\epsilon_0$ and $\eta$ appearing in the proof of Theorem \ref{proptrap} (in particular in \eqref{cond0}) and fixed at the the end of the proof (see Case 2 page \pageref{case2}). Assuming then that $\delta_2$ is small enough so that
\[
C_1(2)\delta_2\le \min(\eta, \epsilon_0),
\]
we see that \eqref{cond0} holds for all $s\in [s_0, s_1]$. Since it is easy to see from the proof of Theorem \ref{proptrap} that estimate \eqref{third} holds on the same interval where \eqref{cond0} holds, we see that \eqref{third} holds on the interval $[s_0, s_1]$, which closes the proof of Proposition \ref{propexpo}.
\end{proof}

\appendix

\section{The Lorentz transform in similarity variables}\label{applor}
In this section, we introduce the similarity version of the Lorentz transform and prove some related estimates.\\
Let us introduce the following transformation derived from the Lorentz transform,
and which keeps solutions of equation \eqref{eqw} invariant:
\begin{lem}{\bf (The Lorentz transform in similarity variables)}\label{proplorentzw} Consider $w(y,s)$ a solution of equation \eqref{eqw} defined for all $|y|<1$ and $s\in (s_0, s_1)$ for some $s_0$ and $s_1$ in $\m R$, and introduce for any $|d|<1$:\\
(i) The coordinates $Y$ defined by the fact that 
\begin{equation}\label{deftd}
y=\theta_d(Y)\equiv \frac 1{1+d\cdot Y}\left[d+\frac{(1-\sqrt{1-|d|^2})}{|d|^2}(d\cdot Y) d+ \sqrt{1-|d|^2}Y\right].
\end{equation}
(ii) The function $W\equiv \q T_d(w)$ defined by
\begin{equation}\label{defW}
W(Y,S)=\frac{(1-|d|^2)^{\frac 1{p-1}}}{(1+d\cdot Y)^{\frac 2{p-1}}}w(y,s) \mbox{ where }y= \theta_d(Y) \mbox{ and }s=S-\log \frac{1+d\cdot Y}{\sqrt{1-|d|^2}}.
\end{equation}
Then, $W(Y,S)=\q T_d(w)$ is also a solution of \eqref{eqw} defined for all $|Y|<1$ and 
$S\in \left(s_0+\frac 12 \log\frac{1+|d|}{1-|d|},
s_1-\frac 12 \log\frac{1+|d|}{1-|d|}\right)$.
\end{lem}
\begin{nb}
From easy calculations, we see that
\begin{equation}\label{invtd}
(\q T_d)^{-1} = \q T_{-d}.
\end{equation}
If $w(y)$ is a stationary solution of \eqref{eqw}, then the function $W(Y)= \q T_d(w)$ depends only on $Y$ and is also a stationary solution of \eqref{eqw}. Introducing $l_d(x) = x\cdot \frac d{|d|}$ if $d\neq 0$ and $l_0(x)=0$ if $d=0$, then $\bot_d(x)$ such that $x=l_d(x)+\bot_d(x)$, we see that the coordinate transformation $y=\theta_d(Y)$ becomes 
\begin{equation}\label{yortho}
 l_d(y)=\frac{l_d(Y)+|d|}{1+d\cdot Y},\;\;
\bot_d(y)=\frac{\sqrt{1-|d|^2}}{1+d\cdot Y}\bot_d(Y).
\end{equation}
\end{nb}
\begin{proof} The proof is omitted since it is similar to the one-dimensional case given in Lemma 2.6 page 54 in \cite{MZjfa07}.
\end{proof}

\medskip

Introducing the projection of the space $\H$ \eqref{defnh} on the first coordinate:
\begin{equation}\label{defnh0}
\H_0 = \left\{r\in H^1_{loc}\;\;|\;\;\|r\|_{\H_0}^2\equiv \iint \left(r^2+ |\nabla r|^2-(y\cdot \nabla r)^2)\right)\rho dy<+\infty\right\},
\end{equation}
we give the following properties for the transformation $\q T_d$, which are simple adaptations from the one-dimensional case (see in particular Lemma 2.7 page 56 and Lemma 2.8 page 57 in \cite{MZjfa07}):
\begin{lem}\label{0lemeffect} $ $\\
(i) {\bf (Transformation of $\kappa(\bar d)$)} For all $|d|<1$ and $|\bar d|<1$, we have $\q T_d(\kappa(\bar d))=\kappa(\theta_d(\bar d))$, where $\theta_d(\bar d)$ is defined in \eqref{deftd}.\\ 
(ii) {\bf (Transformation of the linearized operator of \eqref{eqw} around $\kappa_0$)}  If we consider $w(y,s)$ and $W= \q T_d w$, then it holds that
\begin{eqnarray*}
&&\partial^2_sw-\left(\q L w+\frac{2(p+1)}{p-1}w
-\frac{p+3}{p-1}\partial_sw-2y\cdot \nabla\partial_s w\right)\\
=\frac{(1+d\cdot Y)^{\frac{2p}{p-1}}}{(1-|d|^2)^{\frac p{p-1}}}&&
\left(\partial^2_SW-\left(\q L W + \psi(d,Y)W
-\frac{p+3}{p-1}\partial_SW-2Y\cdot \nabla\partial_S W\right)\right)
\end{eqnarray*}
where $\psi(d,Y)$ is defined in \eqref{defld}.\\
(iii) {\bf (Transformation of the $L^2_\rho$ inner product)}
For all $V_1$ and $V_2$ in $L^2_\rho$,
\begin{align}
\int_{|Y|<1} V_1(Y) V_2(Y) \rho(Y) dY&= \iint \frac{1-|d|^2}{(1-d\cdot y)^2}v_1(y) v_2(y)\rho(y) dy\nonumber\\
\int_{|Y|<1} V_1(Y) V_2(Y) \frac{\rho(Y)}{1-|Y|^2} dY&= \iint v_1(y) v_2(y)\frac{\rho(y)}{1-|y|^2} dy\nonumber
\end{align}
where $v_i = \q T_{-d}V_i$.\\ 
(iv) {\bf (Continuity of $\q T_d$ in $\H_0$)}
There exists $C_0>0$ such that for all $|d|<1$ and $w\in \H_0$, we have
\begin{equation*}
\frac 1{C_0} \|w\|_{\H_0}\le \|\q T_{d}(w)\|_{\H_0} \le C_0 \|w\|_{\H_0}.
\end{equation*}
(v) {\bf (Norm of $\nabla_d \q T_d$)} 
There exists $C_0>0$ such that for all $|d|<1$ and $i=1,\dots,N$, if $w\in \H_0$ and $\partial_{y_i}w\in \H_0$, then
\begin{equation*}
\|\q \partial_{d_i}(\q T_{d}(w))\|_{\H_0} \le \frac{C_0}{1-|d|^2}\left(\|w\|_{\H_0}+ \sum_{i=1}^N \|\partial_{y_i}w\|_{\H_0}\right).
\end{equation*}
\end{lem} 
\begin{nb} If we consider $w(y,s)= w(y)$ in item (i) and $W= \q T_d w$, then it holds that
\begin{equation}\label{translw}
\q L w(y) + \frac{2(p+1)}{p-1}w(y)
=\frac{(1+d\cdot Y)^{\frac{2p}{p-1}}}{(1-|d|^2)^{\frac p{p-1}}}
\left(\q L W(Y) + \psi(d,Y) W(Y)\right).
\end{equation}
\end{nb}
\begin{proof}$ $\\
(i) This is straightforward by definition of $\q T_d$ given in item (ii) of Lemma \ref{proplorentzw}.\\
(ii) The one-dimensional proof given in 
Lemma 2.7 page 56 in \cite{MZjfa07} is valid in $N$ dimensions.\\
(iii) This is a direct consequence of the change of variables $Y\mapsto y$ suggested by the definition \eqref{defW} of the transformation $\q T_d$ and the weight $\rho(y)$ \eqref{defro}.\\
(iv) As in one space dimension (see Lemma 2.8 page 57 in \cite{MZjfa07}), we see from \eqref{invtd} that it is enough to prove the right-hand side identity.
 That proof is calculatory and difficult (though not impossible) to carry out in higher dimensions. For that reason, we present a different proof based on a 
small idea: 
the use of \eqref{translw}, which follows from the just proved item (ii).\\
Consider now $d\neq 0$, $w\in \H_0$ and $W=\q T_d(w)$. Making the transformation $Y\mapsto y$ expressed in \eqref{yortho}, we see from \eqref{defW} and \eqref{translw} that
\begin{align*}
\int_{|Y|<1}W(Y)^2 \frac{\rho(Y)}{1-|Y|^2} dY &= \iint \frac{w(y)^2 \rho(y)}{1-|y|^2}  dy;\\
\int_{|Y|<1} W(Y)(\q L W(Y) + \psi(d,Y) W(Y))\rho dY
&= \iint w(y)(\q L w(y) + \frac{2(p+1)}{p-1}w(y)) \rho dy.
\end{align*}
Since $\psi(d,Y) \le \frac C{1-|Y|^2}$, making an integration by parts in the second line,
we see that 
\[
\int_{|Y|<1}\left(|\nabla W|^2-(Y\cdot \nabla W)^2\right) \rho dY
\le \iint\left(|\nabla w|^2-(y\cdot \nabla w)^2\right) \rho dy
+C\int_{|Y|<1}\frac{ W^2 \rho dY}{1-|Y|^2}.
\]
Since $\int_{|Y|<1}W(Y)^2 \rho(Y) dY\le \int_{|Y|<1}W(Y)^2 \frac{\rho(Y)}{1-|Y|^2} dY$, we see that the result follows from the definition \eqref{defnh0} of the norm in $\q H_0$ and the following identity proved in Appendix B in \cite{MZajm03}:
 \begin{equation}\label{hsajm}
\iint |h(y)|^2\frac{\rho(y)}{1-|y|^2}dy
\le C\iint \left(|\nabla h|^2(1-|y|^2) + h^2\right) \rho(y) dy\le C\|h\|_{\H_0}^2.
\end{equation} 
This concludes the proof of item (iv).\\
(v) Consider $|d|<1$ and $w\in \q H_0$ such that $\partial_{y_j} w\in \q H_0$ for all $j=1,\dots,N$. If $i=1,\dots,N$, we differentiate the expression \eqref{defW} with respect to $d_i$ and see that $\partial_{d_i} W(Y,S)$ is a linear combination with bounded coefficients depending only on $d$, of the following terms: 
$\q T_d \left(\frac{w(y)}{1-|d|^2}\right)$,    
$\q T_d \left(\frac{Y_i w(y)}{1+d\cdot Y}\right)$,    
$\q T_d \left(\frac{\partial_{y_j} w(y)}{1+d\cdot Y}\right)$,    
$\q T_d \left(Y_i\frac{\partial_{y_j} w(y)}{1+d\cdot Y}\right)$ and
$\q T_d \left(Y_iy_j\frac{\partial_{y_j} w(y)}{1+d\cdot Y}\right)$. Replacing $Y$ by its value in terms of $y$ and $d$ (see \eqref{deftd}), we see that $\partial_{d_i} W(Y,S)$ is a linear combination with coefficients depending only on $d$ and bounded by $\frac C{1-|d|^2}$, of the following terms:
$\q T_d(w)$, $\q T_d(y_k w)$, $\q T_d(\partial_{y_l}w)$, $\q T_d(y_k\partial_{y_l}w)$ and  $\q T_d(y_ky_m\partial_{y_l}w)$,  where $k$, $l$ and $m$ are between $1$ and $N$. Since we have
\[
\|\q T_d(y_k w)\|_{\q H_0} \le C \|\q T_d(w)\|_{\q H_0},\;
\|\q T_d(y_k\partial_{y_l}w)\|_{\q H_0} + \|\q T_d(y_ky_m\partial_{y_l}w)\|_{\q H_0}  \le C\|\q T_d(\partial_{y_l}w)\|_{\q H_0},
\]
this concludes the proof of Lemma \ref{0lemeffect}. 
\end{proof}

\section{Spectral properties in similarity variables}\label{subeigenld*}
In this section, we give several estimates related to the operators $L_d$, $L_d^*$ and $\q L$ which are involved in the similarity variables' version \eqref{eqw}. More precisely, we give here the proofs of Lemmas \ref{l10} and \ref{eigenld*}, then, we give in Proposition \ref{spectl} below some spectral properties of the operator $\q L$ \eqref{defro}.


\begin{proof}[Proof of Lemma \ref{l10}]\label{proofl10}$ $\\
(i) Note that for any $\mu \in \m R$ and $|d|<1$, $\kappa^*(d,\mu e^s, y)$ \eqref{defk*} is a particular solution of the following first order form of equation \eqref{eqw}:
\begin{equation}\label{eqw1}
\partial_s \begin{pmatrix} w_1 \\ w_2 \end{pmatrix} =
\begin{pmatrix} w_2 \\ \displaystyle \q L w_1 - \frac{2(p+1)}{(p-1)^2} w_1 + |w_1|^{p-1} w_1 - \frac{p-3}{p-1} w_2 - 2y \cdot\nabla w_2 \end{pmatrix}.
\end{equation}
Therefore, 
it follows that $\partial_\mu \kappa^*(d,0,y)$, $\partial_{d_i} \kappa^*(d,0,y)$ for all $i=1,\dots,N$, 
are particular solutions to the linearized equation around $\kappa^*(d,0,y) = (\kappa(d,y),0)$, which writes precisely $\partial_s (w_1, w_2) = L_d(w_1,w_2)$ by definition \eqref{defld} of $L_d$. Therefore, the same holds for $\partial_{e_i}\kappa^*(d,0,y)=e_i\cdot \nabla_d \kappa^*(d,0,y)$ for all $i=1,\dots,N$, where the orthonormal basis $(e_1,\dots,e_N)$ has been introduced in Lemma \ref{lembase}.\\
Since we have from \eqref{defk*}
\begin{align*}
\partial_\mu \kappa^*(d,0,y) &=-\frac{2\kappa_0e^s}{p-1} (1-|d|^2)^{\frac 1{p-1}}\vc{(1+d\cdot y)^{-\frac {p+1}{p-1}}}{(1+d\cdot y)^{-\frac {p+1}{p-1}}},\\
\nabla_d \kappa^*(d,0,y)&
=-\ds\frac{2\kappa_0(1-|d|^2)^{\frac 1{p-1}-1}}{p-1}\vc{\frac{d(1+d\cdot y) + y(1-|d|^2)}{(1+d\cdot y)^{\frac {p+1}{p-1}}}}{0}\label{linkfd0},
\end{align*}
hence 
\begin{align*}
\mbox{if }i=1,\;\;\partial_{e_1}\kappa^*(d,0,y)&=-\ds\frac{2\kappa_0(1-|d|^2)^{\frac 1{p-1}-1}}{p-1}\vc{\frac{|d|+y\cdot e_1}{(1+d\cdot y)^{\frac {p+1}{p-1}}}}{0}\\
\mbox{for }i\ge 2,\;\;\partial_{e_i}\kappa^*(d,0,y)&=-\ds\frac{2\kappa_0(1-|d|^2)^{\frac 1{p-1}}}{p-1}\vc{\frac{y\cdot e_i}{(1+d\cdot y)^{\frac {p+1}{p-1}}}}{0},
\end{align*}
 this concludes the proof of (i).\\
(ii) Let us note that if $V\in \H$ and $\bar V(z)=\bar V(z_1)$ with $z_1=y\cdot e_1(d)$, then  
\begin{equation}\label{defnhz1}
\|V\|_{\H}^2 =C(N) \int_{-1}^1\left((\bar V_1(z_1))^2 +(\partial_{z_1}\bar V(z_1))^2(1-z_1^2)+(\bar V_1(z_1))^2\right)(1-z_1^2)^{\frac 2{p-1}} dz_1,
\end{equation}
for some $C(N)>0$, 
which is precisely the one-dimensional expression for the norm, already computed in Lemma 4.2 page 83 in \cite{MZjfa07}. This is the case for $F_i(d)$ with $i=0$ or $i=1$.
As for $F_i(d)$ with $i\ge 2$, we note that its second component $F_{i,2}(d)=0$ and note also from \eqref{defW} that 
$F_{i,1}(d)=\q T_d(y\cdot e_i(d))$. Using item (iv) of Lemma \ref{0lemeffect}, we see that
\[
\|F_i(d)\|_{\H} = \|F_{i,1}(d)\|_{\H_0} \le C \|y \cdot e_i(d)\|_{\H_0} \le C. 
\]
This concludes the proof of Lemma \ref{l10}.
\end{proof}

\bigskip

Now, we give the proof of Lemma \ref{eigenld*}.

\begin{proof}[Proof of Lemma \ref{eigenld*}]$ $

\medskip

(i) The following claim allows us to conclude:
\begin{cl}\label{clfrancois*}$ $\\ 
(i) For any $v_2\in L^2_{\frac \rho{1-|y|^2}}$, 
equation \eqref{eqWl1}
has a unique solution $v_1\in \H_0$ such that 
\begin{equation*}
\|v_1\|_{\H_0}\le C \|v_2\|_{L^2_{\frac \rho{1-|y|^2}}}.
\end{equation*}
(ii)  If $L_0^*r=\lambda r$ and $R=(R_1,R_2)$ is such that $R_2(Y)=\frac{(1+d\cdot Y)^\lambda}{(1-|d|^2)^{\lambda/2}}\q T_d(r_2)$ and $R_1$ is the solution of equation \eqref{eqWl1} with $v_2=R_2$, then $L_d^*(R)=\lambda R$. 
\end{cl}
\begin{proof} $ $\\
(i) In the one-dimensional case, we proved the same statement with the space $L^2_{\frac \rho{1-|y|^2}}$ replaced by a smaller one, $\q H_0$ (see the embedding in \eqref{hsajm}, and see the proof of Lemma 4.5 page 88 in \cite{MZjfa07}). In fact, the adaptation to higher dimensions with the larger space $L^2_{\frac \rho{1-|y|^2}}$ costs only a simple integration by parts, therefore, the proof is omitted here.\\
(ii) The proof follows exactly as in one space dimension. See item (ii) of Claim 4.5 page 86 and also page 87 in \cite{MZjfa07}.
\end{proof}
Indeed, if $d=0$, one can check by hand that the solution is given by $r=(r_1, r_2)$ where $r_2(y)=1-|y|^2$ if $\lambda=1$ and $r_2(y)=y_i$ if $\lambda=0$, and $r_1$ is the solution of equation \eqref{eqWl1} with $v_2=r_2$.\\
If $d\neq 0$, the result follows by applying item (ii) in Claim \ref{clfrancois*} with $r_2(y) = 1-|y|^2$, then $r_2(y) = e_i(d) \cdot y$, where $(e_1(d), \dots, e_N(d))$ is the orthonormal basis introduced in Lemma \ref{lembase} (the fact that the eigenfunctions are in $\H$ follows from the explicit formulas given in \eqref{defWl2} and item (i) of Claim \ref{clfrancois*}.

\medskip

(ii) If $\lambda_i\neq \lambda_j$ (i.e. if $i=0$ and $j\ge 1$, or $j=0$ and $i\ge 1$), the result
follows from the orthogonality between eigenfunctions of $L_d$ and $L_d^*$ for different eigenvalues.\\
If $\lambda_i=\lambda_j$ (i.e. if $i=j=0$ or $i,j=1,\dots,N$), recalling from Lemma \ref{l10} that $F_{i,2}=\lambda_i F_{i,1}$,
% ...
then, using the expression \eqref{defphi2} of the inner product $\phi$ and \eqref{eqWl1}, we integrate by parts and obtain:
\begin{align}
&\phi(F_i, W_j) = \iint F_{i,1}(-\q LW_{j,1}+W_{j,1})\rho dy +\lambda_i \iint F_{i,1}W_{j,2}\rho dy\nonumber\\
&=\iint F_{i,1}((\lambda_i - \frac{p+3}{p-1})W_{j,2}-2y \cdot\nabla W_{j,2}+4 \alpha \frac {W_{j,2}}{1-|y|^2}+\lambda_i W_{j,2}) \rho dy.\label{phiwf}
\end{align}
{\bf Case 1}: 
If $i=j$, 
we see from item (i) in the lemma we are proving that $W_{i,2}=c_{\lambda_i}\frac{(1-|y|^2)^{\lambda_i}}{(1-|d|^2)^{\lambda_i}} F_{i,1}$. Therefore, 
\begin{align*}
&\frac{(1-|d|^2)^{\lambda_i}}{c_{\lambda_i}} \phi(F_i,W_i)=\left(2\lambda_i - \frac{p+3}{p-1}\right)\iint F_{i,1}^2 (1-|y|^2)^{\lambda_i}\rho dy\\
&  - \iint y \cdot \nabla [F_{i,1}(1-|y|^2)^{\lambda_i}]^2 \frac \rho {(1-|y|^2)^{\lambda_i}} dy
 +4 \alpha \iint F_{i,1}^2 (1-|y|^2)^{\lambda_i-1}\rho dy\\
&= \left(2\lambda_i - \frac{p+3}{p-1}\right)\iint F_{i,1}^2 (1-|y|^2)^{\lambda_i}\rho dy\\
&+\iint F_{i,1}^2 (1-|y|^2)^{2\lambda_i} \left(N \frac \rho{(1-|y|^2)^{\lambda_i}}-2|y|^2 (\alpha-\lambda_i) \frac \rho{(1-|y|^2)^{1+\lambda_i}}\right)\\
&=2(\lambda_i+\alpha)\iint F_{i,1}^2 (1-|y|^2)^{\lambda_i-1}\rho dy. 
\end{align*}
 Since we see from \eqref{defW} that 
 \begin{equation}\label{defgdi}
 F_{0,1}(d,y) =\q T_d(1-d\cdot y)\mbox{ and }F_{i,1}(d,y) = \q T_d(e_i(d) \cdot y),
 \end{equation}
using 
item (iii) of Lemma \ref{0lemeffect}, we see that $\phi(F_i(d), W_i(d))=1$.\\
{\bf Case 2}: 
If $j\neq i$ with $i\ge 1$ and $j\ge 1$ (in this case, $\lambda_i=\lambda_j=0$),  
using the definition \eqref{defWl2} of $W_j(d)$, we see that 
\begin{equation}\label{y.nabla}
y\cdot \nabla W_{j,2}(d,y) =c_0(1-|d|^2)^{\frac{p+1}{2(p-1)}}\frac{y\cdot e_j(d)}{(1+d\cdot y)^{\frac{p+1}{p-1}}}-\left(\frac{p+1}{p-1}\right)\frac{d\cdot y}{1+d\cdot y}W_{j,2}(d).
\end{equation}
Recalling the definition \eqref{deffid} of $F_{i,1}(d)$ and making the change of variables $y\mapsto z=(y\cdot e_1(d), \dots,y\cdot e_N(d))$ where the orthonormal basis is defined in Lemma \ref{lembase} (recall also that $d\cdot y = |d| y\cdot e_1(d)=|d|z_1$), we see from separation of variables that the integral in \eqref{phiwf} is zero, hence $\phi(F_i(d), W_j(d)) = 0$. This concludes the proof of item (ii) in Lemma \ref{eigenld*}.

\medskip

(iii) Take $0<|d|<1$ and $i=0,\dots,N$. We start with the proof of the first line, then, we prove the second, which is more subtle.\\
{\it - Proof of the first line in item (iii)}: Since $W_{i,1}(d)$ and $\partial_{e_1} W_{i,1}(d)$ are solutions of equation \eqref{eqWl1} with $v_2=W_{i,2}(d)$ or $\partial_{e_1} W_{i,2}(d)$ (respectively), we see from item (i) in Claim \ref{clfrancois*} that it is enough to bound the norms of the latter functions in $L^2_{\frac \rho{1-|y|^2}}$ in order to conclude.\\
Noting that for all $|y|<1$ and $|d|<1$,
\begin{equation}\label{bounds}
(1-|y|^2)+(1-|d|^2)
+\sup_{j=2,\dots,N}|y\cdot e_j|\sqrt{1-|d|^2}
\le C (1+d\cdot y),
\end{equation}
we see from the definitions of $W_{i,2}(d,y)$ and $\kappa(d,y)$ given in Lemma \ref{eigenld*} and \eqref{defkd} that
\begin{equation}\label{wkd}
|W_{i,2}(d,y)|+(1-|d|^2)|\partial_{e_1} W_{i,2}(d,y)|\le C \kappa(d,y).
\end{equation}
Since $\kappa(d,y)=\q T_d(\kappa_0)$ by definition \eqref{defW} of the transformation $\q T_d$, using the embedding \eqref{hsajm} together with the continuity of $\q T_d$ stated in item (iv) of Lemma \ref{0lemeffect}, we see that 
\begin{equation}\label{boundkd}
\|W_{0,2}(d)\|_{L^2_{\frac \rho{1-|y|^2}}}+(1-|d|^2)\|\partial_{e_1}W_{0,2}(d)\|_{L^2_{\frac \rho{1-|y|^2}}}\le C\|\kappa(d)\|_{L^2_{\frac \rho{1-|y|^2}}}
\le C \|\kappa(d)\|_{\q H_0}\le C,
\end{equation}
which closes the proof of the first line in item (iii).\\
%
%
%
%
%
%
{\it - The second line in item (iii)}: Take $j=2,\dots,N$. Using the expression \eqref{defphi2} of the inner product $\phi$, the elliptic equation \eqref{eqWl1} and the Hardy-Sobolev identity \eqref{hsajm}, we see that 
\begin{equation}\label{boundh'}
\|\partial_{e_j} W_i\|_{\H'} \le C\left\|\frac{\partial_{e_j}W_{i,2}}{1-|y|^2}\right\|_{L^2_{\rho(1-|y|^2)}}+C \|y \cdot \nabla \partial_{e_j}W_{i,2}\|_{L^2_{\rho(1-|y|^2)}}.
\end{equation}
Making a rotation of coordinates, we reduce to the case where the basis $(e_1(d),\dots, e_N(d))$ is the canonical basis of $\m R^N$. In other words, we reduce to the case where 
\begin{equation}\label{dreduc}
d=(|d|,0,\dots,0)
\end{equation}
 and we need to compute the norms of $\partial_{d_j}W_{i,2}$.\\
When $i=0$ or $1$, 
using \eqref{bounds}, \eqref{dreduc} and the definition \eqref{defWl2} of $W_{i,2}$,
we see after straightforward calculations that
\begin{equation}\label{bound0}
|y\cdot\nabla \partial_{d_j} W_{i,2}(d,y)|+
\frac{|\partial_{d_j} W_{i,2}(d,y)|}{1-|y|^2}
\le \frac{C|y_j|(1-|d|^2)^{\frac 1{p-1}}}{(1-|y|^2)(1+|d|y_1)^{\frac{p+1}{p-1}}}.
\end{equation}
Using the integral computation table of Claim \ref{cltech0}, 
we see that the $L^2_{\rho(1-|y|^2)}$ of the right-hand side is bounded by $C(1-|d|)^{-1}$, and the result for $i=0$ or $1$ follows from \eqref{bound0} and \eqref{boundh'}.\\
When $i=2,\dots,N$, we see from the definition \eqref{defWl2} of $W_{i,2}$ and \eqref{dreduc} that
\[
\partial_{d_j} W_{i,2}(d,y)=c_0(1-|d|^2)^{\frac{p+1}{2(p-1)}}
\left(\frac{y \cdot \partial_{e_j}e_i}{(1+|d|y_1)^{\frac{p+1}{p-1}}}-\frac{p+1}{p-1}\frac{y_iy_j}{(1+|d|y_1)^{\frac{2p}{p-1}}}\right).
\]
Since $|\partial_{e_j}e_i|\le C|d|^{-1}$ from Lemma \ref{lembase}, we see that
\[
|y\cdot\nabla \partial_{d_j} W_{i,2}(d,y)|+
\frac{|\partial_{d_j} W_{i,2}(d,y)|}{1-|y|^2}
\le \frac{C(1-|d|^2)^{\frac {p+1}{2(p-1)}}}{(1-|y|^2)(1+|d|y_1)^{\frac{p+1}{p-1}}}
\left(\frac 1{|d|^2}+\frac{|y_i||y_j|}{1+|d| y_1}\right).
\]
Arguing as in the case where $i=1$ or $1$, we get the result when $i=2,\dots,N$.This concludes the proof of Lemma \ref{eigenld*}.
\end{proof}

\bigskip

In the following, we give spectral properties of the operator $\q L$ defined in \eqref{defro}: 
\begin{prop}[Properties of the operator $\q L$ \eqref{defro}]\label{spectl}$ $\\
(i) The operator $\q L$ is self-adjoint and compact in $L^2_\rho(|y|<1)$. It has a discrete spectrum $(\gamma_n)_{n\in \m N}$ (an eigenvalue may be repeated in case of multiplicity) with $\gamma_0=0$, $\gamma_1=-\frac{2(p+1)}{p-1}$ and $\gamma_2=-\frac{2(3p+1)}{p-1}$ and the corresponding normalized eigenfunctions 
\begin{align}
 h_0(y)&=\hz,\; h_{1,i}(y)=\hu y_i\mbox{ for }1\le i\le N,\label{eigen}\\
h_2(y) &= \hd \left(a^2-|y|^2\right)\mbox{ with }a=\sqrt{\frac{N(p-1)}{3p+1}}<1.\nonumber
\end{align}
The normalized eigenfunction corresponding to $\gamma_n$ for $n\ge 3$ is denoted by $h_n$. They are such that the set
\begin{equation}\label{basis}
 \{h_0, h_{1,i}, h_2, h_n\;|\;i=1,\dots,N,\;n\ge 3\}
\end{equation}
 makes an orthonormal basis in $L^2_\rho(|y|<1)$.\\
(ii) If $u\in L^2_\rho$ with $\q L u \in L^2_\rho$ and
\begin{equation}\label{sania}
\iint u(y) \rho(y) dy = \iint u(y) y_i \rho(y) dy = 0\mbox{ for all }i=1,\dots,N,
\end{equation}
then, 
\begin{equation}\label{coercivity}
-\iint u \q L u \rho dy \ge |\gamma_2| \int u^2 \rho dy.
\end{equation}
In particular, 
\begin{equation}\label{gng2}
\forall n\ge 3,\;\; -\gamma_n \ge -\gamma_2 = \frac{2(3p+1)}{p-1}.
\end{equation}
 \end{prop}
\begin{proof} $ $\\
 (i) The fact that $\q L$ is self-adjoint in $L^2_\rho$ follows from the divergence form in \eqref{defro}. The fact that it is compact (hence with a discrete spectrum) follows from \eqref{hsajm}.
In order to check the 3 explicit eigenvalues, we expand the expression \eqref{defro} as follows:
\[
\q L w = \Delta w - \sum_{i,j}y_iy_j \partial^2_{y_i,y_j}w - \frac{2(p+1)}{p-1} y\cdot \nabla w,
\]
and directly derive that $\q L 1=0$, $\q L y_k=- \frac{2(p+1)}{p-1} y_k$ and $\q L (a^2-|y|^2) = -\frac{2(3p+1)}{p-1}(a^2-|y|^2)$ where $a$ defined in \eqref{eigen} satisfies $a<1$ from the condition \eqref{condp} on $p$.\\
We then fix the constants $c_m$ for $0\le m \le 2$ by normalization.\\
Since the set $\{h_0, h_{1,i}, h_2\;|\;i=1,\dots,N\}$ makes an orthonormal family in $L^2_\rho(|y|<1)$ and $\q L$ is compact, we can complete this set by a sequence of eigenfunctions $(h_n)_{n\ge 3}$ such that the set \eqref{basis} makes an orthonormal family in $L^2_\rho(|y|<1)$. We will denote by $\gamma_n$ the eigenvalue corresponding to $h_n$.\\
(ii)  Note that \eqref{gng2} follows from \eqref{coercivity} since $h_n$ satisfies \eqref{sania} for all $n\ge 3$, from the orthogonality of the basis \eqref{basis}. It remains then to prove \eqref{coercivity}.\\
 Consider $u\in L^2_\rho$ such that $\q L u \in L^2_\rho$ and \eqref{sania} holds. Using polar coordinates for $u(y)$:
\[
u(y)=U(r,\omega),\mbox{ where }y=r\omega,\;\;
r=|y|\in(0,1)\mbox{ and }\omega \in \m S^{N-1},
\]
we write
\begin{align}
-\iint u \q L u \rho dy &=\int_0^1 \int_{\m S^{N-1}}\left((\partial_r U(r,\omega))^2(1-r^2)+\frac 1{r^2} |\nabla_\omega U(r,\omega)|^2\right)\rho(r) r^{N-1} dr d\omega,\label{ulu}\\
\iint u^2 \rho dy&=
\int_0^1 \int_{\m S^{N-1}}(U(r,\omega))^2 \rho(r) r^{N-1} dr d\omega.\nonumber
\end{align}
Considering $\Delta_\omega$, the Laplace-Beltrami operator on the sphere $\m S^{N-1}$, we know (see for example \cite{BGMlnm71}) that its eigenvalues are given by
\[
\lambda_k = -k(k+N-2)\mbox{ for all }k\in \m N,
\]
with associated eigenspace
\[
H_k = \mathop{span}\{\phi_{k,i}\;|\;i=1,\dots,m_k\}
\]
made of the trace of harmonic homogeneous polynomials of degree $k$. Note that the family $(\phi_{k,i})$ is orthogonal and spans the space $L^2(\m S^{N-1})$. Note that the family $(\nabla_\omega\phi_{k,i})$ is orthogonal too.
We also have
\[
\begin{array}{llll}
&\lambda_0=0,&m_0=1,&\phi_{0,1}(\omega)=1,\\
&\lambda_1=N-1,&m_1=N,&\phi_{1,i}(\omega)=\omega_i \mbox{ (trace of }y_i\mbox{)}.
\end{array}
\]
Introducing the decomposition
\begin{equation}\label{decomposition}
U(r,\omega) = \sum_{k\ge 0} \sum_{i=1}^{m_k}U_{k,i}(r)\phi_{k,i}(\omega)
\end{equation}
and using orthogonality, we see that our goal reduces to proving that if
\begin{equation}\label{condorth}
\int_0^1 U_{k,i}(r) \rho(r)r^{N-1}dr=0\mbox{ for }k=0\mbox{ and }i=1,\dots,m_k,
\end{equation}
then, for all $k\in \m N$ and $i=1,\dots,m_k$,
\begin{align*}
\int_0^1 \left((1-r^2)\left(U_{k,i}'(r)\right)^2- \frac {\lambda_k}{r^2}\left(U_{k,i}(r)\right)^2 \right)\rho(r) r^{N-1} dr
 \ge 
|\gamma_2|\int_0^1 \left(U_{k,i}(r)\right)^2\rho(r) r^{N-1} dr
\end{align*}
where $\gamma_2$ is given in (i).\\
From the monotonicity of $\lambda_k$, the question reduces to the cases $k=0$ and $k=2$ only. In other words, we need to prove that:\\
- ($k=0$) if the left-hand side in \eqref{0deux} is finite and
\begin{equation}\label{0un}
\int_0^1 v(r) \rho(r)r^{N-1}dr=0,
\end{equation}
then
\begin{equation}\label{0deux}
\int_0^1 (1-r^2)\left(v'(r)\right)^2\rho(r) r^{N-1} dr
 \ge |\gamma_2|\int_0^1 \left(v(r)\right)^2\rho(r) r^{N-1} dr; 
\end{equation}
- ($k=2$)  if the left-hand side in \eqref{2deux} is finite, then
\begin{equation}\label{2deux}
\int_0^1 \left((1-r^2)\left(v'(r)\right)^2+\frac {2N}{r^2}\left(v(r)\right)^2 \right)\rho(r) r^{N-1} dr
 \ge |\gamma_2|\int_0^1 \left(v(r)\right)^2\rho(r) r^{N-1} dr.
\end{equation}
The proof follows from the following:
\begin{lem}\label{lemfirsta} We have 
\begin{equation}\label{firsta}
-\bar\lambda^{(1)}_{|y|<a}=-\bar\lambda^{(1)}_{a<|y|<1}=-\gamma_2=\frac{2(3p+1)}{p-1}\mbox{ and }-\bar\lambda^{(2)}_{|y|<1}\ge \frac{2(3p+1)}{p-1}
\end{equation}
where $\bar \lambda^{(i)}_{\Omega}$ is the $i$-th eigenvalue of $\q L$ restricted to radial functions of $L^2_\rho(\Omega)$ with homogeneous Dirichlet boundary conditions on $\partial\Omega \backslash \{|y|=1\}$.
\end{lem}
Let us use this Lemma to finish the proof of (ii). Then, we will prove it.

\medskip

{\bf Case $k=0$}: Consider $v$ such that the left-hand side in \eqref{0deux} is finite and \eqref{0un} holds. From \eqref{0un}, we see that $v$ is orthogonal to the first eigenfunction $1$, hence, since $V$ is radial, introducing $V(y) = v(|y|)$ for all $|y|<1$, we see from \eqref{firsta} that
\[
-\int_{|y|<1} V \q L V \rho(y) dy \ge |\bar\lambda^{(2)}_{|y|<1}|\int_{|y|<1} V^2 \rho(y) dy\ge \frac{2(3p+1)}{p-1}\int_{|y|<1} V^2 \rho(y) dy.
\]
Returning back to radial coordinates, we get \eqref{0deux}.

\medskip

{\bf Case $k=2$}: It is a direct consequence of the case $k=0$. Indeed, consider $v$ such that the left-hand side in \eqref{2deux} is finite. Introducing
\[
\bar v(r) = v(r) - 
\frac{\int_0^1 v(r)\rho(r) r^{N-1}dr}{\int_0^1\rho(r) r^{N-1}dr},
\]
we see that $\bar v$ satisfies \eqref{0un} and that the first member of \eqref{0deux} is finite for $\bar v$, hence, \eqref{0deux} holds for $\bar v$. 
Since we have
\begin{align*}
\int_0^1 (\bar v(r))^2\rho(r) r^{N-1} dr&=
\int_0^1 (v(r))^2\rho(r) r^{N-1} dr-\frac{\left(\int_0^1 v(r)\rho(r) r^{N-1}dr\right)^2}{\int_0^1\rho(r) r^{N-1}dr},
\end{align*}
using the Cauchy-Schwartz inequality, we write
\begin{align*}
|\gamma_2| \int_0^1 (v(r))^2 \rho(r) r^{N-1}dr &\le \int_0^1 (v'(r))^2 (1-r^2)\rho(r) r^{N-1}dr+ |\gamma_2|\frac{(\int_0^1 v(r)\rho(r) r^{N-1}dr)^2}{\int_0^1\rho(r) r^{N-1}dr}\\
&\le \int_0^1 (v'(r))^2 (1-r^2)\rho(r) r^{N-1}+\beta\int_0^1 \left(\frac{v(r)}r\right)^2 \rho(r) r^{N-1}dr
\end{align*}
where
\begin{equation}\label{defbeta}
\beta=\frac{|\gamma_2| I_{N+1}}{I_{N-1}}=\frac{2(3p+1)I_{N+1}}{(p-1)I_{N-1}}
\mbox{ and }I_j= \int_0^1\rho(r) r^jdr.
\end{equation}
It is enough to prove that 
\begin{equation}\label{beta=}
\beta=2N
\end{equation}
in order to conclude. Let us prove \eqref{beta=} in the following.\\
Using integration by parts, we write
\begin{align*}
I_{N+1}=&-\frac 1{2(\alpha+1)}\int_0^1 (-2)(\alpha+1) r(1-r^2)^\alpha r^N dr
=\frac N{2(\alpha+1)}\int_0^1(1-r^2)^{\alpha+1}r^{N-1} dr \\
=&\frac N{2(\alpha+1)}\left(I_{N-1}-I_{N+1}\right)
\end{align*}
Therefore, using the definition \eqref{defro} of $\alpha$, we see that
\[
\frac{I_{N+1}}{I_{N-1}}=\frac 1{1+\frac{2(\alpha+1)}N}=\frac {N(p-1)}{3p+1}
\]
and \eqref{beta=} follows from \eqref{defbeta}. Thus, \eqref{2deux} holds.
It remains to prove Lemma \ref{lemfirsta} in order to conclude the proof of Proposition \eqref{spectl}.

\medskip

\begin{proof}[Proof of Lemma \ref{lemfirsta}]
Since $\q L$ has a discrete spectrum on the one hand, and, on the other hand, $h_2$ defined in (i) of Proposition \ref{spectl} is positive on the ball $\{|y|<a\}$, negative on the annulus $\{a<|y|<1\}$  and zero on the sphere $\{|y|=a\}$, the first part in \eqref{firsta} follows from the standard elliptic theory.\\ 
For the second part, let us consider $V_2(y)=v_2(|y|)$ a radial eigenfunction of $\q L$ in $L^2_\rho(|y|<1)$ associated to the second eigenvalue $\bar \lambda^{(2)}_{|y|<1}$. Since $1$ is the first eigenvalue, we see from orthogonality that $\int_{|y|<1} V_2(y) \rho(y) dy=0$. Since $V_2$ is regular and radially symmetric, there exists $r_0\in (0,1)$ such that $V_2(y)=0$ whenever $|y|=r_0$. In particular, $V_2$ is an eigenfunction for $\q L$ in $L^2_\rho(|y|<r_0)$ and $L^2_\rho(r_0<|y|<1)$, with Dirichlet boundary conditions and associated eigenvalue $\bar\lambda^{(2)}_{|y|<1}$. Therefore, $\lambda^{(2)}_{|y|<1}$ is larger (in absolute value) than the first eigenvalue on $\{|y|<r_0\}$ and $\{r_0<|y|<1\}$, in the sense that
\[
|\bar\lambda^{(2)}_{|y|<1}|\ge |\bar\lambda^{(1)}_{|y|<r_0}|\mbox{ and }|\bar\lambda^{(2)}_{|y|<1}|\ge |\bar\lambda^{(1)}_{r_0<|y|<1}|.
\]
If $r_0\le a$ defined in \eqref{eigen}, using the monotonicity of the eigenvalues with respect to the domain and the first part of \eqref{firsta}, 
we write
\[
|\bar\lambda^{(1)}_{|y|<r_0}|\ge |\bar\lambda^{(1)}_{|y|<a}|=\frac{2(3p+1)}{p-1}
\]
and the second part in \eqref{firsta} follows.\\
If $r_0>a$, we similarly write
\[
|\bar\lambda^{(1)}_{r_0<|y|<1}|\ge |\bar\lambda^{(1)}_{a<|y|<1}|=\frac{2(3p+1)}{p-1}.
\]
and the second part in \eqref{firsta} follows. This concludes the proof of Lemma \ref{lemfirsta}.
\end{proof}
This concludes the proof of Proposition \ref{spectl} too.
\end{proof}

\section{Coercivity of the second derivative of the energy near $\kappa(d,y)$}\label{appcoer}
We give the proof of Proposition \ref{lemdefpos} here. The proof follows the same pattern as the one-dimensional case. The adaptation is only technical. For the sake of completeness, we give the details below. The reader interested only in the ideas may only read the one-dimensional case given in Proposition 4.7 page 90 in \cite{MZjfa07}.

\begin{proof}[Proof of Proposition \ref{lemdefpos}]

In the following, we reduce the proof of Proposition \ref{lemdefpos} to the proof of the nonnegativity of the following approximation of $\varphi_d$ defined for $\epsilon>0$ by:
\begin{align}
&\varphi_{d,\epsilon}\left(q, r\right)\nonumber\\
&=\varphi_d(q,r)-\epsilon\iint(\nabla q_1\cdot\nabla r_1 -(y\cdot \nabla q_1)(y\cdot \nabla r_1)+\frac{2(p+1)}{(p-1)^2}q_1r_1 +q_2 r_2)\rho dy\label{defvarphide}\\
&=\int_{|y|<1} q_1\left(-(1-\epsilon)\q L r_1+\left(-(1-\epsilon)\psi(d,y)-\epsilon\frac{2p(p+1)}{(p-1)^2}\frac{(1-|d|^2)}{(1+d\cdot y)^2}\right)r_1\right)\rho dy\label{goal}\\
&+(1-\epsilon)\int_{-1}^1q_2 r_2\rho dy.\nonumber
\end{align}
With the same proof as in one space dimension, we see that the following lemma directly implies Proposition \ref{lemdefpos}:
\begin{lem}[Reduction of the proof of Proposition \ref{lemdefpos}]\label{reduc}
There exists $\epsilon_0 \in (0,1)$ such that for all $|d|<1$ and $q_-\in \H^d_-$, $\varphi_{d,\epsilon_0}\left(q_-, q_-\right)\ge 0$ where $\varphi_{d, \epsilon_0}$ is defined in \eqref{defvarphide}.
\end{lem}
\medskip

 Let us now prove Lemma \ref{reduc}.

\medskip

\begin{proof}[Proof of Lemma \ref{reduc}] We proceed in 3 parts:

- In Part 1, we find a subspace of $\H$ of codimension $N+1$ where $\fe$ is nonnegative.

- In Part 2, we find a subspace of $\H$ of dimension $N+1$, where $\fe$ is negative and which is orthogonal to $\H^d_-$ with respect to $\fe$.

- In Part 3, we proceed by contradiction and prove that $\fe$ is nonnegative on $\H^d_-$.

\bigskip

{\bf Part 1: $\fe$ is nonnegative on a subspace of codimension $N+1$}

We claim the following:
\begin{lem}[$\fe$ is nonnegative on a subspace of codimension $N+1$] \label{lemfdepos}There exists $\epsilon_1>0$ such that for all $|d|<1$ and $\epsilon \in (0, \epsilon_1]$, $\fe$ is nonnegative on the subspace
\begin{equation}\label{defE}
E_2= \left\{q\in \H\;\mid\;\iint \q T_{-d}(q_1) \rho(y) dy = \iint \q T_{-d}(q_1) y_i\rho(y) dy = 0,\;\;\forall i=1\dots,N\right\}
\end{equation}
where $\q T_{-d}$ is defined in \eqref{defW}.
\end{lem}
\begin{proof} From the spectral properties of the operator $\q L$ given in Proposition \ref{spectl}, the proof is the same as in dimension 1. See Lemma 4.9 page 92 in \cite{MZjfa07} for the one-dimensional statement.
\end{proof}

\bigskip

{\bf Part 2: $\fe$ is negative on a $(N+1)$-dimensional subspace orthogonal to $\H^d_-$}

We need to find a set of linearly independent vectors $\{\bar W^{d,\epsilon}_i\in \q H\;|\; i=0\dots,N\}$ such that $\fe(\bar W^{d,\epsilon}_i, r)=0$ for any $r\in \H^d_-$. 
Since we know from the definition of $\H^d_-$ \eqref{defhd-} that 
\[
\forall r\in \H^d_-,\;\; 
\phi(W_i(d),r)=\pi^d_i(r)=0\mbox{ for all }i=0,\dots,N,
\]
 a convenient way to conclude is to find $\bar W^{d,\epsilon}_i$ such that
\begin{equation}\label{lina}
\forall q\in \H, \;\;
\phi(W_i(d),q) = \fe(\bar W^{d,\epsilon}_i, q)\mbox{ for all }i=0,\dots,N. 
\end{equation}
Then, we will show that $\fe$ is negative on the subspace spanned by $\bar W^{d,\epsilon}_i$ for all $i=0,\dots,N$. Consider $\epsilon>0$ to be fixed small enough and take $|d|<1$. We claim the following:
\begin{lem}\label{lemexpress}
There exists $\epsilon_2>0$ such that for all $\epsilon\in(0, \epsilon_2]$ and $|d|<1$:\\
(i) There exist continuous functions $\bar W^{d,\epsilon}_i$ for all $i=0,\dots,N$
such that \eqref{lina} holds. \\
(ii) Moreover, it holds that for $\epsilon$ small enough and for all $|d|<1$ and $i=1,\dots,N$,
\begin{equation*}
\left\|\bar W_0^{d,\epsilon}(y) - \vc{-W_{0,2}(d,y)}{W_{0,2}(d,y)}
-\alpha_1(d) F_1(d,y)
\right\|_{\H}+\left\|\epsilon \bar W^{d,\epsilon}_i(y)+\bar \alpha_2 F_i(d,y)\right\|_{\H}\le C\epsilon
\end{equation*}
where $\alpha_1(d)$ is continuous, $\bar \alpha_2=\frac{\hu (p-1)^2}{2p(p+1)}$ and $\bar c_1$ is the normalizing constant introduced in \eqref{eigen}.\\
(iii) The bilinear form $\varphi_{d, \epsilon}$ is negative on the subspace of $\H$ spanned by $\bar W^{d,\epsilon}_i$ for $i=0,\dots,N$. 
\end{lem}
\begin{proof}[Proof of Lemma \ref{lemexpress}]
We proceed in 3 steps:\\
- In Step 1, we find a PDE satisfied by $\bar W^{d, \epsilon}_i$, 
then we transform it with the Lorentz transform in similarity variables defined in \eqref{defW}.\\
- In Step 2, we solve the transformed PDE and find the asymptotic behavior of $\bar W^{d, \epsilon}_i$ 
as $\epsilon \to 0$, uniformly in $|d|<1$, which gives (i) and (ii).\\
- In Step 3, we use that asymptotic behavior to show that $\fe$ is negative on the subspace spanned by $\bar W^{d, \epsilon}_i$ for all $i=0,\dots,N$, which gives (iii).

\bigskip

{\bf Step 1: Reduction to the solution of some PDE}


From the definitions of $\fe$ \eqref{goal} and $\phi$ \eqref{defphi2},
we see that in order to satisfy \eqref{lina}, it is enough to take 
\begin{equation}\label{defVl2}
\bar W^{d,\epsilon}_{i,2}= W_{i,2}(d)/(1-\epsilon)
\end{equation}
then to prove the existence of $\bar V_1=\bar W^{d,\epsilon}_{i,1}$ solution to 
\begin{equation}\label{eqVi1}
-(1-\epsilon)\q L \bar V_1+\left(-(1-\epsilon)\psi(d,y)-\epsilon\frac{2p(p+1)}{(p-1)^2}\frac{(1-|d|^2)}{(1+d\cdot y)^2}\right)\bar V_1
=-\q L V_1+V_1
\end{equation}
where $V_1=W_{i,1}(d)$.\\
In the following, we use the Lorentz transform \eqref{defW} and transform this equation to make it ready to solve using the spectral properties of $\q L$ stated in Proposition \ref{spectl}. More precisely, we have the following:
\begin{cl}[Reduction to an explicitly solvable PDE]\label{step1}$ $\\
(i) $\bar V_1$ is a solution to \eqref{eqVi1} if and only if $\bar v_1\equiv \q T_{-d}(\bar V_1)$ defined in \eqref{defW} is a solution to the equation
\begin{equation}\label{eqvl1t}
 (1-\epsilon) \q L \bar v_1 + \left( -\gamma_1 + \frac{2(p+1)}{(p-1)^2}\epsilon\right) \bar v_1 =\bar f\equiv \frac{1-|d|^2}{(1-d\cdot z)^2}\q T_{-d}\left(\q L V_1- V_1\right)
\end{equation}
and $\gamma_1 = - \frac{2(p+1)}{p-1}$.\\
(ii) It holds that $\bar f\in \q H_0'$ and for some $C_0>0$, we have
\[
\forall |d|<1,\;\;\|\bar f\|_{\q H_0'}\le C_0 \|V_1\|_{\q H_0}.
\]
\end{cl}
\begin{nb} We define $\H_0'$, the dual space of $\H_0$ by means of the $L^2_\rho$ inner product, as the set of all $f\in L^2_\rho$ such that the linear form $h\to \iint f h \rho$ is continuous on $\H_0$. In particular, $\|f\|_{\H_0'} = \sup_{\|h\|_{\H_0}<1}|\int f(y)h(y)\rho(y) dy|$. 
 Using the orthonormal basis of $\q L$ in $L^2_\rho$ given in Proposition \ref{spectl} below, we see that
 \begin{equation}\label{fdecom}
 \|f\|_{\H_0'}^2 = f_0^2+\sum_{i=1}^N\frac {f_{1,i}^2}{1+|\gamma_1|}+\sum_{n \ge 2} \frac {f_n^2}{1+|\gamma_n|}\mbox{ if }f=\hz f_0+\hu \sum_{i=1}^N f_{1,i}y\cdot e_i+\sum_{n=2}^\infty f_n h_n(y)\in \H_0'.
 \end{equation}
Note that when $n=1$, we make a rotation of coordinates in order to work in the basis $(e_1(d), \dots, e_N(d))$, which is more suitable, since it already appears in the eigenfunctions of $L_d$ and $L_d^*$ given in Lemmas \ref{l10} and \ref{eigenld*}.
\end{nb}
\begin{proof}[Proof of Claim \ref{step1}] The proof is omitted since it is straightforward from the properties of the transformation $\q T_d$ given in Lemma \ref{0lemeffect}.  For the one-dimensional case, see Claim 4.11 page 94 in \cite{MZjfa07}.
\end{proof}

\bigskip

{\bf Step 2: Solution of equation \eqref{eqvl1t} and asymptotic behavior as $\epsilon \to 0$}

We prove items (i) and (ii) of Lemma \ref{lemexpress} in this step.

\begin{proof}[Proof of items (i) and (ii) of Lemma \ref{lemexpress}]$ $\\
%
(i) We have  the following claim which follows directly from Proposition \ref{spectl}:
\begin{cl}[Solution of equation \eqref{eqvl1t}]\label{zerai}
Consider $\bar f\in \H_0'$ decomposed as in \eqref{fdecom}.
Then, for any $\epsilon \in (0,\frac 12)$, equation \eqref{eqvl1t} 
has a unique solution $\bar v_1 \in \H_0$ given by
\begin{equation}\label{defv}
\bar v_1 = \frac{\hz \bar f_0}{- \gamma_1 + \epsilon(\frac{2(p+1)}{(p-1)^2})}
+\frac {\hu (p-1)^2}{2p(p+1)\epsilon}\sum_{i=1}^N\bar f_{1,i}y\cdot e_i(d)
+\sum_{n=2}^\infty \frac{\bar f_n}{\gamma_n - \gamma_1 + \epsilon(\frac{2(p+1)}{(p-1)^2}-\gamma_n)}h_n
\end{equation}
where $\gamma_n\le 0$ are the eigenvalues of $\q L$ introduced in Proposition \ref{spectl}. 
\end{cl}
\begin{proof} This claim directly follows from the spectral decomposition of $\q L$ in Proposition \ref{spectl} above. 
\end{proof}
Since $W_{i,1}(d)\in \q H_0$ from Lemma \ref{eigenld*}, using this claim together with the beginning of Step 1 and Claim \ref{step1},
we get the conclusion of item (i) of Lemma \ref{lemexpress} (note that the continuity follows from standard elliptic theory).



\medskip

\noindent (ii) 
{\bf Case 1: $i=0$}. From Step 1, Claim \ref{step1} and Claim \ref{zerai}, we have $\q T_{-d}(\bar W_0^{d,\epsilon})=\bar v_1$ given by \eqref{defv}, provided that $V_1=W_{0,1}(d)$ in \eqref{eqVi1} and $\bar f = \frac{1-|d|^2}{(1-d\cdot z)^2}\q T_{-d}\left(\q L W_{0,1}(d)- W_{0,1}(d)\right)$ in \eqref{eqvl1t}. Using item (ii) in Claim \ref{step1} and item (iii) in Lemma \ref{eigenld*}, we see that
\begin{equation}\label{bound1}
\forall |d|<1,\;\;\|\bar f\|_{\q H_0'} \le C_0.
\end{equation}
Using the decomposition along the orthonormal basis \eqref{fdecom} and item (iii) in Lemma \ref{0lemeffect}, we compute for $i=1,\dots,N$ (note that $\hu$ is a normalizing constant introduced in \eqref{eigen}):
%
\begin{align*}
\frac 1{\hu }\bar f_{1,i}&=\int_{|z|<1} \bar f(z)(z\cdot e_i(d)) \rho(z) dz\\
&=\int_{|y|<1} \left( \q L W_{0,1}(d,y) -W_{0,1}(d,y)\right) \q T_d(z\cdot e_i(d)) \rho(y) dy\\
&=\phi(W_0(d), (\q T_d(z\cdot e_i(d)),0))=\phi(W_0(d), F_i(d))=0
\end{align*}
where we used 
\eqref{defgdi} and the duality relation between eigenfunctions of the operator $L_d$ and its conjugate, stated in item (ii) of Lemma \ref{eigenld*}.
%
Therefore, we see from Claim \ref{zerai} and 
\eqref{bound1}
that for $\epsilon$ small enough,
\begin{equation*}
\sup_{|d|<1}\left\|\q T_{-d}(\bar W_0^{d,\epsilon})- v^*\right\|_{\H_0}\le C\epsilon\|\bar f\|_{\H_0'} \le C_0\epsilon
\mbox{ where }
v^*(z) = \sum_{n\neq 1} \frac{\bar f_n}{\gamma_n - \gamma_1}h_n(z)
\end{equation*}
is the unique solution of 
\begin{equation*}
\q L v(z)-\gamma_1 v(z) = \bar f(z)\mbox{ with } \iint v(z) (z\cdot e_i(d)) \rho(z) dz =0
\mbox{ for all }i=1,\dots,N.
\end{equation*}
Therefore, we see from item (iv) in Lemma \ref{0lemeffect} that for $\epsilon$ small enough,
\begin{equation}\label{sonia}
\sup_{|d|<1}\left\|\bar W^{d,\epsilon}_{0,1}- V^*\right\|_{\H_0}\le C_0\epsilon,
\end{equation}
where $V^*=\q T_d v^*$ is the unique solution of 
\begin{align}
&\q L V(y) +\psi(d,y) V(y)= \q L V_1(d,y)- V_1(d,y)\nonumber\\
\mbox{with}&
\iint V(y) F_{i,1}(d,y) \frac{\rho(y)}{(1+d\cdot y)^2} dy =0
\mbox{ for all }i=1,\dots,N\label{orth}
\end{align}
(note that this equation is the version of \eqref{eqVi1} with $\epsilon=0$ and use item (iii) in Lemma \ref{0lemeffect}
and \eqref{defgdi}
to get the orthogonality condition).
Since 
\[
-\q L W_{0,1}(d)+ W_{0,1}(d)= \q L W_{0,2}(d)+\psi(d,y) W_{0,2}(d)\mbox{ and }
\q L F_{i,1}(d)+\psi(d,y) F_{i,1}(d)=0
\]
(use the fact that $L_d^*\left(W_0(d)\right) = W_0(d)$  and $L_d(F_{i,1}(d))=0$ from Lemmas \ref{eigenld*} and \ref{l10}),
we see from uniqueness that
 $V^*(y) = -W_{0,2}(d,y)+ \sum_{j=1}^N \alpha_j(d)F_{j,1}(d,y)$ where the $\alpha_j(d)$ are determined by \eqref{orth}, which gives for all $i=1,\dots,N$,
\[
-\iint W_{0,2}(d,y)\frac{F_{i,1}(d,y)\rho(y)}{(1+d\cdot y)^2} dy +\sum_{j=1}^N \alpha_j(d)\iint F_{j,1}(d,y)\frac{F_{i,1}(d,y)\rho(y)}{(1+d\cdot y)^2} dy=0.
\]
Since we have $W_{0,2}(d)=c_1\q T_d(\frac{1-|z|^2}{1-d\cdot z})$ from \eqref{defWl2} and \eqref{defW} (where $c_1$ is a normalizing constant introduced in \eqref{defWl2}),
and $F_{i,1}(d) = \q T_d(z\cdot e_i(d))$ from \eqref{defgdi}, using item (iii) in Lemma \ref{0lemeffect} and separation of variables to compute the integrals (remember that $e_1=\frac d{|d|}$), we write
\begin{align*}
\iint W_{0,2}(d,y)\frac{F_{i,1}(d,y)\rho(y)}{(1+d\cdot y)^2} dy&=\frac{c_1}{1-|d|^2}\int_{|z|<1}
\frac{1-|z|^2}{1-d\cdot z}(z\cdot e_i) \rho(z) dz=0\mbox{ if }i\ge 2,\\
\iint F_{j,1}(d,y)\frac{F_{i,1}(d,y)\rho(y)}{(1+d\cdot y)^2} dy&=\frac{\displaystyle\int_{|z|<1}(z\cdot e_i)(z\cdot e_j)\rho(z) dz}{1-|d|^2} = \delta_{i,j}\frac{\displaystyle\int_{|z|<1}z_1^2\rho(z) dz}{1-|d|^2}.
\end{align*}
Therefore, $\alpha_j(d)=0$ if $j\ge 2$ and $\alpha_1(d)$ is continuous in terms of $d$. Thus, the first identity in item (ii) of Lemma \ref{lemexpress} follows from \eqref{sonia} and \eqref{defVl2}.\\ 
%
% ...
%
{\bf Case 2: $i\ge 1$}. From Step 1 and Claims \ref{step1} and \ref{zerai}, we have $\q T_{-d}(\bar W^{d,\epsilon}_{i,1})=\bar v_1$ given by \eqref{defv}, provided that $V_1=W_{i,1}(d)$ in \eqref{eqVi1} and $\bar f = \frac{1-|d|^2}{(1-d\cdot z)^2}\q T_{-d}\left(\q L W_{i,1}(d)- W_{i,1}(d)\right)$ in \eqref{eqvl1t}. Using item (ii) in Claim \ref{step1} and item (iii) in Lemma \ref{eigenld*}, we see that \eqref{bound1} holds in this case too.
Since $F_j(d)=(\q T_d(z\cdot e_j(d),0))$ from \eqref{defgdi}, we apply orthogonality to the decomposition \eqref{fdecom}, and use item (iii) in Lemma \ref{0lemeffect} together with the orthogonality result in item (ii) of Lemma \ref{eigenld*} to write: 
\begin{align*}
\bar f_{1,j}
&=\int_{|z|<1} \bar f(z)(z\cdot e_j(d)) \rho(z) dz
=\iint \left(\q L W_{i,1}(d,y) - W_{i,1}(d,y)\right)F_j(d,y) \rho(y) dy\\
&=-\phi(W_i(d),F_j(d))=-\delta_{i,j}.
\end{align*}
Therefore, 
we see from Claim \ref{zerai} and 
\eqref{bound1}
that for $\epsilon$ small enough,
\begin{equation}\label{besma3a}
\left\|\q T_{-d}(\bar W^{d,\epsilon}_{i,1})+\frac{\hu (p-1)^2}{2p(p+1)\epsilon} z\cdot e_i(d)
\right\|_{\H_0}\le C\|\bar f\|_{\H_0'}\le C_0.
\end{equation}
Since the estimate for $\bar W_{i,2}^{d,\epsilon}$ follows from \eqref{defVl2} and item (iii) of Lemma \ref{eigenld*}, we see that item (ii) of Lemma \ref{lemexpress} follows from \eqref{besma3a}, item (iv) of Lemma \ref{0lemeffect} and \eqref{defgdi}.
This closes the proof of items (i) and (ii) in Lemma \ref{lemexpress}.
\end{proof}

\bigskip

{\bf Step 3: Sign of $\varphi_{d,\epsilon}$ on $\spa\{\bar W^{d, \epsilon}_i,\;\;i=0\dots,N\}$}

We give the proof of item (iii) in Lemma \ref{lemexpress}.

\begin{proof}[Proof of item (iii) in Lemma \ref{lemexpress}] 
We finish the proof of Lemma \ref{lemexpress} here, by proving that $\fe$ is negative on the subspace of $\H$ spanned by $\bar W^{d, \epsilon}_i$ for $i=0\dots,N$. 
Introducing the $(N+1)\times (N+1)$ symmetric matrix $M=(M_{i,j})_{0\le i,j\le N}$ defined for $i,j=0,\dots,N$ by
\[
M_{i,j}=\varphi_{d,\epsilon}(\bar W^{d,\epsilon}_i, \bar W^{d,\epsilon}_j),
\]
 we see that it is enough to find $\epsilon_4$ such that for all $0<\epsilon\le\epsilon_4$, $|d|<1$ and $k=0,\dots,N$,
\begin{equation}\label{fdeV0V0}
(-1)^{k+1} \det(M_{i,j})_{0\le i,j\le k}>0.
\end{equation}
We claim that it is enough to prove that for some constant $\bar \alpha_3>0$, for $\epsilon$ small enough and for all $|d|<1$:
\begin{equation}\label{mij}
|M_{0,0}+\bar \alpha_3|\le C\epsilon\mbox{ and }\sup_{1\le i\le N}|M_{i,i}+\frac{\bar \alpha_2}\epsilon|+\sup_{i\neq j}|M_{i,j}|\le C.
\end{equation}
Indeed, if \eqref{mij} holds, then we have $M_{0,0}<0$ and for all $k=1,\dots,N$,
\begin{align*}
\sup_{|d|<1}|\det(M_{i,j})_{0\le i,j\le k}-(-1)^{k+1}\bar \alpha_3\left(
\frac{\bar \alpha_2}\epsilon
\right)^k|\le \frac C{\epsilon^{k-1}}
\end{align*}
and \eqref{fdeV0V0} follows. It remains to prove \eqref{mij} in order to conclude.
%
%
%
In the following, we will estimate the coefficients $M_{i,j}$ of the matrix $M$
 as $\epsilon\to 0^+$, uniformly for $|d|<1$, using the asymptotic behavior of $\bar W^{d,\epsilon}_i$ given in item (ii) of Lemma \ref{lemexpress}.
\\
- First, using \eqref{lina}, we see that
\begin{equation}\label{asensi}
\forall i,j=0,\dots,N,\;\;\varphi_{d,\epsilon}(\bar W_i^{d,\epsilon}, \bar W_j^{d,\epsilon})= \phi (W_i(d), \bar W_j^{d,\epsilon}).
\end{equation}
- Then, taking $i=0,\dots,N$ and $j=1,\dots,N$ with $j\neq i$ in \eqref{asensi}, we see from item (ii) in Lemma \ref{lemexpress} and the orthogonality condition in Lemma \ref{eigenld*} that
\begin{equation*}
\sup_{|d|<1}
\left|\varphi_{d, \epsilon}(\bar W^{d,\epsilon}_i, \bar W^{d,\epsilon}_j) \right|\le C.
\end{equation*}
- Now, taking $i=j\ge 1$ in \eqref{asensi} and using Lemmas \ref{lemexpress} and \ref{eigenld*}, we see that
\begin{equation*}
\sup_{|d|\le d_0}
\left|\varphi_{d, \epsilon}(\bar W^{d,\epsilon}_i, \bar W^{d,\epsilon}_i) +\frac{\bar \alpha_2}{\epsilon}\right|\le C.
\end{equation*}
- Finally, proceeding similarly when $i=j=0$ in \eqref{asensi}, we see that 
\begin{equation*}
\sup_{|d|<1}
\left|\varphi_{d, \epsilon}(\bar W_0^{d,\epsilon}, \bar W_0^{d,\epsilon})-\phi\left(W_0(d), \vc{-W_{0,2}(d)}{W_{0,2}(d)}\right)\right|\le C \epsilon.
\end{equation*}
Using \eqref{defphi2} together with \eqref{eqWl1}, we write
\begin{eqnarray}
&&\phi\left(W_0(d), \vc{-W_{0,2}(d)}{W_{0,2}(d)}\right)\nonumber\\
&=&\iint W_{0,2}(d,y) \left(\q L W_{0,1}(d,y)-W_{0,1}(d,y)+W_{0,2}(d,y)\right)\rho(y) dy\nonumber\\
&=& \iint W_{0,2}(d,y) \left( \frac{p+3}{p-1}W_2(d,y) +2 y \cdot\nabla {W_{0,2}(d)}(y)- 4\alpha\frac{W_{0,2}(d,y)}{1-|y|^2}\right)\rho(y) dy\nonumber\\
&=& \iint W_{0,2}(d,y)^2\left(\frac{p+3}{p-1}-\frac {2\alpha}{1-|y|^2}\right)\rho(y) dy -\iint W_{0,2}(d,y)^2 \dv (y\rho(y)) dy\nonumber\\
&=&-2\alpha \iint W_{0,2}(d,y)^2 \frac{\rho(y)}{1-|y|^2} dy\equiv \bar \alpha_3>0.\nonumber
\end{eqnarray}
Thus, \eqref{mij} holds and the bilinear form $\fe$ is negative on $\spa\{\bar W^{d, \epsilon}_i,\;\;i=0\dots,N\}$. This gives the conclusion of item (iii) in Lemma \ref{lemexpress}.
\end{proof}
This concludes the proof of Lemma \ref{lemexpress}.
\end{proof}

\bigskip

{\bf Part 3: End of the proof of Lemma \ref{reduc}}: 

From Lemmas \ref{lemfdepos} and \ref{lemexpress}, we define $\epsilon_0=\min(\epsilon_1, \epsilon_2)\in (0,1)$.
We will now prove by contradiction that $\varphi_{d, \epsilon_0}$ is negative on $\H^d_-$ for all $|d|<1$.

\medskip

From Lemma \ref{lemexpress} and \eqref{lina}, for all $|d|<1$ and $\epsilon \in (0, \epsilon_0]$, we write the definition of $\H^d_-$ \eqref{defhd-} as follows:
\begin{equation}\label{newdefhd-}
\H^d_- = \left\{r\in \H\;\;|\;\;
\varphi_{d, \epsilon}\left(\bar W^{d,\epsilon}_i,r\right)=0 \mbox{ for all }i=0\dots,N\right\}.
\end{equation}
We proceed by contradiction and assume that 
\begin{equation}\label{fder1}
\mbox{there is }r\in \H^d_-\mbox{ such that }\varphi_{d,\epsilon}\left(r,r\right)<0.
\end{equation}
Using \eqref{newdefhd-}, 
we see that $r\not \in \spa\left(\bar W^{d,\epsilon}_i,\;i=0,\dots,N\right)$. Since $\det M\neq 0$ from \eqref{fdeV0V0}, the dimension of the vector subspace
\[
E_1= \spa \left(r, \bar W^{d,\epsilon}_i\mbox{ for }i=0,\dots,N \right)
\]
is $N+2$. Hence, since the subspace $E_2$ \eqref{defE} is of codimension $N+1$, there exists a non zero $u \in E_1\cap E_2$.

\medskip

\noindent On the one hand, since $u \in E_2$, we have from Lemma \ref{lemfdepos} that
\begin{equation}\label{fdepos}
\fe\left(u, u\right)\ge 0.
\end{equation}
On the other hand, since $\varphi_{d, \epsilon}$ is negative on $E_1$ by (iii) of Lemma \ref{lemexpress}, we must have from \eqref{newdefhd-} 
and \eqref{fder1},  
\[
\fe\left(u, u\right)<0.
\]
This contradicts \eqref{fdepos}. Thus, \eqref{fder1} does not hold, and $\fe$ is nonnegative on $\H^d_-$. This concludes the proof of Lemma \ref{reduc}.
\end{proof}
Since Lemma \ref{reduc} implies Proposition \ref{lemdefpos} by the same argument as in one space dimension, this is also the conclusion of Proposition \ref{lemdefpos}.
\end{proof}

\section{A modulation technique}\label{appmod}
We give the proof of Lemma \ref{lemmod} here. As in one space dimension, the proof relies on the application of the implicit functions theorem around the family $\kappa^*(d,\nu)$ defined in \eqref{defk*}. Nevertheless, the proof in our case is much more delicate 
and becomes particularly intricate when $d\to 0$ or $|d|\to 1$. Indeed, if $d\to 0$, then the new degenerate directions of the linearized operator become singular, and when $|d|\to 1$, the columns of the jacobian matrix have different sizes, making the application of the implicit function theorem particularly delicate, since we crucially need to make explicit the dependence of our bounds in terms of $d$. 

\medskip

\begin{proof}[Proof of Lemma \ref{lemmod}]
The modulation works because the derivatives of $\kappa^*(d,\nu)$ with respect to $d$ and $\nu$ (when  $(d,\nu)=(d,0)$) are precisely the directions we aim at ``killing'' in \eqref{ortho} (see the proof of Lemma \ref{l10} given in page \pageref{proofl10}). Thus, by slightly changing the parameters $(d,\nu)$, we can eliminate the projections of $v$ along those directions. As a matter of fact, the proof of Lemma \ref{lemmod} starts with the computation of the projections of the derivatives of $\kappa^*(d,\nu)$: 
\begin{cl}[Projections of the derivatives of $\kappa^*(d,\nu)$] \label{cltech} 
For all $B\ge 2$, there exists $C_1(B)>0$ such that if 
\begin{equation}\label{condnuB}
|d|<1\mbox{ and }-1+\frac 1B\le \frac \nu{1-|d|}\le B,
\end{equation}
 then
\begin{equation}\label{old}
\begin{array}{rclrcl}
|\pi_0^{d^*}(\partial_{e_1} \kappa^*(d,\nu))|&\le &\frac{C_1}{1-{|d|}},&
-\frac {C_1}{1-{|d|}}\le \pi_0^{d^*}(\pnu \kappa^*(d,\nu))&\le& -\frac 1{C_1(1-{|d|})},\\
 \pi_1^{d^*}(\pnu \kappa^*(d,\nu))&=& 0,&-\frac {C_1}{1-{|d|}}\le \pi_1^{d^*}(\partial_{e_1} \kappa^*(d,\nu))&\le& -\frac 1{C_1(1-{|d|})},\\
\end{array}
\end{equation}
and for $i=2,\dots,N$, $j=0\dots,N$ with $j\neq i$, 
\begin{equation}\label{new}
\begin{array}{rclrcl}
\pi_i^{d^*}(\pnu \kappa^*(d,\nu))&=&0,&\pi_i^{d^*}(\partial_{e_1} \kappa^*(d,\nu))&=& 0,\\
\pi_j^{d^*}(\partial_{e_i}\kappa^*(d,\nu))&=& 0,&
-\frac {C_1}{\sqrt{1-{|d|}}}\le \pi_i^{d^*}(\partial_{e_i} \kappa^*(d,\nu))&\le& -\frac 1{C_1\sqrt{1-{|d|}}},
\end{array}
\end{equation}
where $d^*=\frac{d}{1+\nu}$.
\end{cl}
\begin{proof} The proof is left to Appendix \ref{secal}, a section dedicated to all the properties of $\kappa^*(d,\nu)$.
\end{proof}
As we said earlier, the proof follows from small adaptations in the proof of Lemma 2.1 in \cite{MZisol10} where we worked in one space dimension with a modulation around a decoupled sum of solitons. In one way, our case is easier, since we have only one soliton here. Of course, we are left with the difficulty of adapting the proof to higher dimensions. A major difficulty, is the singularity appearing in the derivatives of the eigenfunctions of $L_d^*$, as one can see from item (iii) in Lemma \ref{eigenld*}. For that reason, we consider $\dun\in (0, \frac 14)$ to be fixed some enough later, and consider the cases where $|\bar d|\ge \dun$, then, $|\bar d|\le \dun$.

\bigskip

{\bf Case 1: $|\bar d|\ge \dun$}.

In this section, we insist on the fact that $\dun$ is arbitrary in the interval $(0, \frac 14)$. Its value will be chosen small enough when handling Case 2, that is when $|\bar d|\le \dun$.
Define
\begin{equation}\label{defPsi}
\begin{array}{llll}
\Psi:&\H\times(-1,\infty)\times B(0,1)&\rightarrow&\m R^{N+1}\\
&(v,\nu,d)&\mapsto& (\pi_0^{d^*}(q),\dots,\pi_N^{d^*}(q))
\end{array}
\end{equation}
where $q=v-\kappa^*(d,\nu)$ and ${d^*}= \frac{d}{1+\nu}$. We recall from \eqref{defpdi} that
\[
\pi_i^{d^*}(q) =\phi(W_i(d^*),q)
\]
where $W_i(d^*)$ is defined in \eqref{defWl2}.\\
Given 
$\eun>0$ to be fixed later small enough in terms of $A$ and $\dun$, together with $v\in\H$ and $(\bar \nu,\bar d)\in(-1,\infty)\times B(0,1)$ satisfying \eqref{condnu}, the conclusion follows from the application of the implicit function theorem to $\Psi$ near the point
\begin{equation}\label{condnu1}
(\bar v,\bar \nu,\bar d)=\left(\kappa^*(\bar d,\bar \nu),\bar \nu, \bar d\right).
\end{equation}
Three facts need to be checked:\\
1- Note first that
\begin{equation*}
\Psi\left(\bar v,\bar \nu,\bar d\right)=0.
\end{equation*}
2- From \eqref{condnu}, we may consider $(v,d,\nu)$ in a neighborhood of $(\kappa^*(\bar d, \bar \nu),\bar d,\bar\nu)$ such that 
\begin{equation}\label{condnu2}
|d-\bar d|\le \frac\dun 2,\;\;
|\bar d|\ge \frac\dun 2,\;\;
\;\;-1+\frac 1{A+1}\le \frac\nu{1-|d|}\le A+1
\mbox{ and }
\|q\|_{\H}\le 2\eun,
\end{equation}
where $q=v-\kappa^*(d,\nu)$. Computing
\[
D_v\Psi(v,\nu,d)(u) =(\pi_0^{{d^*}}(u),\dots,\pi_N^{{d^*}}(u)),
\]
we see from the Cauchy-Schwartz inequality and the boundedness of $W_i(d)$ (stated in Lemma \ref{eigenld*}) that 
\[
\|D_v\Psi(v,\nu,d)\|\le C.
\]
3- Consider now $M$, the Jacobian matrix of $\Psi$ 
with respect to the variables $(\nu, d)$, computed in the basis $(1, e_1(d), \dots,e_N(d))$. We need to prove that $M$ is invertible, provided that $\eun>0$ is small enough. 
Since $e_i(d) = e_i(d^*)$ from Lemma \ref{lembase}, we will simplify the notation and write $e_i$ instead. Let us compute for all $i=0,\dots,N$ and
 $j=1,\dots,N$,
\begin{align}
\partial_\nu\pi_i^{d^*}(q) =&-\pi_i^{d^*}(\partial_\nu \kappa^*(d,\nu))+\phi(\partial_\nu d^*\cdot\nabla_d W_i(d^*),q),\label{Dv}\\
\partial_{e_j}\pi_i^{d^*}(q) =&-\pi_i^{d^*}(\partial_{e_j} \kappa^*(d,\nu))+\phi(\partial_{e_j}d^*\cdot\nabla_d W_i(d^*),q),\nonumber
\end{align}
in other words,
\[
M=M_0+M_R
\]
where $M_0$ is the value of $M$ when $q=0$, and $M_R$ is the perturbation involving $q$.\\
Using Claim \ref{cltech}, we see that $M_0$ is an upper triangular matrix with a non-zero diagonal, and possibly only one non-zero value outside the diagonal. In particular, $M_0$ is invertible. 
Since we see from \eqref{condnu2} that $M_0$ has its first and second column of size larger than $\frac 1{C_1(A+1)(1-|d^*|)}$ and the others of size $\frac 1{C_1(A+1)\sqrt{1-|d^*|}}$ where $C_1(A+1)$ appears in Claim \ref{cltech}, it is enough to prove that the columns of the perturbation $M_R$ are bounded by $\frac 1{10C_1(A+1)(1-|d^*|)}$
for the first and the second column, and the others columns are bounded by 
$\frac 1{10C_1(A+1)\sqrt{1-|d^*|}}$
in order to show that $M$ is invertible too.
Let us then handle the columns of $M_R$ in order to conclude.\\  
Since $\partial_\nu d^*=-\frac d{(1+\nu)^2}=-\frac {|d|}{(1+\nu)^2}e_1$, $\partial_{e_j} d^*=\frac{\partial_{e_j} d}{1+\nu}=\frac{e_j}{1+\nu}$ and 
$-1+\frac 1{2A}\le \nu$ 
(use \eqref{condnu2}), we see from \eqref{condnu1} and item (iii) of Lemma \ref{eigenld*} that 
\begin{align}
&|\phi(\partial_\nu d^*\cdot\nabla_d W_i(d^*),q)|=\frac {|d|}{(1+\nu)^2}|\phi(\partial_{e_1}W_i(d^*),q)|\le \frac {C|d|\|q\|_{\q H}}{(1+\nu)^2(1-|d^*|)}\le 
\frac{C(A)\eun}{1-|d^*|},\nonumber\\
& |\phi(\partial_{e_1} d^*\cdot\nabla_d W_i(d^*),q)|=\frac {1}{1+\nu}|\phi(\partial_{e_1}W_i(d^*),q)|\le \frac{C(A)\eun}{1-|d^*|},\label{antigua}\\
&|\phi(\partial_{e_j} d^*\cdot\nabla_d W_i(d^*),q)|=\frac {1}{1+\nu}|\phi(\partial_{e_j}W_i(d^*),q)|\le \frac {C(1+\nu)^{-1}\|q\|_{\q H}}{|d^*|\sqrt{1-|d^*|}}\le\frac{C(A)\eun}{\dun\sqrt{1-|d^*|}}.\nonumber
\end{align}
This is precisely the desired estimates for the columns of $M_R$,
provided that 
\begin{equation}\label{esmall}
\eun\le \bar \eun(A)\dun,
\end{equation}
for some small enough constant $\bar \eun(A)>0$. 
Thus, 
$M$ is invertible.\\ 
Let us just note that replacing the radial derivative ``$\partial_{e_1}$'' by the ``$\partial_\zeta$'', the derivative with respect to 
$\zeta=-\arg\tanh |d|$
already defined in \eqref{defzeta},
we see that $\partial_\zeta = (1-|d|^2) \partial_{e_1}$. Therefore, the second column in the (new) Jacobian becomes of order $1$. This means that we may work in a neighborhood of $(\kappa^*(\bar d, \bar \nu), \bar \nu,\bar d)$ where 
\begin{equation}\label{equiv0}
|\arg\tanh |d| - \arg\tanh |\bar d||
\le C(A)\|\bar q\|_{\q H} 
\le C(A)\eun,
\mbox{ hence }\frac 1{C(A)} \le \frac{1-|d|}{1-|\bar d|}\le C(A),
\end{equation}
where $\bar q = v-\kappa^*(\bar d,\bar \nu)$.\\
{\it Conclusion}: From items 1-, 2- and 3- above, we see that the implicit function theorem applies and given $v\in \H$, $\bar \nu$ and $\bar d$ satisfying \eqref{condnu} with $\eun$ 
satisfying \eqref{esmall}, 
we get the existence of other parameters $d(v)$ and $\nu(v)$ such that \eqref{ortho} holds. 
Since we know from item 3- above that the two first lines of the Jacobian are of order $(1-|d|)^{-1}$ and the others are of order $(1-|d|)^{-1/2}$, which is smaller, using \eqref{equiv0}, we deduce that
\begin{equation}\label{mimi}
\frac{|\nu-\bar \nu|}{1-|\bar d|}+\frac{|d-\bar d|}{\sqrt{1-|\bar d|}}\le C(A)\|\bar q\|_{\q H}.
\end{equation}
In the following, we may take the constant $\bar \eun(A)$ introduced in \eqref{esmall}  even smaller.
Using \eqref{condnu} and the fact that $|\bar d|\ge \delta_1$, we see that
\begin{equation}\label{majdi}
|d-\bar d|\le C(A)\eun\sqrt{1-|\bar d|}\le \frac{|\bar d|}2,\mbox{ hence }\frac{|\bar d|}2 \le |d|\le \frac 32 |\bar d|. 
\end{equation}
Using \eqref{equiv0},
we see that
\begin{equation}\label{omkalsoum}
\frac{||\bar d|-|d||}{1-|\bar d|}\le C|\arg\tanh |d| - \arg\tanh |\bar d||
\le C(A)\|\bar q\|_{\q H},
\end{equation}
hence, from \eqref{mimi}, \eqref{condnu2} and \eqref{equiv0}, it follows that
\begin{equation}\label{prea3}
\left|\frac \nu{1-|d|} - \frac {\bar\nu}{1-|\bar d|}\right|\le \frac{|\nu-\bar \nu|}{1-|\bar d|}+\frac{|\nu|}{1-|\bar d|}\frac{||\bar d|-|d||}{1-|d|}\le C(A)\|\bar q\|_{\q H}.
\end{equation}
Using \eqref{mimi}, \eqref{omkalsoum} and \eqref{condnu}, we compute from \eqref{condnu} and \eqref{majdi}
\begin{align}
&2|\bar d||l_{\bar d}(d-\bar d)|=2|\bar d\cdot(d-\bar d)|
=|(d+\bar d)\cdot (d-\bar d)-|d-\bar d|^2|\nonumber\\
&\le ||d|^2-|\bar d|^2|+|d-\bar d|^2
\le C(A)\|\bar q\|_{\q H}(|d|+|\bar d|)(1-|\bar d|)+C(A)\eun\frac{|\bar d|}{\dun}\|\bar q\|_{\q H}(1-|\bar d|)\nonumber\\
&\le C(A)|\bar d|\|\bar q\|_{\q H}(1-|\bar d|),\label{radial}
\end{align}
where the projection $l_{\bar d}$ has been introduced right before \eqref{yortho}. 
Using \eqref{condnu2} and the fact that $|\bot_{\bar d}(d-\bar d)|\le |d-\bar d|$, 
where the projection $\bot_{\bar d}$ has also been introduced right before \eqref{yortho},
we see from \eqref{mimi}, \eqref{radial} and \eqref{condnu} that
\begin{equation}\label{arga}
\frac{|\nu-\bar \nu|}{1-|\bar d|}+\frac{|l_{\bar d}(d-\bar d)|}{1-|\bar d|}
+\frac{|\bot_{\bar d}(d-\bar d)|}{\sqrt{1-|\bar d|}}\le C(A)\|\bar q\|_{\q H}\le C(A)\eun.
\end{equation}
Therefore, we see from \eqref{condnu2} that Claim \ref{propk*} applies and using \eqref{prea3}, \eqref{equiv0} and \eqref{arga}, we get
\[
\|\kappa^*(\bar d, \bar \nu) - \kappa^*(d,\nu)\|_{\q H} \le C(A)\|\bar q\|_{\q H}=C(A)\|v-\kappa^*(\bar d, \bar \nu)\|_{\q H}.
\]
Thus, 
\[
\|v-\kappa^*(d,\nu)\|_{\q H}\le \|v-\kappa^*(\bar d, \bar \nu)\|_{\q H}+\|\kappa^*(\bar d, \bar \nu) - \kappa^*(d,\nu)\|_{\q H} \le C(A)\|v-\kappa^*(\bar d, \bar \nu)\|_{\q H}.
\]
Using \eqref{equiv0}, \eqref{mimi} and \eqref{prea3},
this concludes the proof of Lemma \ref{lemmod}, in the case where $|\bar d|\ge \dun$, provided that $\eun\le \bar \eun(A) \dun$, for some small enough constant $\bar \eun(A)>0$. We recall that here, the constant $\dun$ is arbitrary in the interval $(0, \frac 14)$.

\bigskip

{\bf Case 2: $|\bar d|\le \dun$}.

 We will show that the conclusion of Lemma \ref{lemmod} holds, provided that $\dun>0$ is chosen small enough. In order to avoid the singularity $\frac 1{|d^*|}$ appearing in \eqref{antigua}, itself coming from the bound on the norm of $\|\partial_{e_j} W_i(d)\|_{\q H'}$ proved in item (iii) of Lemma \ref{eigenld*}, we will choose another basis for the eigenspace of $L_d^*$ corresponding to $\lambda=0$, though keeping the only eigenfunction for $\lambda=1$. More precisely, we introduce 
$\bar W_i(d)$ defined by 
\begin{equation}\label{defwb}
\bar W_0(d,y) =W_0(d,y)\mbox{ and for all }i=1,\dots,N,\;\;\bar W_{i,2}(d,y)=c_0\ds\frac{y_i+d_i}{(1+d\cdot y)^{\frac {p+1}{p-1}}},
\end{equation}
where $c_0$ is defined in \eqref{defWl2}, and $\bar W_{i,1}(d,y)$ is uniquely determined as the solution $v_1$ of equation \eqref{eqWl1} with $v_2=\bar W_{i,2}(d,y)$. From straightforward calculations, we see that $\bar W_i(d,y)$ for $i=1,\dots,N$ are linearly independent, and span all the $W_i(d)$ for all $i=1,\dots,N$, hence, it is a basis for the eigenspace of $L_d^*$ corresponding to $\lambda=0$. 
For that reason, our goal \eqref{ortho} is equivalent to the fact that
\begin{equation}\label{defbpi}
\forall i=0,\dots,N,\;\;\bar \pi^{d^*}_i(q)=0,
\mbox{ where } 
\bar \pi^d_i(q)=\phi\left(\bar W_i(d), q\right).
\end{equation}
Introducing $\bar \Psi$ defined as in \eqref{defPsi}, with $\pi_i^{d^*}$ replaced by  $\bar \pi_i^{d^*}$, the conclusion follows from the application of the implicit function theorem to $\bar \Psi$ near the point given in \eqref{condnu1}. Since $\bar W_{i,2}$ is a $C^1$ function of $d$ (see \eqref{defwb}), using  item (i) in Lemma \ref{clfrancois*}, we see that
\[
\forall |d|<\frac 12\mbox{ and }i=0,\dots,N,\;\;\|\bar W_i(d)\|_{\q H}+\|\nabla_d\bar W_i(d)\|_{\q H}\le C. 
\]
Therefore, items 1- and 2- follow as for $\Psi$, and item 3- reduces to the computation of $\bar \pi_i^0(\partial_\nu \kappa^*(0,\nu))$ and $\bar \pi_i^0(\partial_{d_j} \kappa^*(0,\nu))$, for all $i=0,\dots,N$, $j=1,\dots,N$ and $-1+\frac 1{A+1}\le \nu \le A+1$. But in that case, we see from \eqref{defwb} and \eqref{defWl2} that $\bar W_i(0)=W_i(0)$, hence $\bar \pi_i^0=\pi_i^0$. Since the basis $(e_1(0),\dots,e_N(0))$ is the canonical basis of $\m R^N$, as we have chosen in the remark following Lemma \ref{lembase}, these projections follow from Claim \ref{cltech} with $d=0$, and $\bar M_0$ is invertible for $d=0$, where $\bar M_0$ is the analogous of $M_0$ with $\pi^{d^*}_i$ replaced by $\bar \pi^{d^*}_i$. From continuity with respect to $d$, we deduce that for $\dun$ small enough, $\bar M_0$ remains invertible and so does the Jacobian $\bar M$, provided that $\|q\|_{\q H}$ is small. Therefore, item 3- holds when $\delta_1$ and $\|q\|_{\q H}$ are small enough. Thus, the implicit function theorem applies, and we get the conclusion of Lemma \ref{lemmod}, provided that we fix $\dun< \frac 14$ is small enough and $|\bar d|\le \dun$. Recalling that we have reached the same conclusion in Case 1, when $|\bar d|\ge \dun$, whatever was the value of $\dun$ in the interval $(0, \frac 14)$, this concludes the proof of Lemma \ref{lemmod}. 
\end{proof}

\section{Dynamics in self-similar variables}\label{appdyn}
We prove Proposition \ref{propdyn} here. We have already performed such a technique in our earlier work: \cite{MZjfa07}, \cite{MZajm11}, \cite{MZisol10} and \cite{CZmulti11}. Given that those papers were performed in one space dimension, we have to ask ourselves at each step, how to extend the techniques to higher dimensions.\\
In most cases, the extension can be done in a straightforward way, hence, we will omit the proofs of those cases, and we will refer the reader to our previous papers.\\
Nevertheless, some of the techniques need some special care to be adapted to higher dimensions. We will of course give all the details for them. This is the case for the control of the nonlinear term, where the use of the following Hardy-Sobolev inequality is crucial:
\begin{lem}[A Hardy-Sobolev identity]\label{lemhs} For all $v\in \H_0$, it holds that 
\begin{align*}
\int_{|z|<1} |v(z)|^{p+1}&(1-|z|^2)^\alpha dz\\
 &\le C\left(\int_{|z|<1}\left( |\nabla v(z)|^2-|z\cdot \nabla v(z)|^2+v(z)^2\right)(1-|z|^2)^\alpha dz
\right)^{\frac {p+1}2}
\end{align*}
where $\alpha$ is introduced in \eqref{defro}. 
\end{lem}
\begin{proof}
See Section \eqref{sechs} below for the proof.
\end{proof}
Let us first use this estimate to prove Proposition \ref{propdyn}, then we will prove it.

\subsection{Projection of the linearization of equation \eqref{eqw1}}
We linearize equation \eqref{eqw1} around $\kappa^*(d,\nu)$, then project the obtained equation on the different directions of its linear part, proving this way Proposition \ref{propdyn}.
\begin{proof}[Proof of Proposition \ref{propdyn}] Using equation \eqref{eqw1} satisfied by $(w, \partial_s w)(y,s)$ and the decomposition \eqref{defq}, we derive the following equation for $q(y,s)$ for all $s\in [0, \hat s)$:
\begin{equation}\label{eqq0}
\ds\partial_s q =
\bar L(q) +\vc{0}{f(q_1)}-(\nu'-\nu)\partial_\nu \kappa^*(d,\nu)
-\sum_{i=1}^N (d' \cdot e_i)\partial_{e_i} \kappa^*(d,\nu)
\end{equation}
where 
\begin{align}
\bar L\vc{q_1}{q_2}&=\vc{q_2}{\q L q_1+\bar\psi(d,y)q_1-\frac{p+3}{p-1}q_2-2y.\nabla q_2},\label{defld0}\\
f(q_1)&=|\kappa^*_1(d,\nu)+q_1|^{p-1}(\kappa^*_1(d,\nu)+q_1)-\kappa^*_1(d,\nu)^p-p\kappa^*_1(d,\nu)^{p-1}q_1,\nonumber\\
\bar \psi(y,s)&=p\kappa^*_1(d,\nu,y)^{p-1}-\frac{2(p+1)}{(p-1)^2}.\nonumber
\end{align}
Note from \eqref{condmod} that 
\begin{equation}\label{equiv}
\frac{|d|}{1+A}\le |d^*|\le A|d|\mbox{ and } \frac 1{A(1+A)} \le \frac{1-|d|}{1-|d^*|}\le A(1+A),
\end{equation}
where $d^* = \frac{d}{1+\nu}$

{\it - Proof of \eqref{first}}: 
From the orthogonality condition \eqref{kill}, it is convenient to see the linear operator $\bar L$ as a perturbation of the operator $L_{d^*}$ defined in \eqref{defld}, in the sense that
\begin{equation}\label{detail}
\bar L(q) = L_{d^*}(q)+\vc{0}{\bar V(y,s)q_1}
\mbox{ where }
\bar V(y,s) = p\kappa^*_1(d,\nu,y)^{p-1}-p\kappa(d^*,y)^{p-1}.
\end{equation}
 Consider $i=0,\dots,N$. We need to project equation \eqref{eqq0} with the projector $\pi^{d^*}_i$ defined in \eqref{defpdi}, and estimate in the following the different terms (with the term $\bar L(q)$ expanded as in \eqref{detail}). Since $\pi^{d^*}_i$ depends on $W_i^{d^*}$, which is singular at $d^*=0$, our proof splits into two cases: $|d^*|\ge \delta_2$ and $|d^*|\le \delta_2$, where $\delta_2$ will be arbitrary in $(0,\frac 14)$ in the first case, then fixed small enough in the second. We had to do the same in the proof of the modulation result in Lemma \ref{lemmod}.

\bigskip

{\bf Case 1: $|d(s)|\ge \delta_2$.}

In this section, $\delta_2$ is {\it arbitrary} in $(0,\frac 14)$. The constant $C$ may depend on $A$ in the following.
From Section C.2 page 2896 in \cite{MZisol10}, we know that
\begin{align}
|\pi^{d^*}_i(\partial_s q)|&\le C\|q\|_{\H}\ds\sum_{j=1}^N|{d^*}'.e_j|\|\partial_{e_j}W_i(d^*)\|_{\H'},\label{dsq}&
\pi^{d^*}_i(L_{d^*}(q))=\lambda_i \pi^{d^*}_i(q)=0,&\\
|\bar V(y,s)|&\le \ds\frac{C(A)|\nu|}{1-|d|}\kappa(d^*,y)^{p-1}.\label{barv}
\end{align}

- Since ${d^*}'.e_j=-\frac{\nu'|d|\delta_{j,1}}{(1+\nu)^2}+\frac{d'.e_j}{1+\nu}$, using \eqref{dsq}, item (iii) in Lemma \ref{eigenld*} and \eqref{equiv}, we see that
\[
|\pi^{d^*}_i(\partial_s q)|\le C\|q\|_{\H}\left[\frac {|\nu|+|\nu'-\nu|+|d' \cdot e_1|}{|d|(1-|d|)}+\sum_{j=2}^N\frac{|d' \cdot e_j|}{|d|\sqrt{1-|d|}}\right].
\]

\medskip

- Since we know from the definition \eqref{defkd} of $\kappa(d^*,y)$ that
\begin{equation}\label{kp1}
\kappa(d^*,y)\le C(1-|y|^2)^{-\frac 1{p-1}}, 
\end{equation}
we use the definition \eqref{defpdi} of $\pi^{d^*}_i$, \eqref{barv}, the Cauchy-Schwarz inequality and the identity \eqref{hsajm} to write 
\begin{align*}
&\left|\pi^{d^*}_i \vc{0}{\bar V(y,s)q_1}\right|
\le \frac{C|\nu|}{1-|d|}\int_{|y|<1} |W_{i,2}(d^*,y)|q_1(y,s)|\frac{\rho(y)}{1-|y|^2} dy\\
&\le \frac{C|\nu|}{1-|d|}\|W_{i,2}(d^*)\|_{L^2_{\frac \rho{1-|y|^2}}}\|q_1\|_{L^2_{\frac \rho{1-|y|^2}}}\le \frac{C|\nu|}{1-|d|}\|q\|_{\q H}.
\end{align*}

- Using \eqref{wkd}, 
the definition \eqref{defpdi} of the projection $\pi_i^{d^*}$ and \eqref{kp2}, we see that
\begin{equation*}
\left|\pi^{d^*}_i \vc{0}{f(q_1)}\right|\le C(A) \iint \kappa(d^*,y)|f(q_1)|\rho dy.
\end{equation*}
Since we see from the definition of $f(q_1)$ given in \eqref{defld0} that
\begin{equation}\label{bfq1}
|f(q_1)|\le C \delta_{p>2}|q_1|^p+C\kappa_1^*(d,\nu)^{p-2}q_1^2,
\end{equation}
we write 
\[
\iint \kappa(d^*,y)|f(q_1)|\rho dy
\le C(A) \delta_{p>2}\iint \kappa(d^*)|q_1|^p \rho dy + C(A) \iint \kappa(d^*)^{p-1}q_1^2 \rho dy.
\]
Using H\"older's inequality, the Hardy-Sobolev estimate of Lemma \ref{lemhs}, \eqref{kp1}, \eqref{boundkd} and the embedding \eqref{hsajm}, we write
\begin{align*}
\iint \kappa(d^*)|q_1|^p \rho dy &\le \|\kappa(d^*)\|_{L^{p+1}_\rho}\|q_1\|_{L^{p+1}_\rho}^p\le C \|\kappa(d^*)\|_{\q H_0}\|q\|_{\q H}^p\le C\|q\|_{\q H}^p,\\
\iint \kappa(d^*)^{p-1}q_1^2 \rho dy&\le C \iint q_1^2 \frac \rho{1-|y|^2} dy\le C \|q\|_{\q H}^2.
\end{align*}
Thus, using \eqref{condmod}, we see that for $\eb\le 1$, we have
\begin{equation}\label{boundfq1}
\left|\pi^{d^*}_i \vc{0}{f(q_1)}\right|\le C(A) \iint \kappa(d^*,y)|f(q_1)|\rho dy\le C(A) \|q\|_{\q H}^2.
\end{equation}

- Using Claim \ref{cltech} for the projections of the derivatives of $\kappa^*(d,\nu)$, we write from equation \eqref{eqq0} and the above-stated estimates (starting first with $i=1$, then $i=0$ and finally $i=2,\dots,N$)
\begin{align*}
&\frac {|\nu'-\nu|+|d' \cdot e_1|}{1-|d|}+\sum_{i=2}^N\frac{|d' \cdot e_i|}{\sqrt{1-|d|}}\\
\le& \frac C{\ddeux}\|q\|_{\H}\left[\frac {|\nu|+|\nu'-\nu|+|d' \cdot e_1|}{1-|d^*|}+\sum_{j=2}^N\frac{|d' \cdot e_j|}{\sqrt{1-|d|}}\right]+C\|q\|_{\q H}^2.
\end{align*}
Using \eqref{condmod}, \eqref{equiv} and taking $\frac{\eb}{\ddeux}$ small enough yields the first estimate in Proposition \ref{propdyn} with $C_2=\frac{C(A)}{\ddeux}$, if $|d(s)|\ge \ddeux$.

\bigskip

{\bf Case 2: $|d(s)|\le \ddeux$.}

In this section, we will fix $\ddeux$ small enough in $(0, \frac 14)$. Since $\pi_i^{d^*}$ becomes singular as $d\to 0$ (see definition \ref{lemexpansion} and Lemma \ref{eigenld*}), we will project equation \eqref{eqq0} with the projector $\bar \pi_i^{d^*}$ defined in \eqref{defbpi}, which is not singular. 
Since all the functions we are handling are $C^1$ for $(\nu,d) \in [-1+\frac 1A, A]\times B(0, \ddeux)$, we will always compute the following quantities for $d=0$, then add $O(d)$ to the result (note that here, the notation $g_1=O(g_2)$ stands for $|g_1|\le C(A)|g_2|$).

\medskip

Applying the projector $\bar \pi_i^{d^*}$ defined in \eqref{defbpi} to the different terms appearing in equation \eqref{eqq0} and arguing as for Case 1 and Section C.2 page 2896 in \cite{MZisol10}, we see that
\begin{align}
|\bar \pi_i^{d^*}(\partial_s q)| = |{d_i^*}'|\sum_{i=1}^N|\phi(\partial_{d_i} \bar W_i(d^*), q)|\le C(A)\|q\|_{\q H}\left[|d'|+|\nu|+|\nu-\nu'|\right],\nonumber\\
\bar \pi_i^{d^*}(L_{d^*}(q))=0,\;\;
|\bar \pi_i^{d^*}(0,\bar V(y,s)q_1)|\le C(A)|\nu|\|q\|_{\q H},\;\;
|\bar \pi_i^{d^*}(0,f(q_1))|\le C(A)\|q\|_{\q H}^2,\nonumber\\
|\bar \pi_i^{d^*}(\partial_\nu\kappa^*(d,\nu))- \bar \pi_i^0(\partial_\nu\kappa^*(0,\nu))|
+|\bar \pi_i^{d^*}(\partial_{d_j}\kappa^*(d,\nu))- \bar \pi_i^0(\partial_{d_j}\kappa^*(0,\nu))|\le C(A)|d|.\label{proj0}
\end{align}
Since we have by definitions \eqref{defpdi} and \eqref{defbpi} together with Lemma \ref{lembase} that $\bar \pi_i^0= \pi_i^0$, and $(e_i(0),\dots,e_N(0))$ is the canonical basis of $\m R^N$, we can use Claim \ref{cltech} with $d=0$ to estimate the projections in \eqref{proj0}. Writing the last term in the first line of equation \eqref{eqq0} as $\sum_{i=1}^N d_j' \partial_{d_j}\kappa^*(d,\nu)$, then using the above estimates (starting first with $i=1$, then $i=0$ and finally $i=2,\dots,N$), we see that
\[
|\nu'-\nu|+|d'|
\le C(A)(\ddeux+\|q\|_{\H})\left[|\nu|+|\nu'-\nu|+|d'|\right]+C(A)\|q\|_{\q H}^2.
\]
Using \eqref{condmod}, fixing $\ddeux=\ddeux(A)$ small enough, then taking $\eb$  small enough yields the first estimate in Proposition \ref{propdyn} with $C_2=C(A)$, if $|d(s)|\le \ddeux$. Since we have already proved that estimate when $|d(s)|\le \ddeux$ with a constant $C_2=\frac{C(A)}{\ddeux}$ and $\ddeux$ has just been fixed now, this concludes the proof of the first estimate in Proposition \ref{propdyn}.

\bigskip

{\it - Proof of \eqref{second} and \eqref{third}}: As for Claim 4.8 page 2867 in \cite{MZisol10}, the idea is simple: choose the right multiplying factor for equation \eqref{eqq0}, then integrate. More precisely, \\
- For the proof of \eqref{second}, we compute $\frac d{ds}\|q\|_{\q H}^2=2\phi(q, \partial_s q)$, and use equation \eqref{eqq0} to estimate this;\\
- For the proof of \eqref{third}, we find a Lyapunov functional for equation \eqref{eqq0}, by multiplying the equation on $q_1$ derived from \eqref{eqq0} by $\partial_s q_1\rho$ then integrating on the unit ball. Without the modulation terms, that functional would be $\frac 12 \bar \varphi(q,q)-\iint \q F(q_1)\rho dy$, where
\begin{align}
\bar \varphi\left(q, r\right)&= \int_{|y|<1} \left(-\bar \psi(d,y)q_1r_1+\nabla q_1\cdot \nabla r_1-(y\cdot \nabla r_1)(y\cdot \nabla q_1)+q_2r_2\right)\rho dy,\label{defphib}\\
\q F(q_1)& = \int_0^{q_1}f(\xi) d\xi = \frac{|\kappa^*_1+q_1|^{p+1}}{p+1}-\frac{{\kappa^*_1}^{p+1}}{p+1}-{\kappa^*_1}^p q_1 - \frac p2 {\kappa^*_1}^{p-1}q_1^2, \label{defF}
\end{align}
and $\bar\psi(d,y)$ is defined in \eqref{defld0}. Because we ``killed'' the nonnegative directions in \eqref{kill}, we will see that our Lyapunov functional controls the square of the norm of the solution, if \eqref{cond0} holds.
However, because of the modulation, we need to slightly change the functionals we intend to study, by defining:
\begin{equation}\label{defh}
\begin{array}{rl}
h_1(s)& = \frac 12\|q\|_{\q H}^2-\iint \q F(q_1)\rho dy,\\
h_2(s)& = \frac 12 \bar \varphi(q,q)-\iint \q F(q_1)\rho dy+\eta_0 \iint q_1 q_2 \rho dy,
\end{array}
\end{equation}
where $\eta_0>0$ will be fixed later as a small enough universal constant.
Then, we clearly see that the following identity allows to conclude:

\medskip

 {\it There exist $\delta>0$ such that 
\begin{equation}\label{lemlyap}
\begin{array}{rl}
\forall s\in [0, \hat s),&\delta h_1(s) \le \|q(s)\|_{\q H}^2 \le \delta^{-1}h_1(s)\mbox{ and }h_1'(s)\le \delta^{-1} h_1(s),\\
\forall s\in [0, \tilde s),&\delta h_2(s) \le \|q(s)\|_{\q H}^2 \le \delta^{-1}h_2(s)\mbox{ and }h_2'(s)\le -\delta h_2(s).
\end{array}
\end{equation}
}
It remains then to prove \eqref{lemlyap} in order to conclude. 
The proof follows the proof of Claim 4.8 page 2898 in \cite{MZisol10}. For that reason, we will recall estimates from that paper, and only stress the novelties. We claim that \eqref{lemlyap} follows from the following:
\begin{lem}\label{lemproj*} There exists $\evingt>0$ such that if $\epsilon_0\le \evingt$, then for all $s\in [0, \hat s)$, we have:
\begin{align}
\frac 12 \frac d{ds}\|q\|_{\q H}^2\le& 
\iint q_2f(q_1) \rho dy+C(A) \|q\|_{\q H}^2 ,\label{eq-1}\\
\iint q_2f(q_1) \rho dy\le& \frac d{ds} \iint \q F(q_1)\rho dy+ C(A)\epsilon_0\|q\|_{\q H}^2,\label{eq0}\\
\left|\iint \q F(q_1) \rho dy\right|\le& C(A) \|q\|_{\H}^{\bar p+1}\le C(A) \epsilon_0^{\bar p-1}\|q\|_{\H}^2.\label{eq2a}
\end{align}
where $\bar p = \min(p,2)$.\\
Moreover, there exists $\eta_3(A)>0$ small enough such that 
if $\eta\le \eta_3(A)$, then, for all $s\in [0, \tilde s)$,
\begin{align}
&\frac 12 \frac d{ds}\bar\varphi(q,q)\le - 2\alpha\iint q_{2}^2 \frac \rho{1-|y|^2} dy+\iint q_2f(q_1) \rho dy
+C(A)\left(\epsilon_0
+\eta
\right)\|q\|_{\H}^2,\label{eq1}\\
&\frac {\bar\varphi(q,q)}{C} \le \|q\|_{\H}^2\le C \bar\varphi(q,q),
\label{eq2}\\
&\frac d{ds}\iint q_1q_2 \rho \le
-\frac 7{10}\bar\varphi(q,q)+ C \iint q_{2}^2 \frac \rho{1-|y|^2}dy.\label{eq3}
\end{align}
\end{lem}
\noindent Indeed, if $\epsilon_0\le \epsilon_3$, we see from this lemma and \eqref{defh}
that for all $s\in [0, \hat s)$,
\[
|h_1(s) - \frac 12 \|q\|_{\q H}^2|\le C(A)\epsilon_0^{\bar p-1}\|q\|_{\q H}^2
\mbox{ and }h_1'\le C(A) \|q\|_{\q H}^2.
\]
Furthermore, if $\eta \le \eta_3$, then we have for all $s\in [0, \tilde s)$,
\begin{align*}
|h_2(s) - \frac 12 \bar \varphi(q,q)|&\le (C(A)\epsilon_0^{\bar p-1}+C\eta_0)\|q\|_{\q H}^2,\\
h_2'(s)&\le \left(-2\alpha+C\eta_0\right)\iint q_2^2 \frac \rho{1-|y|^2} dy 
+\left(C(A) (\epsilon_0+\eta)-\frac 7{10}\eta_0\right)\bar \varphi(q,q).
\end{align*}
Fixing first $\eta_0=\eta_0(N,p)>0$ small enough so that the multiplying factor in front of $\iint q_2^2 \frac \rho{1-|y|^2} dy $ in the above inequality is negative, 
then imposing the conditions $\eta\le \eta_2(A)$ and $\epsilon_0\le \epsilon_2(A)$ for some $\eta_2(A)>0$ and $\epsilon_2(A)>0$ small enough, 
we see from \eqref{eq2} that \eqref{lemlyap} holds, and so does \eqref{second} and \eqref{third} in Proposition \ref{propdyn}. It remains then to justify Lemma \ref{lemproj*} in order to conclude the proof of \eqref{lemlyap} and Proposition \ref{propdyn} too.

\begin{proof}[Proof of Lemma \ref{lemproj*}] Following our techniques performed for the proof of Lemma C.2 page 2896 in \cite{MZisol10}, recalling identity (3.41) page 2857 in that paper, using \eqref{equiv} and \eqref{estl2}, we see that for all $s\in [0, \hat s)$, 
\begin{align}
\frac 12 \frac d{ds}\|q\|_{\q H}^2&\le \iint q_2f(q_1) \rho dy+C(A) \|q\|_{\q H}^2+CR_1,\label{eqcarre}\\
\iint q_2f(q_1) \rho dy&\le \frac d{ds}\iint \q F(q_1) \rho dy
+
 C|\nu|\iint |\partial_\nu\kappa^*_1(d,\nu)||f(q_1)|\rho dy+CR_2,\label{eqnl}\\
\frac 12 \frac d{ds} \bar\varphi(q,q)&\le -2\alpha \iint q_2^2 \frac \rho{1-|y|^2} dy+\iint q_2f(q_1) \rho dy+C(A)R_1+CR_2\label{eqa-}\\
\frac d{ds} \iint q_1 q_2 \rho dy &\le -\frac 9{10}\bar \varphi(q,q)+C\iint q_2^2 \frac \rho{1-|y|^2}+\iint q_1 f(q_1)\rho dy+CR_1,\label{q1q2}\\
|\partial_\nu \kappa^*_1(d,\nu)|&+|\partial_{e_1} \kappa^*_1(d,\nu)|\le C(A) \frac {\kappa(d^*,y)}{1-|d|},\label{dnde}
\end{align}
where
\begin{align}
 R_1(s) &= \|q\|_{\q H}\left(|\nu'-\nu|\|\partial_\nu \kappa^*_1(d,\nu)\|_{\q H}+\sum_{i=1}^N |d'\cdot e_i|\|\partial_{e_i} \kappa^*(d,\nu)\|_{\q H}\right),\label{defR1}\\
R_2(s) &=|\nu'|\iint |\partial_{\nu} \kappa^*_1(d,\nu)|(\kappa^*_1(d,\nu))^{p-2}q_1^2 \rho dy\nonumber\\
+&\sum_{i=1}^N |d'\cdot e_i|\iint |\partial_{e_i} \kappa^*_1(d,\nu)|(\kappa^*_1(d,\nu))^{p-2}q_1^2 \rho dy.\nonumber
\end{align}
Using \eqref{kp2}, \eqref{dnde}, \eqref{peik*} and the interpolation identity \eqref{hsajm}, we see that
\begin{align*}
\iint \left(|\partial_{\nu} \kappa^*_1(d,\nu)|+|\partial_{e_1} \kappa^*_1(d,\nu)|\right)(\kappa^*_1(d,\nu))^{p-2}q_1^2 \rho dy
&\le \frac {C(A)}{1-|d|}\|q\|_{\q H}^2,\\
\mbox{and if }i=2,\dots,N,\;\;\iint |\partial_{e_i} \kappa^*_1(d,\nu)|(\kappa^*_1(d,\nu))^{p-2}q_1^2 \rho dy
&\le \frac {C(A)}{\sqrt{1-|d|}}\|q\|_{\q H}^2.
\end{align*}
Using \eqref{normdnk}, \eqref{normd2k} and the differential inequalities \eqref{first} satisfied by the parameters $\nu$ and $d$, together with \eqref{condmod}, we see that
\begin{equation}\label{R1R2}
|R_1(s)| \le C(A)\left(\epsilon_0+\frac{|\nu|}{1-|d|}\right)\|q\|_{\q H}^2\mbox{ and }|R_2(s)|\le C(A)\epsilon_0\|q\|_{\q H}^2.
\end{equation}

{\it - Proof of \eqref{eq-1}}: The estimate follows directly from \eqref{eqcarre}, \eqref{R1R2} and \eqref{condmod}.

\medskip

{\it - Proof of \eqref{eq0}}: This is a direct consequence of \eqref{eqnl}, \eqref{dnde}, \eqref{boundfq1} and \eqref{R1R2}.

\medskip

{\it - Proof of \eqref{eq2a}}: Noting that 
$|\q F(q_1)|\le C|q_1|^{p+1}+C\delta_{p\ge 2}\kappa^*_1(d,\nu)^{p-2}|q_1|^3$, by definition \eqref{defF}, 
and arguing as for \eqref{boundfq1}, then, using \eqref{condmod}, we get \eqref{eq2a} (note that the Hardy-Sobolev estimate of Lemma \ref{lemhs} is again crucial here).

\medskip

For the following estimates, we assume that $s\in [0, \hat s)$, hence \eqref{cond0} holds. 

\medskip

{\it - Proof of \eqref{eq1}}: This is a direct consequence of \eqref{eqa-}, \eqref{R1R2} and \eqref{cond0}.

\medskip

{\it - Proof of \eqref{eq2}}: Using the definitions \eqref{defphib}, \eqref{defphid} and \eqref{barv} of $\bar \varphi(q,q)$, $\varphi_{d^*}(q,q)$ and $|\bar V|$, together with \eqref{kp1}, the embedding \eqref{hsajm} and \eqref{cond0}, we see that
\[
|\bar\varphi(q,q)- \varphi_{d^*}(q,q)|=\iint |\bar V|q_1^2\rho dy
\le \frac{C(A)|\nu|}{1-|d|}\iint q_1^2 \frac\rho{1-|y|^2} dy
\le C(A)\eta \|q\|_{\q H}^2.
\]
Using the expansion \eqref{expansion} for $q(y,s)$, we see from \eqref{kill} and \eqref{ii}
that \eqref{eq2} follows for $\eta\le \bar\eta_3(A)$ for some $\bar\eta_3(A)$ small enough.

\medskip

{\it - Proof of \eqref{eq3}}: Since $|q_1f(q_1)|\le C|q_1|^{p+1}+C\delta_{p\ge 2}\kappa^*_1(d,\nu)^{p-2}|q_1|^3$ from \eqref{bfq1}, arguing as for \eqref{boundfq1}, we see that
\[
\left|\iint q_1 f(q_1)\rho dy \right|\le C(A)\|q\|_{\q H}^{\bar p+1} \le \frac{\bar\varphi(q,q)}{100},
\]
for $\epsilon_0$ small enough, where we used \eqref{eq2} together with the smallness condition in \eqref{condmod} for the last inequality.\\
Plugging this estimate in \eqref{q1q2}, and using \eqref{R1R2}, we get \eqref{eq3} for $\eta$ and $\epsilon_0$ small enough. This concludes the proof of Lemma \ref{lemproj*}. 
\end{proof}
Since we have already seen after the statement of that lemma that it implies identity \eqref{lemlyap}, this concludes also the proof of identity \eqref{lemlyap}.
Since estimates \eqref{second} and \eqref{third} of Proposition \ref{propdyn} are direct consequences of \eqref{lemlyap}, this concludes the proof of Proposition \ref{propdyn}, assuming that the Hardy-Sobolev estimate stated in Lemma \ref{lemhs} holds (see the next subsection for that proof).
\end{proof}

\subsection{A Hardy-Sobolev inequality}\label{sechs}
We prove the Hardy-Sobolev estimate of Lemma \ref{lemhs} here. Note that we have already proved the one-dimensional case in Lemma 2.2 page 51 in \cite{MZjfa07}, thanks to the transformation $y = \tanh \xi$ with $\xi \in \m R$. In higher dimensions, we can't do that, and we need to build-up a completely new framework to get the result. Let us give the details in the following.

\medskip

First, taking
\[
q=2^*=\frac{2N}{N-2}\mbox{ if }N\ge 3\mbox{ and any }q\ge 2\mbox{ if }N=2,
\] 
we give a series of useful lemmas and then give the proof of Lemma \ref{lemhs} at the end: 
\begin{lem}\label{lem0} For any $g\in H^1_0(B)$, we have the following identities:
\begin{align*}
&(i) 
\left(\iint |g(y)|^qdy\right)^{2/q}\le C \iint |\nabla g(y)|^2dy,\\
&(ii) \iint \frac{g(y)^2}{(1-|y|^2)^2} dy \le \iint |\nabla g(y)|^2dy.
\end{align*}
\end{lem}
\begin{proof}The first identity follows from the Sobolev embedding. In order to prove the second identity, we consider $\gamma\ge 0$ and write
\begin{align}
0&\le \iint \left|\nabla g +\gamma y \frac{g}{1-|y|^2}\right|^2 dy\label{prodrem}\\
& = \iint |\nabla g|^2 dy + 2\gamma  \iint \frac{gy\cdot\nabla g}{(1-|y|^2)} dy +\gamma^2\iint \frac{g^2|y|^2}{(1-|y|^2)^2}dy.\nonumber
\end{align}
Since 
\begin{align*}
&2\iint  \frac{gy\cdot \nabla g}{(1-|y|^2)} dy = \iint  \frac{y\cdot \nabla (g^2)}{(1-|y|^2)} dy=-\iint  g^2\dv\left( \frac{y}{(1-|y|^2)}\right) dy\\
=&-\iint g^2\left(\frac {N-2} {(1-|y|^2)}+\frac {2}{(1-|y|^2)^2}\right) dy \le -2\iint \frac{g^2}{(1-|y|^2)^2}dy,
\end{align*}
 we write from \eqref{prodrem}
\[
0\le \iint |\nabla g(y)|^2dy+(\gamma^2-2\gamma)\iint \frac{g^2}{(1-|y|^2)^2}dy.
\]
Taking $\gamma=1$, we 
get the conclusion.
\end{proof}
Now, we give the following lemma:
\begin{lem}\label{lem1} Consider $a\ge 0$ with $a\neq 1$. Then, 
\begin{align*}
&\left(\iint |h|^q(1-|y|^2)^{qa/2} dy \right)^{2/q}+\iint h^2(1-|y|^2)^{a-2} dy\\
\le& C \iint \left(|\nabla h|^2+ h^2\right) (1-|y|^2)^ady.
\end{align*}
\end{lem}
\begin{proof} Consider $a\ge 0$ with $a\neq 1$ and $h$ such that
\begin{equation}\label{hypo-new}
\iint \left(|\nabla h(y)|^2+h(y)^2\right)(1-|y|^2)^a dy<+\infty.
\end{equation}
Let us compute
\begin{align*}
&\iint |\nabla (h(1-|y|^2)^{a/2})|^2 dy
=\iint |\nabla h (1-|y|^2)^{a/2}-ayh(1-|y|^2)^{\frac a2 -1}|^2 dy\\
=&\iint |\nabla h|^2 (1-|y|^2)^ady-2a\iint hy\cdot \nabla h (1-|y|^2)^{a-1}dy\\
+&a^2\iint h^2|y|^2(1-|y|^2)^{a-2}dy.
\end{align*}
Using integration by parts, we write
\begin{align*}
&-2a\iint h y\cdot \nabla h (1-|y|^2)^{a-1}dy
=-a \iint y\cdot \nabla (h^2) (1-|y|^2)^{a-1}dy\\
=& a \iint h^2 \dv \left(y(1-|y|^2)^{a-1}\right) dy\\
=&a \iint h^2 \left(N(1-|y|^2)^{a-1}-2(a-1)|y|^2(1-|y|^2)^{a-2}\right)dy.
\end{align*}
Therefore, we get
\begin{align}
&\iint |\nabla (h(1-|y|^2)^{a/2})|^2 dy
\le \iint |\nabla h|^2 (1-|y|^2)^ady\nonumber\\
+&Na\iint h^2 (1-|y|^2)^{a-1}dy+(a^2 -2a(a-1))\iint h^2|y|^2(1-|y|^2)^{a-2} \nonumber\\
=&\iint |\nabla h|^2 (1-|y|^2)^ady\nonumber\\
+&[Na+a^2-2a]\iint h^2 (1-|y|^2)^{a-1}dy+(-a^2+2a)\iint h^2 (1-|y|^2)^{a-2}dy\nonumber. 
\end{align}
Since for all $\epsilon>0$, we have
\[
 \iint h^2 (1-|y|^2)^{a-1}dy \le \epsilon \iint h^2 (1-|y|^2)^{a-2}dy+\frac 1\epsilon \iint h^2 (1-|y|^2)^ady,
\]
we get 
\begin{align}
&\iint |\nabla (h(1-|y|^2)^{a/2})|^2 dy
\le \iint |\nabla h|^2 (1-|y|^2)^ady\nonumber\\
+&\frac C\epsilon\iint h^2 (1-|y|^2)^ady+(-a^2+2a+C\epsilon)\iint h^2 (1-|y|^2)^{a-2}dy. \label{rabih}
\end{align}
Now, we claim that 
\begin{align}
&\iint |\nabla (h(1-|y|)^{a/2})|^2 dy\label{concl}
\le& C\iint |\nabla h|^2 (1-|y|)^ady+C\iint h^2 (1-|y|)^ady.
\end{align}
Indeed:\\
- If $a> 2$, then we may choose $\epsilon>0$ small enough so that $-a^2+ 2a+C \epsilon< 0$, and \eqref{concl} follows from \eqref{rabih}.\\
- If $0\le a\le 2$ and $a\neq 1$, then, we write the following from Lemma \ref{lem0}:  
\begin{equation}\label{sob}
\iint h^2 (1-|y|^2)^{a-2}dy\le \iint |\nabla (h(1-|y|^2)^{a/2})|^2 dy.
\end{equation} 
Using \eqref{rabih} and \eqref{sob}, we see that
\begin{align*}
&\iint |\nabla (h(1-|y|^2)^{a/2})|^2 dy
\le \iint |\nabla h|^2 (1-|y|^2)^ady\\
+&\left(-a^2+2a+C\epsilon\right)\iint |\nabla (h(1-|y|^2)^{a/2})|^2 dy+\frac C\epsilon \iint h^2 (1-|y|)^ady. 
\end{align*}
Since $a\neq 1$, we may choose $\epsilon>0$ small enough so that  $-a^2+2 a+C\epsilon<1$, hence, \eqref{concl} follows. Using \eqref{hypo-new}, we see that $g(y)\equiv h(y) (1-|y|^2)^{a/2}\in H^1_0(B)$. Applying Lemma \ref{lem0} to $g$, we conclude the proof of Lemma \ref{lem1}.
\end{proof}
Now, we claim the following:
\begin{lem}\label{lem3}
If 
\begin{equation}\label{conda0}
0\le a\le \frac 4{p-1}+2-N\mbox{ and }a\neq 1,
\end{equation}
 then we have 
\[
\left(\iint |h|^{p+1}(1-|y|^2)^a dy \right)^{\frac 2{p+1}}
\le C \iint \left(|\nabla h|^2+h^2\right) (1-|y|^2)^a dy
\]
\end{lem}
\begin{proof}Taking $q> p+1>2$
we interpolate as follows
\begin{align*}
\iint |h|^{p+1}&(1-|y|^2)^a dy\\
 &\le \left(\iint |h|^q (1-|y|^2)^{\frac{qa}2}dy\right)^{\eta} \left(\iint |h|^{2} (1-|y|^2)^{\bb}dy\right)^{1-\eta}
\end{align*}
where 
\begin{equation}\label{defeta}
p+1=q\eta+2(1-\eta)\mbox{ and }a=\frac{qa}2\eta+\bb(1-\eta),
\end{equation}
that is
\[
\eta=\frac{p+1-2}{q-2}
\mbox{ and }\bb = \frac{((3-p)q-4)a}{2(q-p-1)}.
\]
If $N\ge 3$, we have $q=2^*$ and we see from \eqref{conda0} that
\[
\bb -(a-2)=-\frac{2(p-1)[a+N-2-\frac 4{p-1}]}{4-(N-2)(p-1)}\ge 0.
\]
If $N=2$, then we see from \eqref{conda0} that
\[
\bb -(a-2) = \frac{a(p-1)(-q+2)}{q-p-1}+2\searrow -\frac{(p-1)a}2+2\ge 0\mbox{ as }q\to \infty
\]
hence, taking $q$ large enough, we will have $b\ge a-2$.\\
 Therefore, 
\begin{align*}
\iint |h|^{p+1}&(1-|y|^2)^a dy\\
& \le\left(\iint |h|^q (1-|y|^2)^{\frac{qa}2}dy\right)^{\eta} \left(\iint |h|^{2} (1-|y|^2)^{a-2}dy\right)^{1-\eta}.
\end{align*}
Since $a\neq 1$,
the conclusion follows from Lemma \ref{lem1} and \eqref{defeta}.
\end{proof}

\bigskip

Now, we are ready to give the proof of Lemma \ref{lemhs}.
\begin{proof}[Proof of Lemma \ref{lemhs}]
Using radial and angular coordinates for $v(z)$:
\[
v(z)=V(s,\omega)\mbox{ with }z = s\omega,\;\;s=|z|\mbox{ and }\omega \in \m S^{N-1}
\]
and introducing $\partial_s V$ and $\nabla_\omega V$ such that
\[
\partial_s V = \frac z{|z|}\cdot \nabla v\mbox{ and }\nabla v= \partial_s V \frac z{|z|}+\frac 1{|z|}\nabla_\omega V,
\]
we see that $z\cdot \nabla_\omega V=0$, hence
\[
|\nabla v(z)|^2= (\partial_s V(s,\omega))^2+\frac 1{s^2} (\nabla_\omega V(s,\omega))^2\mbox{ and }|z \cdot \nabla v(z)|^2 = s^2 (\partial_s V(s,\omega))^2. 
\]
Therefore,
\begin{align}
&\iint [|\nabla v(z)|^2-(z \cdot \nabla v(z))^2](1-|z|^2)^\alpha dy\label{gradVrad}\\
=&\int_0^1 \int_{\m S^{N-1}}\left((\partial_s V(s,\omega))^2(1-s^2)+\frac 1{s^2} |\nabla_\omega V(s,\omega)|^2\right) (1-s^2)^\alpha s^{N-1} ds d\omega.\nonumber
\end{align}
Transforming similarly the wighted $L^{p+1}$ and $L^2$ norms, we transform  the aimed identity in Lemma \ref{lemhs} to the following:
\begin{align}
&\int_0^1 \int_{\m S^{N-1}}|V(s,\omega)|^{p+1}(1-s^2)^\alpha s^{N-1} ds d\omega\label{estV}\\
&\le C\int_0^1 \int_{\m S^{N-1}}\left((\partial_s V(s,\omega))^2(1-s^2)+
\frac{|\nabla_\omega V(s,\omega)|^2}{s^2}+|V(s,\omega)|^2\right)(1-s^2)^\alpha s^{N-1} ds d\omega.\nonumber
\end{align}
Making the following change of variables:
\[
h(y)=H(r,\omega) = V(s,\omega)=v(z)\mbox{ with }
r=\psi(s)=1-\sqrt{1-s},
\]
and introducing $a=2\alpha+1=\frac 4{p-1}+2-N$ (note that $a\neq 1$ and $a\ge 0$ from the condition \eqref{condp} on $p$), we apply Lemma \ref{lem3} to $h(y)$ in the $(r,\omega)$ coordinates: 
\begin{align}
&\int_0^1 \int_{\m S^{N-1}}|H(r,\omega)|^{p+1}(1-r^2)^a r^{N-1} dr d\omega\label{estH}\\
&\le C\int_0^1 \int_{\m S^{N-1}}\left((\partial_r H(r,\omega))^2+
\frac 1{r^2}|\nabla_\omega V(r,\omega)|^2+|V(r,\omega)|^2\right)(1-r^2)^a r^{N-1} dr d\omega.\nonumber
\end{align}
Transforming all the integrals in this identity as follows (note that we use the weight $1-r$ instead of $1-r^2$ for convenience), we write
\begin{align*}
&\int_0^1 \int_{\m S^{N-1}}|H(r,\omega)|^{p+1}(1-r)^ar^{N-1}dr d\omega
=\frac 12\int_0^1 \int_{\m S^{N-1}}|V(s,\omega)|^{p+1}(1-s)^\alpha \psi(s)^{N-1}ds d\omega,\\
&\int_0^1 \int_{\m S^{N-1}}|H(r,\omega)|^2(1-r)^ar^{N-1}dr d\omega
=\frac 12\int_0^1 \int_{\m S^{N-1}}|V(s,\omega)|^2(1-s)^\alpha \psi(s)^{N-1}ds d\omega,\\
& \int_0^1 \int_{\m S^{N-1}}|\partial_r H(r,\omega)|^2(1-r)^ar^{N-1} dr d\omega
=2\int_0^1 \int_{\m S^{N-1}}|\partial_s V(s,\omega)|^2(1-s)^{\alpha+1}\psi(s)^{N-1}ds d\omega,\\
& \int_0^1 \int_{\m S^{N-1}}|\nabla_\omega H(r,\omega)|^2(1-r)^ar^{N-1} dr d\omega
=\frac 12\int_0^1 \int_{\m S^{N-1}}|\nabla_\omega V(s,\omega)|^2(1-s)^{\alpha}\psi(s)^{N-1}ds d\omega.
\end{align*}
Since
\[
\forall s\in (0,1),\;\;\frac sC\le \psi(s) \le Cs,
\]
we see that \eqref{estV} follows from \eqref{estH}. This concludes the proof of Lemma \ref{lemhs}. 
\end{proof}

\section{Some properties of $\kappa^*(d,\nu)$ and the Lyapunov functional $E$}\label{secal}
In this section, we give some properties of $\kappa^*(d,\nu)$ defined in \eqref{defk*} and the Lyapunov function $E$ defined in \eqref{defenergy}. In particular, we prove Claim \ref{cltech}.
Let us first recall the following estimate from \cite{MZjfa07} (see Claim 4.3 page 84 in that paper):
\begin{cl}[Integral computation table]\label{cltech0}Consider for some $\gamma>-1$ and $\beta\in \m R$ the following integral
\[
I(d)= \int_{-1}^1 \frac{(1-\xi^2)^\gamma}{(1+d \xi)^\beta}dy\mbox{ where }d\in(-1,1).
\]
Then, there exists $K(\gamma, \beta,N)>0$ such that the following limits hold as $|d|\to 1$:\\
(i) if $\gamma+1-\beta>0$, then $I(d)\to K$,\\
(ii) if $\gamma+1-\beta=0$, then $I(d)|\log(1-|d|)|\to K$,\\
(iii) if $\gamma+1-\beta<0$, then $I(d)(1-|d|)^{- (\gamma+1)+\beta}\to K$.
\end{cl}
Then, we give the proof of Lemma \ref{cltech}:
\begin{proof}[Proof of Claim \ref{cltech}] Using the definition \eqref{defk*} of $\kappa^*(d,\nu)$, we see that
\begin{eqnarray*}
\pnu \kappa^*(d,\nu,y) &=& \vc{-\frac{2\kappa_0}{p-1}\frac{(1-|d|^2)^{\frac 1{p-1}}}{(1+\nu+d\cdot y)^{\frac {p+1}{p-1}}}}{-\frac{2\kappa_0}{p-1}\frac{(1-d^2)^{\frac 1{p-1}}}{(1+\nu+d\cdot y)^{\frac {p+1}{p-1}}}+\frac{2(p+1)}{(p-1)^2}\kappa_0 \nu \frac{(1-|d|^2)^{\frac 1{p-1}}}{(1+\nu+d\cdot y)^{\frac {2p}{p-1}}}},\\
\nabla_d \kappa^*(d,\nu,y)&=&\vc{-\frac{2\kappa_0}{p-1}\frac{(1-|d|^2)^{\frac {2-p}{p-1}}}{(1+\nu+d\cdot y)^{\frac 2{p-1}}}d- \frac{2\kappa_0}{p-1}\frac{(1-|d|^2)^{\frac 1{p-1}}}{(1+\nu+d\cdot y)^{\frac {p+1}{p-1}}}y}
{\frac{4\kappa_0\nu}{(p-1)^2}\frac{(1-|d|^2)^{\frac {2-p}{p-1}}}{(1+\nu+d\cdot y)^{\frac {p+1}{p-1}}}d+ \frac{2(p+1)\kappa_0\nu}{(p-1)^2}\frac{(1-|d|^2)^{\frac 1{p-1}}}{(1+\nu+d\cdot y)^{\frac {2p}{p-1}}}y}.
\end{eqnarray*}
Since we see from these expressions and Lemma \ref{eigenld*} that $\partial_\nu \kappa^*(d,\nu,y)$, $\partial_{e_1} \kappa^*(d,\nu,y)=\frac d{|d|} \cdot \nabla_d \kappa^*(d,\nu,y)$, $\frac{W_{0,1}(d,y)}{1-|y|^2}$ and $W_{1,1}(d,y)$ depend only on $|d|$ and the one-dimensional variable $y\cdot e_1 = y \cdot \frac d{|d|}$, it follows by definition \eqref{defpdi} of the projections $\pi_0^{d^*}$ and $\pi_1^{d^*}$ that the computation of the projections in estimate \eqref{old} in Claim \ref{cltech} reduces to the one-dimensional case, already treated in Claim 2.2 in \cite{MZisol10} (note that \eqref{equiv} holds here with $A$ replaced by $B$).
Furthermore, taking $i=2,\dots,N$, by definitions \eqref{defpdi} and \eqref{defphi} of $\pi_i^{d^*}$ and the inner product $\phi$, Lemma \ref{eigenld*} and \eqref{y.nabla}, we see from separation of variables in the integrals that the projections in the first line of \eqref{new} are zero.\\
For that reason, we focus in the remaining part on the projections of $\partial_{e_i} \kappa^*(d,\nu,y)$ for $i=2,\dots,N$.
Since $\partial_{e_i}\kappa^*=e_i \cdot \nabla_d \kappa^*$, using the eigenvalues of $L_d$ introduced in Lemma \ref{l10}, we see after straightforward calculations that
\begin{equation}\label{dejk}
\frac 1{L_2} \partial_{e_i}\kappa^*(d,\nu,y) = -F_i(d^*,y) + \frac{(p+1)\nu}{(p-1)(1+\nu)}\vc{0}
{\frac{(1-|d^*|^2)^{\frac{p+1}{2(p-1)}}}{(1+d^*\cdot y)^{\frac{2p}{p-1}}}y\cdot e_i}
\end{equation}
where 
\begin{equation}\label{deflambda}
L_2 = \frac{2\kappa_0(1-|d|^2)^{\frac 1{p-1}}(1+\nu)^{-\frac {p+1}{p-1}}}{(p-1)(1-|d^*|^2)^{\frac {p+1}{2(p-1)}}} =\frac{2\kappa_0\lambda(1+\nu)^{-1}}{(p-1)\sqrt{1-{|d^*|}^2}},\;\;
\lambda(d,\nu)=\frac{(1-|d|^2)^{\frac 1{p-1}}}{[(1+\nu)^2-|d|^2]^{\frac 1{p-1}}}.
\end{equation}
Using again the definitions \eqref{defpdi} and \eqref{defphi} of $\pi_j^{d^*}$ and the inner product $\phi$ together with the orthogonality relation in Lemma \ref{eigenld*} and separation of variables in the integrals, we also get the left estimate in the second line of estimate \eqref{new}. We are left only with the last estimate to prove.\\
Using 
\eqref{condnuB}
and \eqref{equiv} (which holds here with $A$ replaced by $B$), we see that for some $C^*(B)>0$, we have
\begin{equation}\label{estl2}
\frac 1{C^*} \le \lambda \le C^*
\mbox{ and }-1+\frac 1B\le \nu \le B,
\mbox{ hence }
\frac 1{C^*\sqrt{1-|d|}}\le L_2 \le \frac {C^*}{\sqrt{1-|d|}}.
\end{equation}
Using \eqref{dejk}, the definition \eqref{defpdi} of $\pi_j^{d^*}$, Lemma \ref{eigenld*}, together with the transformation \eqref{defW} and \eqref{yortho}, we write
\begin{align}
\pi_i^{d^*}(\frac 1{L_2} \partial_{e_i}\kappa^*(d,\nu,y))&=-1+\frac{c_0(p+1)\nu}{(p-1)(1+\nu)}(1-|d^*|^2)^{\frac{p+1}{p-1}}\iint \frac{(y\cdot e_i)^2}{(1+d^*\cdot y)^{\frac{3p+1}{p-1}}}\rho(y) dy\nonumber\\
&=c_0\left( -\frac 1{c_0}+\frac{(p+1)\nu}{(p-1)(1+\nu)}\frac {J}{N-1}\right)\label{start}
\end{align}
where from \eqref{defWl2} and symmetry,
\begin{equation}\label{defIJ}
\frac 1{c_0}=2\alpha\int z_1^2\frac{\rho(z)}{1-|z|^2}dz,\;\;
J=(N-1)\int_{|z|<1}\frac{z_i^2}{1-|d^*|z_1}\rho(z) dz=
\int_{|z|<1}\frac{|\tilde z|^2}{1-|d^*|z_1}\rho(z) dz
\end{equation}
with $\tilde z=(z_2,\dots,z_N)\in \m R^{N-1}$. Since $0<J\le C$ and $\nu \ge -1+\frac 1B$ from \eqref{condnuB}, we readily see that
\[
\pi_i^{d^*}\left(\frac 1{L_2} \partial_{e_i}\kappa^*(d,\nu,y) \right)\ge -C -CB,
\]
and the lower bound in the last estimate of \eqref{new} follows from \eqref{estl2}. It only remains to get the upper bound.\\ 
Making the change of variables $\tilde z=\zeta\sqrt{1-z_1^2}$ with $\zeta\in \m R^{N-1}$ in \eqref{defIJ}, we see from the definition \eqref{defro} of $\rho(z)$ that
\[
\frac 1{c_0}=2\alpha J_{\alpha-1}\int_{-1}^1 y_1^2 (1-y_1^2)^{\beta -1} dy_1\mbox{ and }J=\bar J\int_{|\zeta|<1}|\zeta|^2(1-|\zeta|^2)^\alpha d\zeta
\]
with
\begin{equation}\label{defJxi}
\beta=\frac 2{p-1},\;\;\bar J=\int_{-1}^1 \frac{(1-y_1^2)^\beta}{1-|d^*|y_1} dy_1\mbox{ and }
J_\sigma = \int_{|\zeta|<1}(1-|\zeta|^2)^\sigma d\zeta.
\end{equation}
Since
\[
\bar J=\frac 12\left(\int_{-1}^1 \frac{(1-y_1^2)^\beta}{1-|d^*|y_1} dy_1
+\int_{-1}^1 \frac{(1-y_1^2)^\beta}{1+|d^*|y_1} dy_1\right)
=\int_{-1}^1 \frac{(1-y_1^2)^\beta}{1-|d^*|^2y_1^2} dy_1\le I_{\beta-1}
\]
where
\[
I_\sigma=\int_{-1}^1 (1-y_1^2)^\sigma dy_1,
\]
and 
\begin{equation}\label{ipp}
\int_{|\zeta|<1}|\zeta|^2(1-|\zeta|^2)^\alpha d\zeta=\frac{\alpha(N-1)(p-1)^2}{4(p+1)}J_{\alpha-1},\;\;
\int_{-1}^1 y_1^2 (1-y_1^2)^{\beta -1} dy_1=\frac{p-1}{p+3}I_{\beta-1}
\end{equation}
(see below for a proof of this fact), we write from \eqref{start}, the above estimates and the fact that $p$ is subconformal (see \eqref{condp}),
\begin{align*}
\pi_i^{d^*}\left(\frac 1{L_2} \partial_{e_i}\kappa^*(d,\nu,y) \right)
\le &c_0\left( -\frac 1{c_0} +\frac{(p+1)}{(p-1)}\frac{I_{\beta-1}}{(N-1)}\int_{|\zeta|<1}|\zeta|^2(1-|\zeta|^2)^\alpha d\zeta\right)\\
=&-1+\frac{p+3}8<\frac{(2-N)}{2(N-1)}\le 0.
\end{align*}
Using \eqref{estl2}, we get the upper bound in the last estimate of \eqref{new}. It remains to prove \eqref{ipp} in order to conclude. 

\medskip

{\it Proof of \eqref{ipp}}: We only prove the first estimate, since the second follows in the same way, and is even easier.
Integrating by parts in the unit ball of $\m R^{N-1}$ and using the definition \eqref{defJxi} of $J_\sigma$, we see that for $\sigma>0$,
\begin{align}
&\int_{|\zeta|<1}|\zeta|^2(1-|\zeta|^2)^{\sigma-1} d\zeta
=-\frac 1{2\sigma}\int_{|\zeta|<1}\zeta \cdot \nabla(1-|\zeta|^2)^{\sigma} d\zeta\nonumber\\
=&\frac {N-1}{2\sigma}\int_{|\zeta|<1}(1-|\zeta|^2)^{\sigma} d\zeta
=\frac {N-1}{2\sigma}I_\sigma,\label{green}
\end{align}
on the one hand. Since we also have on the other hand, $\int_{|\zeta|<1}|\zeta|^2(1-|\zeta|^2)^{\sigma-1} d\zeta=I_{\sigma-1}-I_\sigma$, it follows that
$I_{\sigma -1}=(1+\frac{N-1}{2\sigma})I_\sigma$. Using this identity with $\sigma=\alpha+1$, then, $\sigma=\alpha$, where $\alpha$ is given in \eqref{defro}, we get the result from \eqref{green} applied with $\sigma=\alpha+1$. This concludes the proof of \eqref{ipp} and Claim \ref{cltech} too.
\end{proof}
Now, we give more properties of $\kappa^*(d,\nu)$ in the following:
\begin{cl}[Some properties of $\kappa^*(d,\nu)$]\label{propk*}$ $\\
(i) For all $B\ge 2$, 
$|d|<1$ and $-1+\frac 1B \le \frac{\nu}{1-|d|}\le B$,
we have
\begin{align}
\kappa^*_1(d,\nu,y)=\lambda(d,\nu)\kappa(d^*,y) \le C(B)\kappa(d^*,y)&\le \frac{C(B)}{(1-|y|^2)^{\frac 1{p-1}}},\label{kp2}\\
\ds\left\|\kappa^*\left(d,\nu\right)\right\|_{\H}
\le C\lambda\left(1+1_{\{\nu<0\}}\frac{|\nu|}{\sqrt{1-|d|^2}}\lambda^{\frac{p-1}2}\right)&\le C(B),\label{normk*}\\
\mbox{for }i=2,\dots,N,\;\;|\partial_{e_i} \kappa^*_1(d,\nu,y)|\le C(B)\frac{|y\cdot e_i|}{1+d^*\cdot y}\kappa(d^*,y)& \le \frac {C(B)\kappa(d^*,y)}{\sqrt{1-|d|}},\label{peik*}\\
\|\partial_\nu \kappa^*(d,\nu)\|_{\q H}+\|\partial_{e_1} \kappa^*(d,\nu)\|_{\q H}&\le \frac {C(B)}{1-|d|},\label{normdnk}\\
\|\partial_{e_i} \kappa^*(d,\nu)\|_{\q H}&\le \frac {C(B)}{\sqrt{1-|d|}},\label{normd2k}
\end{align}
where $d^*=\frac d{1+\nu}$ and $\lambda(d,\nu)$ is introduced in \eqref{deflambda}.\\
(ii) For all $|d|<1$ and $\nu >-1+|d|$, we have
\[
E(\kappa^*(d,\nu))=\frac{E(\kappa_0)}{p-1}\lambda^2\left(p+1+2\lambda^{p-1}\left(\frac{\nu^2}{(1-|d|^2)}-1\right)\right)
\]
where $\lambda$ is defined in \eqref{deflambda}.\\
(iii)  We have $\sup_{\{|d|<1,\;\frac\nu{1-|d|}=\sigma\}}E(\kappa^*(d,\nu))\to -\infty$ as $\sigma\searrow -1$.\\  
(iv) We have $\sup_{\{|d|<1,\;\frac\nu{1-|d|}=\sigma\}}\|\kappa^*(d,\nu)\|_{\q H}\to 0$ as $\sigma\to \infty$.\\
(v) For all $B\ge 2$, there exists $\equatre(B)>0$ such that if $|d_1| <1$, $|d_2|<1$, 
\begin{equation}\label{prox}
-1+\frac 1B \le \frac{\nu_i}{1-|d_i|}\le B\mbox{ for }i=1,\;2,
\mbox{ and }|l_{d_1}(d_1-d_2)|+|d_1-d_2|^2\le \equatre(1-|d_1|),
\end{equation}
then, it follows that for some $C^*=C^*(B)$, 
\begin{align*}
\|\kappa^*(d_1,\nu_1) - \kappa^*(d_2, \nu_2)\|_{\q H}
\le& C^*\left|\frac{\nu_1}{1-|d_1|}-\frac{\nu_2}{1-|d_2|}\right|
+C^*\left|\arg \tanh |d_1|-\arg \tanh |d_2|\right|\\
&+C^* \frac{|l_{d_1}(d_2 -d_1)|}{1-|d_1|}+C^*\frac{|d_2-d_1|}{\sqrt{1-|d_1|}}.
\end{align*}
\end{cl}
\begin{nb} We recall from the paragraph right before \eqref{yortho} that when $d_1=0$, $l_{d_1}(d_2)$ and $\bot_{d_1}(d_2)$ are respectively the component of $d_2$ along $d_1$ and orthogonal to $d_1$, and when $d_1=0$,  $l_0(d_2)=0$ and $\bot_0(d_2)=d_2$. 
\end{nb}
\begin{proof} (i) Consider $B\ge 2$. Note that \eqref{estl2} holds and so does \eqref{equiv}, with $A$ replaced by $B$.\\
{\it - Proof of \eqref{kp2}}: This directly follows from the definitions \eqref{defkd} and \eqref{defk*} of $\kappa(d^*,y)$ and $\kappa^*_1(d,\nu)$ together with \eqref{kp1}.\\
{\it - Proof of \eqref{normk*}}: From rotation invariance of the expression of the norm \eqref{defnh} and identity \eqref{defnhz1}, we reduce to the case where $d=(|d|,0,\dots,0)$, hence to one-space dimension, 
and the result follows from Lemma A.2 page 2878 in \cite{MZisol10}.\\
 {\it - Proof of \eqref{peik*}}: This follows from the expressions \eqref{dejk} and \eqref{deffid} of $\partial_{e_i} \kappa^*_1(d,\nu)$ and $F_i(d^*)$.\\
{\it - Proof of \eqref{normdnk}}: Let us recall from the proof of Claim 3.2 given in page 2852 of \cite{MZisol10} that
\[
\|\partial_\nu \kappa^*(d,\nu)\|_{\q H} +\|\partial_{e_1} \kappa^*(d,\nu)\|_{\q H}
\le \frac C{1-|d|} \left(\|F_1(d^*)\|_{\q H}+\|F_0(d^*)\|_{\q H} + |\nu|\gamma(d^*)\right)
\]
where $\gamma(d^*) = \left\|\frac{(1-|d^*|^2)^{\frac p{p-1}}}{(1+d^*\cdot y)^{\frac{2p}{p-1}}}\right\|_{L^2_\rho}\le C(1-|d|)^{-1/2}$ from Claim \ref{cltech0}. Since $|\nu|\le A(1-|d|)$, using item (ii) in Lemma \ref{l10}, we get the result.\\
{\it - Proof of \eqref{normd2k}}: Consider $i=2,\dots,N$. From the expression \eqref{dejk} of $\partial_{e_i} \kappa^*$ and \eqref{estl2}, we see that
\[
\|\partial_{e_i} \kappa^*(d,\nu)\|_{\q H}
\le \frac {C(B)}{\sqrt{1-|d|}} \left(\|F_i(d^*)\|_{\q H}+ |\nu|\gamma_i(d^*)\right)
\]
where $\gamma_i(d^*)= \left\|\frac{(1-|d^*|^2)^{\frac {p+1}{2(p-1)}}}{(1+d^*\cdot y)^{\frac{2p}{p-1}}}y\cdot e_i\right\|_{L^2_\rho}\le C(1-|d|)^{-1/2}$ from Claim \ref{cltech0}. We then conclude as for \eqref{normdnk}. 

\medskip

\noindent (ii) Consider $|d|<1$ and $\nu >-1+|d|$. From rotation invariance of the expression of $E$ \eqref{defenergy}, we reduce to the case where $d=(|d|,0,\dots,0)$. In this case $\kappa^*(d,\nu,y)$ depends only on $y_1$, and we see that the $N$-dimensional expression of $E$ is equal to the one-dimensional expression, up to a multiplying constant depending only on $N$. Thus, we reduce to the one-dimensional case, which is given in Lemma A.2 page 2878 in \cite{MZisol10}.

\medskip

\noindent (iii) From straightforward computations, we see that when $\nu=(-1+\epsilon)(1-|d|)$ where $\epsilon \in (0,1)$, we have 
\[
\lambda^{p-1} =\frac{1+|d|}{\epsilon(2|d|+\epsilon(1-|d|))}\ge \frac 1\epsilon
\mbox{ and }\frac{\nu^2}{1-|d|^2} -1 = \frac{-2|d|-(2\epsilon -\epsilon^2)(1-|d|)}{1+|d|}\le -\frac{\lambda^{1-p}}\epsilon,
\]
hence, from item (ii), we see that
\[
E(\kappa^*(d,\nu) \le \frac{E(\kappa_0)}{p-1}\epsilon^{-\frac 2{p-1}}\left(p+1-\frac 2\epsilon\right)\to -\infty\mbox{ as }\epsilon\searrow 0.
\]

\medskip

\noindent (iv) Consider $\nu=\sigma(1-|d|)$. Since we have by definition \eqref{deflambda},
\[
\lambda^{p-1}=\frac{1-|d|^2}{(1+\nu)^2-|d|^2}=\frac{1-|d|}{1+\nu -|d|}\cdot \frac{1+|d|}{1+\nu +|d|}\le\frac 1{1+\sigma},
\]
and $\nu\ge 0$, the conclusion follows from \eqref{normk*}.

\medskip

\noindent (v) We proceed in three steps.

\medskip

{\bf Step 1: The case where $\nu_1=\nu_2=0$}. 

Using the definitions \eqref{defk*} and \eqref{defkd} of $\kappa^*(d,\nu)$ and $\kappa(d)$, together with items (i) and (iv) in Lemma \ref{0lemeffect}, then using the local continuity near $0$, we write
\[
\|\kappa^*(d_1,\nu_1)-\kappa^*(d_2,\nu_2)\|_{\q H} =
\|\kappa(d_1)-\kappa(d_2)\|_{\q H_0}
\le C\|\kappa_0-\kappa(\theta_{-d_1}(d_2))\|_{\q H_0}\le C|\theta_{-d_1}(d_2)|
\]
where $\theta_{-d_1}(d_2)$ is defined in \eqref{deftd}.
Decomposing $\theta_{-d_2}(d_1)$ as in the paragraph right before \eqref{yortho}, we see that
\begin{align*}
l_{d_1}(\theta_{-d_1}(d_2))&=\frac{l_{d_1}(d_2)-|d_1|}{1-d_1\cdot d_2}=
\frac{l_{d_1}(d_2 -d_1)}{1-d_1\cdot d_2},\\
\bot_{d_1}(\theta_{-d_1}(d_2))&=\frac{\sqrt{1-|d_1|^2}}{1-d_1\cdot d_2}\bot_{d_1}(d_2)= \frac{\sqrt{1-|d_1|^2}}{1-d_1\cdot d_2}\bot_{d_1}(d_2-d_1).
\end{align*}
Since we see from \eqref{prox} that
\begin{equation}\label{dalida}
1-d_1\cdot d_2 =1-|d_1|^2-|d_1|l_{d_1}(d_2-d_1)\ge (1-|d_1|^2)(1-\equatre),
\end{equation}
we get the desired estimate.

\medskip

{\bf Step 2: The case where $d_2=0$}

In this case, we see that the function $y\mapsto\kappa^*(d_1,\nu_1,y)-\kappa^*(0,\nu_2,y)$ is either independent from $y$ (if $d_1=0$), or a function of the one-dimensional variable $y\cdot \frac {d_1}{|d_1|}$ (if $d_1\neq 0$). Using \eqref{defnhz1}, we see that we reduce to the one-dimensional case, already treated in item (ii) of Lemma A.2 page 2878 in \cite{MZisol10}.

\medskip

{\bf Step 3: The case where $d_2\neq 0$}

In this case, we introduce 
\[
d_3 = \frac{|d_1|}{|d_2|}d_2\mbox{ and }\nu_3 = \nu_1
\]
(in other words, $d_3$ is collinear to $d_2$ with norm $|d_1|$). 
By a triangular inequality, we see that
\begin{equation}\label{triangle}
\|\kappa^*(d_1,\nu_1) - \kappa^*(d_2, \nu_2)\|_{\q H} 
\le \|\kappa^*(d_1,\nu_1) - \kappa^*(d_3, \nu_1)\|_{\q H} 
+\|\kappa^*(d_3, \nu_1) - \kappa^*(d_2, \nu_2)\|_{\q H}.
\end{equation}
Concerning the second norm on the right-hand side of this inequality, noting that the function $y\mapsto\kappa^*(d_3,\nu_1,y)-\kappa^*(d_2,\nu_2,y)$ is 
a function of the one-dimensional variable $y\cdot \frac {d_2}{|d_2|}$,
we proceed as in Step 2, and recalling that $|d_3|=|d_1|$, we write 
\begin{equation}\label{oned}
\|\kappa^*(d_3,\nu_1) - \kappa^*(d_2, \nu_2)\|_{\q H} 
\le C^*\left|\frac {\nu_1}{1-|d_1|} - \frac {\nu_2}{1-|d_2|}\right|
+C^*|\arg\tanh |d_1| -\arg\tanh|d_2||
\end{equation}
(here and in the following, $C^*=C^*(B)>0$).\\
We now handle the first norm in the right-hand side of \eqref{triangle}. Introducing $d_i^*=\frac {d_i}{1+\nu_i}$, we see that for $\etrois$ small enough, we have
\begin{equation*}
|d_3|=|d_1|,\;\;\nu_3=\nu_1,\;\;|d_3^*|=|d_1^*|\mbox{ and }d_3^*-d_1^*=\frac{d_3-d_1}{1+\nu_1}\mbox{ with }\frac 1B\le 1+\nu_1\le B+1.
\end{equation*}
(this identity will be frequently used in the following, together with \eqref{equiv} and \eqref{estl2} which hold here with $A$ replaced by $B$). Using the definitions \eqref{defk*} and \eqref{defkd} of $\kappa^*$ and $\kappa$, we write
\begin{equation}\label{ch1}
\|\kappa^*(d_1,\nu_1) - \kappa^*(d_3, \nu_3)\|_{\q H}
\le C^*\|\kappa(d_1^*,y)-\kappa(d_3^*,y)\|_{\q H_0}+C^* J^{1/2}
\end{equation}
where
\[
J =\nu_1^2(1-|d_1|)^{\frac 2{p-1}}\iint \left|\frac 1{(1+d_3^*\cdot y)^{\frac{p+1}{p-1}}}-\frac 1{(1+d_1^*\cdot y)^{\frac{p+1}{p-1}}}\right|^2 \rho dy.
\]
Using the result of Step 1, we see that 
\begin{align}
\|\kappa(d_3^*,y)-\kappa(\hat d_1^*,y)\|_{\q H_0}
&\le C \frac{|l_{d_1^*}(d_3^* -d_1^*)|}{1-|d_1^*|}+C\frac{|\bot_{d_1^*}(d_3^*-d_1^*)|}{\sqrt{1-|d_1^*|}},\nonumber\\
&\le C^* \frac{|l_{d_1}(d_3 -d_1)|}{1-|d_1|}+C^*\frac{|\bot_{d_1}(d_3-d_1)|}{\sqrt{1-|d_1|}}.\label{ch1.5}
\end{align}
Since $|X^{-\frac{p+1}{p-1}} - Y^{-\frac{p+1}{p-1}}|\le \frac{p+1}{p-1}|(X^{-\frac{2p}{p-1}}+Y^{-\frac{2p}{p-1}}|$ 
for any $X>0$ and $Y>0$,
using the calculation rule of Claim \ref{cltech0} and \eqref{prox}, we write 
\begin{equation}\label{ch2}
J\le C \nu_1^2|d_1-d_3|^2\int_{|z|<1} \frac{(1-|d_1^*|)^{\frac 2{p-1}}\rho(z)}{(1+|d_1^*| z_1)^{\frac{4p}{p-1}}}dz\le C \nu_1^2\frac{|d_1-d_3|^2}{(1-|d_1^*|)^3}
\le C^* \frac{|d_1-d_3|^2}{1-|d_1|}.
\end{equation}
Since $|d_3-d_2|=||d_1|-|d_2||$, we see that
\begin{equation}\label{ch4}
|l_{d_1}(d_3 -d_1)|\le |l_{d_1}(d_2 -d_1)|+||d_1|-|d_2||\mbox{ and }
|\bot_{d_1}(d_3-d_1)|\le |d_1-d_2|,
\end{equation}
Using \eqref{prox} and arguing as for \eqref{dalida}, we see that
\begin{equation}\label{lip}
\frac{||d_2|-|d_1||}{1-|d_1|}\le C|\arg\tanh |d_1| -\arg\tanh|d_2||.
\end{equation}
Using \eqref{triangle}-\eqref{lip}, we get the conclusion of item (v) in Claim \ref{propk*}.
\end{proof}

\bigskip

In the following, we recall from \cite{MZajm11}, \cite{MZjfa07} and \cite{MZcmp08} some properties of the Lyapunov functional $E$ defined in \eqref{defenergy}.
\begin{lem}[Properties of the functional $E$ \eqref{defenergy}]\label{lemE}$ $\\
(i) For all $r$ and $v$ in $\q H$, we have
\[
|E(r) - E(v)|\le C(1+\|r\|_{\q H}^p+\|v\|_{\q H}^p)\|r-v\|_{\q H}.
\]
(ii) Consider $w$ a solution of equation \eqref{eqw} such that $E(w(s_0))<0$. Then, $w(y,s)$ cannot be defined for all $(y,s) \in B(0,1) \times [s_0, \infty)$.
\end{lem}
\begin{proof}$ $\\
(i) As for Claim B.1 page 622 in \cite{MZajm11}, this is a direct consequence of the definitions \eqref{defenergy} and \eqref{defnh} of $E(v)$ and $\|v\|_{\q H}$, together with the Hardy-Sobolev identity of Lemma \ref{lemhs}.\\
(ii) See Theorem 2 page 1147 in Antonini and Merle \cite{AMimrn01}.
\end{proof}
\def\cprime{$'$} \def\cprime{$'$}
\providecommand{\bysame}{\leavevmode\hbox to3em{\hrulefill}\thinspace}
\providecommand{\MR}{\relax\ifhmode\unskip\space\fi MR }
\providecommand{\MRhref}[2]{%
  \href{http://www.ams.org/mathscinet-getitem?mr=#1}{#2}
}
\providecommand{\href}[2]{#2}

\noindent{\bf Address}:\\
Universit\'e de Cergy Pontoise, D\'epartement de math\'ematiques, 
2 avenue Adolphe Chauvin, BP 222, F-95302 Cergy Pontoise cedex, France.\\
\vspace{-7mm}
\begin{verbatim}
e-mail: merle@math.u-cergy.fr
\end{verbatim}
Universit\'e Paris 13, Sorbonne Paris Cit\'e, LAGA, CNRS (UMR 7539), 
99 avenue J.B. Cl\'ement, F-93430 Villetaneuse, France.\\
\vspace{-7mm}
\begin{verbatim}
e-mail: Hatem.Zaag@univ-paris13.fr
\end{verbatim}
\end{document}